\documentclass[leqno]{article}

\usepackage[frenchb,english]{babel}
\usepackage[utf8]{inputenc}
\usepackage{amsmath}
\usepackage{amssymb}
\usepackage{amsfonts}
\usepackage{enumerate}
\usepackage{vmargin}
\usepackage[all]{xy}
\usepackage{mathrsfs} 
\usepackage{mathtools}
\usepackage{lmodern}
\usepackage{slashed}
\usepackage[colorlinks=true,linkcolor=blue,pagebackref=true]{hyperref}%
\setmarginsrb{3cm}{3cm}{3.5cm}{3cm}{0cm}{0cm}{1.5cm}{3cm}

\newcommand{\C}{\mathrm{C}}

\newcommand{\E}{\ensuremath{\mathbb{E}}}
\newcommand{\F}{\mathcal{F}}
\newcommand{\N}{\ensuremath{\mathbb{N}}}
\newcommand{\B}{\mathrm{B}} 

\let\L\relax 
\newcommand{\L}{\mathrm{L}}
\newcommand{\J}{\mathcal{J}} 

\newcommand{\SG}{\mathrm{SG}} 

\newcommand{\scr}{\mathscr}

\newcommand{\M}{\mathrm{M}}
\newcommand{\CB}{\mathrm{CB}}
\newcommand{\RCD}{\mathrm{RCD}}
\newcommand{\dist}{\mathrm{d}}
\newcommand{\meas}{\mu}

\let\div\relax
\newcommand{\div}{\mathrm{div}}
\let\cal\relax
\newcommand{\cal}{\mathcal}
\newcommand{\Z}{\ensuremath{\mathbb{Z}}}
\newcommand{\R}{\ensuremath{\mathbb{R}}}
\newcommand{\T}{\ensuremath{\mathbb{T}}}

\newcommand{\W}{\mathrm{W}}

\newcommand{\Id}{\mathrm{Id}}
\newcommand{\VN}{\mathrm{VN}}

\newcommand{\la}{\langle}
\newcommand{\ra}{\rangle}

\newcommand{\CD}{\mathrm{CD}}

\renewcommand{\leq}{\ensuremath{\leqslant}}
\renewcommand{\geq}{\ensuremath{\geqslant}}
\newcommand{\qed}{\hfill \vrule height6pt  width6pt depth0pt}
\newcommand{\bnorm}[1]{ \big\| #1  \big\|}

\newcommand{\norm}[1]{\left\Vert#1\right\Vert}

\newcommand{\xra}{\xrightarrow}
\newcommand{\co}{\colon}

\newcommand{\ot}{\otimes}
\newcommand{\ovl}{\overline}

\newcommand{\otvn}{\overline{\otimes}} 

\newcommand{\cb}{\mathrm{cb}}

\newcommand{\Ch}{\mathrm{Ch}}

\newcommand{\op}{\mathrm{op}} 
\let\i\relax 
\newcommand{\i}{\mathrm{i}}

\newcommand{\ov}{\overset}

\newcommand{\epsi}{\varepsilon}
\renewcommand{\d}{\mathop{}\mathopen{}\mathrm{d}} 
\newcommand{\e}{\mathrm{e}} 
\renewcommand{\d}{\mathop{}\mathopen{}\mathrm{d}}

\DeclareMathOperator{\Span}{span} 
\DeclareMathOperator{\Lip}{\mathrm{Lip}} 
\DeclareMathOperator{\lip}{\mathrm{lip}} 
\DeclareMathOperator{\supp}{supp} 
\DeclareMathOperator{\Fix}{Fix} 
\DeclareMathOperator{\tr}{Tr} 
\let\ker\relax 
\DeclareMathOperator{\ker}{Ker} 
\DeclareMathOperator{\Ran}{Ran} 
\DeclareMathOperator{\dom}{dom} 
\let\Re\relax 
\DeclareMathOperator{\Re}{Re} 
\DeclareMathOperator{\card}{card} 

\DeclareMathOperator{\diam}{diam}
\newtheorem{thm}{Theorem}[section]
\newtheorem{defi}[thm]{Definition}
\newtheorem{prop}[thm]{Proposition}
\newtheorem{conj}[thm]{Conjecture}

\newtheorem{cor}[thm]{Corollary}
\newtheorem{lemma}[thm]{Lemma}

\newtheorem{remark}[thm]{Remark}
\newtheorem{example}[thm]{Example}
\newtheorem{prob}[thm]{Problem}

\newenvironment{proof}[1][]{\noindent {\it Proof #1} : }{\hbox{~}\qed
\smallskip
}

\usepackage{tocloft}
\numberwithin{equation}{section}
\usepackage[nottoc,notlot,notlof]{tocbibind}

\let\OLDthebibliography\thebibliography
\renewcommand\thebibliography[1]{
  \OLDthebibliography{#1}
  \setlength{\parskip}{0pt}
  \setlength{\itemsep}{0pt plus 0.3ex}
}

\newcommand\reallywidehat[1]{\arraycolsep=0pt\relax%
\begin{array}{c}
\stretchto{
  \scaleto{
    \scalerel*[\widthof{\ensuremath{#1}}]{\kern-.5pt\bigwedge\kern-.5pt}
    {\rule[-\textheight/2]{1ex}{\textheight}} 
  }{\textheight} %
}{0.5ex}\\           
#1\\                 
\rule{-1ex}{0ex}
\end{array}
}

\begin{document}
\selectlanguage{english}
\title{\bfseries{Spectral triples, Coulhon-Varopoulos dimension and heat kernel estimates}}
\date{}
\author{\bfseries{C\'edric Arhancet}}
\maketitle


\begin{abstract}
We investigate the relations between the (completely bounded) local Coulhon-Varopoulos dimension and the spectral dimension of spectral triples associated to sub-Markovian semigroups (or Dirichlet forms) acting on classical (or noncommutative) $\L^p$-spaces associated to finite measure spaces. More precisely, we prove that the completely bounded local Coulhon-Varopoulos dimension $d$ exceeds the spectral dimension, i.e.~that the associated Hodge-Dirac operator is $d^+$-summable. We explore different settings to compare these two values: compact Riemannian manifolds, compact Lie groups, sublaplacians, metric measure spaces, noncommutative tori and quantum groups. Specifically, we prove that, while very often equal in smooth compact settings, these dimensions can diverge. Finally, we show that the existence of a symmetric sub-Markovian semigroup on a von Neumann algebra with finite completely bounded local Coulhon-Varopoulos dimension implies that the von Neumann algebra is necessarily injective.



%
%
%
%
%
%
%
%
%
%
%
%
%
%
%
%
%
\end{abstract}




\makeatletter
 \renewcommand{\@makefntext}[1]{#1}
 \makeatother

{
  \hypersetup{linkcolor=blue}
 \tableofcontents
}


\section{Introduction}
\label{sec:Introduction}
Symmetric sub-Markovian semigroups of operators on $\L^p$-spaces of a measure space $\Omega$ is a classical topic of analysis. Cipriani and Sauvageot proved in \cite{CiS03} that the $\L^2$-generator $-A_2$ of such semigroup $(T_t)_{t \geq 0}$ acting on the Hilbert space $\L^2(\Omega)$ can be written $A_2=\partial^*\partial$ where $\partial$ is a (unbounded) closed derivation defined on a dense subspace of $\L^2(\Omega)$ with values in a Hilbert $\L^\infty(\Omega)$-bimodule $\cal{H}$. Here $T_t=\e^{-tA_2}$ for any $t \geq 0$. The map $\partial$ can be seen as an analogue of the gradient operator $\nabla$ of a smooth Riemannian manifold $M$, which is a closed operator defined on a subspace of $\L^2(M)$ into the space $\L^2(M,\mathrm{T} M)$ satisfying the relation $-\Delta=\nabla^*\nabla$ where $\Delta$ is the Laplace-Beltrami operator and where $\nabla^*=-\div$.

This fundamental result allows anyone to introduce a triple $(\L^\infty(\Omega),\L^2(\Omega) \oplus_2 \cal{H},D)$ associated to the semigroup in the spirit of noncommutative geometry. Here $D$ is the unbounded selfadjoint operator acting on a dense subspace of the Hilbert space $\L^2(\Omega) \oplus_2 \cal{H}$ defined by 
\begin{equation}
\label{Hodge-Dirac-I}
D
\ov{\mathrm{def}}{=}
\begin{bmatrix} 
0 & \partial^* \\ 
\partial & 0 
\end{bmatrix}.
\end{equation}
In large cases \cite{CGIS14}, \cite{Cip16}, \cite{HKT15}, this triple determines a possibly kernel-degenerate compact spectral triple (say a measurable version) in the sense of noncommutative geometry. Consequently, we can associate to the semigroup a noncommutative geometry. In this framework, a compact spectral triple $(\cal{A},H,D)$ consists of an algebra $\cal{A}$, encoding the space, with its elements acting as bounded operators on a complex Hilbert space $H$, along with a densely defined selfadjoint operator $D$ that also acts on $H$, satisfying some axioms. In particular, the operator $D^{-1}$ must be compact on ($\ker D)^\perp$ and it is often called the <<unit length>> or <<line element>> and denoted by $\d s$. In this context, we can define the spectral dimension of a compact spectral triple $(\cal{A},H,D)$ (or dimension of $D$) by
\begin{equation}
\dim(\cal{A},H,D)  
\ov{\mathrm{def}}{=} \inf\big\{ p > 0 : |D|^{-1} \in S^{p,\infty}(H) \big\}
\end{equation}
where $S^{p,\infty}(H)$ is the weak Schatten space defined in \eqref{weak-Schatten}.

In this paper, we explore the connections between this dimension and what we refer to as the local Coulhon-Varopoulos dimension of the semigroup $(T_t)_{t \geq 0}$. The latter is defined as the infimum of positive real numbers $d$ for which the following inequality holds:\begin{equation}
\label{Varo-Coulhon-dim}
\norm{T_t}_{\L^1(\Omega) \to \L^\infty(\Omega)} 
\lesssim \frac{1}{t^{\frac{d}{2}}}, \quad 0 < t \leq 1.
\end{equation}
In this introduction and later on in the paper, $\lesssim$ stands for an inequality up to a
constant. Note that it is a local property of the same inequality defined for $0 < t <\infty$ and widely studied in the literature, see \cite{CKS87}, \cite{CoM93} and \cite{VSCC92}. The local Coulhon-Varopoulos dimension of the heat semigroup on the $d$-dimensional torus $\T^d$ is equal to $d$. It is worth mentioning that the semigroup is often called ultracontractive if each $T_t$ induces a bounded map from the space $\L^1(\Omega)$ into $\L^\infty(\Omega)$ for any $t > 0$. This terminology originates from the paper \cite{DaS84}.

This inequality admits several different equivalent important forms and variations. If $-A_2$ is the generator of the semigroup on $\L^2(\Omega)$, it is well-known that if $d>2$ the inequality \eqref{Varo-Coulhon-dim} is equivalent to the Sobolev inequality 
\begin{equation}
\label{Sobolev-inequality}
\norm{f}_{\frac{2d}{d-2}}^2 
\lesssim \la A_2 f,f \ra + \norm{f}_{2}^2
\end{equation}
for any suitable function $f$ on $\Omega$. The proof of this equivalence is a classical result of Varopoulos \cite[Theorem p.~259]{Var85}. See also \cite[Proposition p.~486]{Cou90}, \cite[Theorem 6.3.1 p.~286]{BGL14}, \cite[Corollary 2.4.3 p.~77]{Dav89} and finally \cite[Theorem II.4.2 p.~18]{VSCC92} for a more general statement. 
We also refer to \cite[Corollary 2.4.7 p.~79]{Dav89} and \cite[p.~281 and p.~286]{BGL14} for the equivalence with the Nash-type inequality 
$$
\norm{f}_2^{1+\frac{2}{d}} 
\lesssim \big[\la A_2 f,f \ra +\norm{f}_2^2\big]^{\frac{1}{2}} \norm{f}_1^{\frac{2}{d}}.
$$ 
This kind of inequality has its roots in the famous paper \cite{Nas58} of Nash on the regularity of solutions of elliptic and parabolic equations. Finally, We also direct the reader to \cite[Corollaire 2 p.~490]{Cou90} and \cite[Section 7.3.2]{Are04} for some other characterizations.



Indeed, we will examine symmetric sub-Markovian semigroups acting on a finite von Neumann algebra $\cal{M}$. In this generalization, we need to define an extension of the dimension \eqref{Varo-Coulhon-dim} for the noncommutative context what we term completely bounded local Coulhon-Varopoulos dimension. This dimension is defined as the infimum of positive real numbers $d$ such that 
\begin{equation}
\label{Varo-Coulhon-dim-cb}
\norm{T_t}_{\cb,\L^1(\cal{M}) \to \L^\infty(\cal{M})} 
\lesssim \frac{1}{t^{\frac{d}{2}}}, \quad 0 < t \leq 1
\end{equation} 
where we replaced the operator norm by the \textit{completely bounded} norm $\norm{\cdot}_{\cb,\L^1(\cal{M}) \to \L^\infty(\cal{M})}$. Here, we use noncommutative $\L^p$-spaces, see \cite{PiX03} for a survey. Note that \eqref{Varo-Coulhon-dim} and \eqref {Varo-Coulhon-dim-cb} coincide for semigroups acting on classical $\L^p$-spaces associated to finite measure spaces, see \eqref{min-et-cb}. Observe that the inequality \eqref{Varo-Coulhon-dim-cb} is related to the $\lambda$-complete log-Sobolev inequality (in short $\lambda$-CLSI). Indeed, \eqref{Varo-Coulhon-dim-cb} combined with some mild conditions allows everyone to obtain this property in some contexts, see \cite[Theorem 3.10 p.~3432]{GJL20}. Finally, notice that in general there is a difference between the dimension \eqref{Varo-Coulhon-dim-cb} and the inequality 
\begin{equation}
\label{Varo-Coulhon-dim-AVN}
\norm{T_t}_{\L^1(\cal{M}) \to \L^\infty(\cal{M})} \lesssim \frac{1}{t^{\frac{d}{2}}}, \quad 0 < t \leq 1
\end{equation}
for a semigroup acting on a \textit{nonabelian} von Neumann algebra $\cal{M}$, see Section \ref{sec-dimension}. In the noncommutative setting, the inequality \eqref{Varo-Coulhon-dim-AVN} is still equivalent to a Sobolev inequality as in \eqref{Sobolev-inequality} by \cite[Theorem 5.1]{Xio17}. We also direct the reader to \cite[Theorem 1.1.1 p.~619]{JuM10} for a variant where $0 < t <\infty$. Finally, we refer to \cite{GJP17} and \cite{Zha18} for related results and notions. 


We will prove the following result. 

\begin{thm}
\label{th-spectral-coulhon-intro}
Let $\cal{M}$ be a finite injective von Neumann algebra equipped with a normal finite faithful trace. Consider a symmetric sub-Markovian semigroup $(T_t)_{t \geq 0}$ of operators acting on $\cal{M}$. Suppose the estimate $
\norm{T_t}_{\cb,\L^1(\cal{M}) \to \L^\infty(\cal{M})} 
\lesssim \frac{1}{t^{\frac{d}{2}}}$ for some $d > 0$ and any $0 < t \leq 1$. Then the spectral dimension of $D$ is less than $d$. 
\end{thm}

The previous result implies that the local Coulhon-Varopoulos dimension is bigger than the spectral dimension. 
The proof of Theorem \ref{th-spectral-coulhon-intro} relies on some tricks of our previous paper \cite{Arh24} and some news ideas. 

Indeed, in different smooth settings, we will observe that in general  the spectral dimension and the local Coulhon-Varopoulos dimension are identical. However and more intriguingly, we present a case within a more irregular framework where these dimensions are different, see Theorem \ref{th-RCD}. This is our second main result. Delving deeper into this new phenomenon would certainly be instructive.





Finally, our last main result is the following theorem which describes a constraint imposed by the existence of a symmetric sub-Markovian with a finite local Coulhon-Varopoulos dimension.

\begin{thm}
\label{thm-inj}
Let $\cal{M}$ be a finite von Neumann algebra equipped with a normal finite faithful trace. Suppose that the estimate $
\norm{T_t}_{\cb,\L^1(\cal{M}) \to \L^\infty(\cal{M})} \lesssim \frac{1}{t^{\frac{d}{2}}}$  for some $d > 0$ and any $0 < t \leq 1$ holds for some symmetric sub-Markovian semigroup $(T_t)_{t \geq 0}$ of operators acting on $\cal{M}$. Then the von Neumann algebra $\cal{M}$ is injective. 
\end{thm}
In the opposite direction, we present numerous examples of semigroups in the noncommutative setting with finite completely bounded local Coulhon-Varopoulos dimension, see Theorem \ref{thm-NC} and Theorem \ref{th-ame}.



\paragraph{Structure of the paper}
The paper is organized as follows. Section \ref{sec-preliminaries} gives background on operator theory, noncommutative $\L^p$-spaces and noncommutative geometry. In Section \ref{sec-triples}, we lay the groundwork for our first main theorem by introducing the necessary derivations to define the Dirac operator under consideration. By the way, in the $\Gamma$-regular case, we connect the norm of commutators to some amalgamated $\L^p$-spaces in Proposition \ref{Rem-amal}. The aim of Section \ref{sec-cdimension} is to prove our first main result Corollary \ref{th-spectral-coulhon} which to compare the two dimensions. In Section \ref{Sec-Hardy}, we examine our result from the point of view of the (completely bounded) Hardy-Littlewood-Sobolev theory and we give a second proof of Corollary \ref{th-spectral-coulhon}. In Section \ref{sec-examples}, we delve into specific semigroup examples, contrasting the two dimensions. In particular, we will give in Theorem \ref{th-RCD} an example of semigroup on a $\RCD^*(N-1,N)$-space where the spectral dimension and the local Coulhon-Varopoulos dimension are different. In Section \ref{Sec-noncommutative-examples}, we examine examples of semigroups acting on nonabelian von Neumann algebras. In particular, we present new examples of semigroups with finite local completely bounded Coulhon-Varopoulos dimension. We will also explain that the local Coulhon-Varopoulos dimension and its completely bounded analogue are different in the noncommutative context. In Section \ref{sec-dimension}, we show that an estimate \eqref{Varo-Coulhon-dim-cb} implies that the von Neumann where the operators of the semigroup act is injective. Finally, in Section \ref{sec-future}, we examine future investigations and open questions. 

%
%
%

\section{Preliminaries}
\label{sec-preliminaries}

\paragraph{Interpolation theory}
We begin with a review of the foundational concepts in complex interpolation theory. For a more detailed exploration, readers are encouraged to consult the books \cite{BeL76}, \cite{KPS82}, and \cite{Lun18}. Let $X_0,X_1$ be two Banach spaces which embed into a topological vector space space $\tilde{X}$. We say that $(X_0,X_1)$ is an interpolation couple. Then the sum $X_0 + X_1 \ov{\mathrm{def}}{=} \{x \in \tilde{X} : x=x_1+x_2 \text{ for some } x_1 \in X_1,x_2 \in X_2\}$ is well-defined and equipped with the norm
$$
\norm{x}_{X_0+X_1}
\ov{\mathrm{def}}{=} \inf_{x=x_0+x_1} \big\{\norm{x_0}_{X_0}+\norm{x_1}_{X_1}\big\}.
$$
The intersection $X_0 \cap X_1$ is equipped with the norm $
\norm{x}_{X_0 \cap X_1}
\ov{\mathrm{def}}{=} \max\{\norm{x}_{X_0},\norm{x}_{X_1}\}$. 
Consider the closed strip $\ovl{S} \ov{\mathrm{def}}{=} \{z \in \mathbb{C} : 0 \leq \Re z \leq 1\}$. Let us denote by $\mathscr{F}(X_0,X_1)$ the family of bounded continuous functions $f \co \ovl{S} \to X_0+X_1$, holomorphic on the open strip $S \ov{\mathrm{def}}{=} \{z \in \mathbb{C} : 0< \Re z <1\}$ inducing continuous functions $\R\to X_0$, $t \mapsto f(\i t)$ and $\R \to X_1$, $t \mapsto f(1+\i t)$ which tend to 0 when $|t|$ goes to $\infty$. For any function $f \in \mathscr{F}(X_0,X_1)$, we set
\begin{equation}
\label{norm-funct-inter}
\norm{f}_{\mathscr{F}(X_0,X_1)}
\ov{\mathrm{def}}{=} \max \left\{\sup_{t \in \R} \norm{f(\i t)}_{X_0}, \sup_{t \in \R} \norm{f(1+\i t)}_{X_1}\right\}.
\end{equation}
If $0 \leq \theta \leq 1$, we define the subspace $
(X_0,X_1)_\theta
\ov{\mathrm{def}}{=} \big\{f(\theta) : f \in \mathscr{F}(X_0,X_1)\big\}$ of the Banach space $X_0+X_1$. For any $x \in (X_0,X_1)_\theta$, we let
$$
\norm{x}_{(X_0,X_1)_\theta}
\ov{\mathrm{def}}{=} \inf \big\{\norm{f}_{\mathscr{F}(X_0,X_1)} : f \in \mathscr{F}(X_0,X_1), f(\theta)=x\big\}.
$$
Then by \cite[Theorem 4.1.2 p.~98]{BeL76} $(X_0,X_1)_\theta$ equipped with this norm is a Banach space.

For any $0 < \theta < 1$ and any $x \in X_0 \cap X_1$, by \cite[Corollary 2.8 p.~53]{Lun18} we have
\begin{equation}
\label{ine-interpolation}
\norm{x}_{(X_0,X_1)_\theta} 
\leq \norm{x}_{X_0}^{1-\theta} \norm{x}_{X_1}^\theta.
\end{equation}
If $0 < \theta < 1$ and if we have a bounded inclusion $X_0 \subset X_1$ then by essentially \cite[Proposition 2.4 p.~50]{Lun18} we have a bounded inclusion
\begin{equation}
\label{inclusions-interpolation}
X_0
\subset (X_0,X_1)_\theta.
\end{equation}

\paragraph{Operator space theory} We refer to the books \cite{BLM04}, \cite{EfR00}, \cite{Pau02} and \cite{Pis03} for background on operator spaces and completely bounded maps. If $E$ and $F$ are operator spaces, we denote by $\CB(E,F)$ the space of all completely bounded maps endowed with the norm
\begin{equation}
\label{norm-cb8}
\norm{T}_{\cb, E \to F}
\ov{\mathrm{def}}{=} \sup_{n \geq 1} \bnorm{\Id_{\M_n} \ot T}_{\M_n(E) \to \M_n(F)}.
\end{equation}
It is well-known that by \cite[Corollary 1.2 p.~19]{Pis98} a linear map $T \co E \to F$ is completely bounded if and only if it induces a bounded map $\Id_{S^p} \ot T \co S^p[E] \to S^p[E]$ where $S^p[E]$ is the vector-valued Schatten space of \cite[p.~18]{Pis98}. In this case, we have
\begin{equation}
\label{norm-cb-Sp}
\norm{T}_{\cb, E \to F}
=\norm{\Id_{S^p} \ot T}_{\cb,S^p[E] \to S^p[F]}.
\end{equation}
If $E$ is an operator space and if $Y$ is a Banach space, by \cite[(3.2) p.~71]{Pis03} \cite[(1.10) p.~10]{BLM04} each bounded linear operator $T \co E \to Y$ satisfies
\begin{equation}
\label{min-et-cb}
\norm{T}_{\cb,E \to \min(Y)}
=\norm{T}_{E \to Y}
\end{equation}
where $\min(Y)$ denotes the minimal operator space structure on the Banach space $Y$. Recall that the canonical operator space structure on the commutative algebra $\L^\infty(\Omega)$ of a $\sigma$-finite measure space $\Omega$ is the minimal one, see \cite[p.~72]{Pis03}. For an operator space $E$, note that its opposite operator space $E^\op$ is defined by the same space $E$ but by the norms $\norm{\cdot}_{\M_n(E^\op)}$ defined by $\norm{[x_{ij}]}_{\M_n(E^\op)} \ov{\mathrm{def}}{=} \norm{[x_{ji}]}_{\M_n(E)}$ for any matrix $[x_{ij}]$ of $\M_n(E)$. It is known that if $\cal{A}$ is a $\mathrm{C}^*$-algebra, then these matrix norms on $\cal{A}^\op$ coincide with the canonical matrix norms on the <<opposite>> $\mathrm{C}^*$-algebra which is $\cal{A}$ endowed with its reversed multiplication. 

We also recall a useful description of the operator space structure of an operator space $E$. Suppose $1 \leq p \leq \infty$. If $n \geq 1$ is an integer and if the matrix $x=[x_{ij}]$ belongs to $\M_n(E)$, we have by \cite[Lemma 1.7 p.~23]{Pis98} the equality
\begin{equation}
\label{ope-space-struct}
\norm{x}_{\M_n(E)}
=\sup_{\norm{a}_{S^{2p}_n} ,\norm{b}_{S^{2p}_n} \leq 1} \norm{a\cdot x \cdot b}_{S^{p}_n[E]}
\end{equation}
where the products are defined by $[a \cdot x]_{ij} \ov{\mathrm{def}}{=} \sum_{k=1}^n a_{ik}x_{kj}$ and $[x \cdot b]_{ij} \ov{\mathrm{def}}{=} \sum_{k=1}^n b_{kj}x_{ik}$ for any $1 \leq i,j \leq n$. Here $S^{p}_n[E]$ is the vector-valued Schatten space of \cite[p.~18]{Pis03}. 




A dual operator space $E$ is said to have the dual slice mapping property \cite[p.~48]{BLM04} if 
\begin{equation}
\label{normal-minimal}
\CB(F_*,E)
= E \otvn F
\end{equation}
holds for any dual operator space $F$, where $E \otvn F$ is the normal spatial tensor product \cite[p.~134]{EfR00}, i.e.~the weak* closure of $E \ot_{\min} F$ in the dual space $\CB(F_*,E)$. By \cite[Theorem 11.2.5 p.~201]{EfR00}, this is equivalent to $E_*$ possessing the <<operator space approximation property>>. We refer to \cite{Rua92}, \cite{EKR93}, \cite{EfR03} and \cite{BLM04} for more information on $\otvn$.

\paragraph{Noncommutative $\L^p$-spaces} For a detailed understanding of the theory of von Neumann algebras, we direct readers to the books \cite{Bla06}, \cite{KaR97a} and \cite{KaR97b} and to the survey \cite{PiX03} and references therein for information on noncommutative $\L^p$-spaces. Let $\cal{M}$ be a finite von Neumann algebra equipped with a normal finite faithful trace $\tau$. For any $1 \leq p < \infty$, the noncommutative $\L^p$-space $\L^p(\cal{M})$ is defined as the completion of the von Neumann algebra $\cal{M}$ for the norm
\begin{equation}
\label{norm-Lp}
\norm{x}_{\L^p(\cal{M})}
\ov{\mathrm{def}}{=} \tau(|x|^p)^{\frac{1}{p}}, \quad x \in \cal{M}
\end{equation}
where $|x| \ov{\mathrm{def}}{=} (x^*x)^{^{\frac{1}{2}}}$. We sometimes write $\L^\infty(\cal{M}) \ov{\mathrm{def}}{=} \cal{M}$. The trace $\tau$ induces a canonical continuous embedding $j \co \cal{M} \to \cal{M}_*$  of the von Neumann algebra $\cal{M}$ into its predual $\cal{M}_*$ defined by $\la j(x), y \ra_{\cal{M}_*,\cal{M}} = \tau(xy)$ for any $x,y \in \cal{M}$. Moreover, the previous map $j\co \cal{M} \to \cal{M}_*$ extends to a canonical isometry $j_1 \co \L^1(\cal{M})\to \cal{M}_*$ which allows us to identify the Banach space $\L^1(\cal{M})$ with the predual $\cal{M}_*$ of the von Neumann algebra $\cal{M}$.

With the embedding $j \co \cal{M} \to \cal{M}_*$, $(\cal{M},\cal{M}_*)$ is a compatible couple of Banach spaces in the context of complex interpolation theory. For any $1 \leq p \leq \infty$, we have by \cite[p.~139]{Pis03} the isometric interpolation formula
\begin{equation}
\label{Lp-interpolation}
\L^p(\cal{M})
= (\cal{M},\cal{M}_*)_{\frac{1}{p}}.
\end{equation}
There exists a canonical operator space structure on each noncommutative $\L^p$-space $\L^p(\cal{M})$, that can be equally introduced using complex interpolation theory. Recall that this structure is defined by
\begin{equation*}
\label{}
\M_n(\L^p(\cal{M}))
\ov{\mathrm{def}}{=} (\M_n(\cal{M}),\M_n(\cal{M}_*^\op))_{\frac{1}{p}}, \quad n \geq 1
\end{equation*}
where we use the opposed operator space structure on the predual $\cal{M}_*$. In the second proof of Corollary \ref{th-spectral-coulhon}, we will use the following quite elementary observation. 

\begin{prop}
\label{Prop-inclusion-cb}
Let $\cal{M}$ be a finite von Neumann algebra equipped with a normal finite faithful trace $\tau$. Suppose $1 \leq p \leq \infty$. The inclusion $\L^p(\cal{M}) \subset \L^1(\cal{M})$ is completely bounded with completely bounded norm less than $\tau(1)^{1-\frac{1}{p}}$.
\end{prop}

\begin{proof}
Let $x \in \M_n(\cal{M})$. If $a,b \in S^2_n$ satisfy $\norm{a}_{S^2_n} \leq 1$ and $\norm{b}_{S^2_n} \leq 1$, we have using H\"older's inequality
\begin{align*}
\MoveEqLeft
\norm{(a \ot 1) x (b \ot 1)}_{\L^1(\M_n \otvn \cal{M})}         
\leq \norm{a \ot 1}_{\L^2(\M_n(\cal{M}))} \norm{x}_{\M_n(\L^\infty(\cal{M}))}\norm{b \ot 1}_{\L^2(\M_n(\cal{M}))} \\
&=\big[(\tr \ot \tau)(a^*a \ot 1)\big]^{\frac{1}{2}} \norm{x}_{\M_n(\L^\infty(\cal{M}))} \big[(\tr \ot \tau)(b^*b \ot 1)\big]^{\frac{1}{2}} \\
&=\big[\tr a^*a\big]^{\frac{1}{2}} \norm{x}_{\M_n(\L^\infty(\cal{M}))} \big[\tr b^*b\big]^{\frac{1}{2}} \tau(1) \\
&=\norm{a}_{S^2_n} \norm{x}_{\M_n(\L^\infty(\cal{M}))} \norm{b}_{S^2_n} \tau(1) 
\leq \tau(1) \norm{x}_{\M_n(\L^\infty(\cal{M}))}.
\end{align*}
Recall that we have $\L^1(\M_n \otvn \cal{M})=S^1_n[\L^1(\cal{M})]$ isometrically. Taking the supremum, we obtain with the case $p=1$ of \eqref{ope-space-struct}
$$
\norm{x}_{\M_n(\L^1(\cal{M}))}
\leq \tau(1) \norm{x}_{\M_n(\L^\infty(\cal{M}))}.
$$
We conclude with an argument of interpolation.
\end{proof}

\paragraph{Vector-valued noncommutative $\L^p$-spaces} 
Pisier pioneered the study of vector-valued noncommutative $\L^p$-spaces, focusing on the case where the underlying von Neumann algebra $\cal{M}$ is approximately finite-dimensional and comes with a normal semifinite faithful trace, as outlined in \cite{Pis98}. Within this framework, Pisier demonstrated that for any given operator space $E$, the Banach spaces $\L^\infty_0(\cal{M},E) \ov{\mathrm{def}}{=}\cal{M} \ot_{\min} E$ and $\L^1(\cal{M},E) \ov{\mathrm{def}}{=} \cal{M}_*^{\op} \widehat{\ot} E$ can be injected into a common topological vector space, as detailed in \cite[pp.~37-38]{Pis98}.  This result, relying crucially on the approximate finite-dimensionality of the von Neumann algebra $\cal{M}$, allows everyone to use complex interpolation theory in order to define vector-valued noncommutative $\L^p$-spaces by letting $
\L^p(\cal{M},E)
\ov{\mathrm{def}}{=} (\L^\infty_0(\cal{M},E),\L^1(\cal{M},E))_{\frac{1}{p}}$ for any $1 \leq p \leq \infty$. Actually, if $E$ is a \textit{dual} operator space, we can replace the space $\L^\infty_0(\cal{M},E)$ by the space 
\begin{equation}
\label{def-Lp-vec-1}
\L^\infty(\cal{M},E) \ov{\mathrm{def}}{=} \cal{M} \otvn E
\end{equation}
by the same argument that the one of \cite[p.~39]{Pis98} and we have
\begin{equation}
\label{Lp-vect}
\L^p(\cal{M},E)
= (\L^\infty(\cal{M},E),\L^1(\cal{M},E))_{\frac{1}{p}}.
\end{equation}

\paragraph{Operator theory} We also need basic results in operator theory that we now recall. The following well-known result is \cite[Corollary 5.6 p.~144]{Tha92}.

\begin{thm}
\label{Th-unit-eq}
Let $T$ be a closed densely defined operator on a Hilbert space $H$. Then the operator $T^*T$ on $(\ker T)^\perp$ is unitarily equivalent to the operator $TT^*$ on $(\ker T^*)^\perp$.
\end{thm}
If $T$ is a densely defined operator acting on a Banach space $Y$ then by \cite[Problem 5.27 p.~168]{Kat76} we have
\begin{equation}
\label{lien-ker-image}
\ker T^*
=(\Ran T)^\perp.
\end{equation}
If in addition $T$ is closed and $Y$ is a Hilbert space, we will also have by \cite[Exercise 2.8.45 p.~171]{KaR97a} the following classical equalities 
\begin{equation}
\label{Kadison-image}
\Ran{T^*T}
=\Ran{T^*}
\quad \text{and} \quad 
\ker T^*T
=\ker T. 
\end{equation}
If $A$ is a sectorial operator acting on a \textit{reflexive} Banach space $Y$, we have by \cite[Proposition 2.1.1 (h) p.~20]{Haa06} a decomposition
\begin{equation}
\label{decompo-reflexive}
Y
=\ker A \oplus \ovl{\Ran A}.
\end{equation}


\paragraph{Semigroup theory} Let $(T_t)_{t \geq 0}$ be a strongly continuous semigroup of bounded linear operators acting on a Banach space $Y$ with infinitesimal generator $-A$. This means that $T_t=\e^{-tA}$ for any $t \geq 0$. Recall that $A$ is sectorial, see \cite[p.~24]{Haa06}. 
Furthermore, by \cite[Lemma 1.3 p.~50]{EnN00}, if $x \in \dom A$ and $t \geq 0$, then $T_t(x)$ belongs to $\dom A$ and
\begin{equation}
\label{A-et-Tt-commute}
T_tA(x)
=AT_t(x), \quad t \geq 0.
\end{equation}
Introducing the fixed-point subspace
\begin{equation}
\label{def-fix}
\Fix (T_{t})_{t \geq 0}
\ov{\mathrm{def}}{=} \{ x \in Y : T_t(x)=x \text{ for any } t \geq 0 \}
\end{equation}
of the semigroup $(T_{t})_{t \geq 0}$, we have by \cite[Corollary 3.8 p.~278]{EnN00} the equality $\Fix (T_{t})_{t \geq 0}=\ker A$. We also use the notation \eqref{def-fix} for a weak* continuous semigroup acting on a dual Banach space.

We say that a strongly continuous semigroup $(T_{t})_{t \geq 0}$ is bounded if there exists a positive constant $M$ such that $\norm{T_t}_{Y \to Y} \leq M$ for any $t \geq 0$. The following result is \cite[Corollary 4.11 p.~344]{EnN06}.

\begin{prop}
\label{Prop-mean}
Let $(T_t)_{t \geq 0}$ be a bounded strongly continuous semigroup with generator $-A$ on a Banach space $Y$ such that the operator $-A$ has compact resolvent. Then $(T_t)_{t \geq 0}$ is uniformly mean ergodic and the mean ergodic projection $P \co Y \to Y$ has finite rank.
\end{prop}
Recall that this bounded projection satisfies $\Ran P=\Fix (T_{t})_{t \geq 0}$. 

\paragraph{Symmetric sub-Markovian semigoups}  
Let $\cal{M}$ be a finite von Neumann algebra equipped with a normal finite faithful trace $\tau$. A bounded contraction $T \co \cal{M} \to \cal{M}$ is said to be selfadjoint \cite[Section 5]{JMX06} if $\tau(T(x)y^*)=\tau(xT(y)^*)$ for any $x, y \in \cal{M}$. Such an operator induces a bounded operator $T \co \L^p(\cal{M}) \to \L^p(\cal{M})$ for any $1 \leq p <\infty$, which is is selfadjoint if $p=2$.

A symmetric sub-Markovian semigoup $(T_{t})_{t \geq 0}$ is a weak* continuous semigroup of weak* continuous completely positive contractive selfadjoint operators acting on $\cal{M}$.  
In particular, we have $\tau(T_t(x))=\tau(x^{\frac{1}{2}}T_t(1)x^{\frac{1}{2}}) \leq \tau(x)$ using \cite[1.6.8 and 1.6.9]{Dix77} for any \textit{positive} element $x \in \cal{M}$. Such a semigroup extends to a strongly continuous semigroup of contractions on the Banach space $\L^p(\cal{M})$ for any $1 \leq p < \infty$ and $(T_t)_{t \geq 0}$ is a semigroup of selfadjoint operators on the complex Hilbert space $\L^{2}(\cal{M})$. We refer to \cite[Section 5]{JMX06} for more information where the terminology <<completely positive diffusion semigroup>> is used. We warn the reader that different definitions coexist in the literature. We say that a sub-Markovian semigoup $(T_{t})_{t \geq 0}$ is Markovian if in addition each operator $T_t$ is unital, i.e.~$T_t(1)=1$.

By \cite[Theorem 2.4]{KuN79}, such a semigroup is weak* mean ergodic and the corresponding projection onto the weak* closed fixed-point subalgebra $\{ x \in \cal{M} : T_t(x)=x \text{ for any } t \geq 0 \}$ is a conditional expectation $\E \co \cal{M} \to \cal{M}$. By \cite[Theorem 1.2]{KuN79}, it also satisfies the equalities
\begin{equation}
\label{E-Tt}
T_t\E=\E T_t=\E, \quad t \geq 0.
\end{equation}
Since $\E$ belongs to the closed convex hull of the set $\{T_t : t \geq 0\}$ in the point weak* topology, a classical argument using complex interpolation shows that the map $\E \co \L^\infty(\cal{M}) \to \L^\infty(\cal{M})$ admits a contractive $\L^p$-extension $\E_p \co \L^p(\cal{M}) \to \L^p(\cal{M})$ for any $1 \leq p \leq \infty$. We will use the classical notation 
\begin{equation}
\label{Lp0}
\L_0^p(\cal{M})
\ov{\mathrm{def}}{=} \ker \E_p.
\end{equation} 
If $-A_p$ is the infinitesimal generator of the strongly continuous semigroup $(T_{t})_{t \geq 0}$ on the Banach space $\L^p(\cal{M})$ then we have $\L_0^p(\cal{M})=\ovl{\Ran A_p}$ and the decomposition $\L^p(\cal{M})=\L_0^p(\cal{M}) \oplus \ker A_p$ where $1 \leq p<\infty$. In other words, the subspace $\L^p_0(\cal{M})$ is contractively complemented in the Banach space $\L^p(\cal{M})$ and the associated contractive projection $\E_p \co \L^p(\cal{M}) \to \L^p(\cal{M})$ is induced by interpolation by two compatible contractive projections $\E \co \L^\infty(\cal{M}) \to \L^\infty(\cal{M})$ and $\E_1 \co \L^1(\cal{M}) \to \L^1(\cal{M})$. By the way, using \cite[Theorem 1 p.~118]{Tri95} or \cite[p.~58]{Pis03}, we have the isometric formula $\L_0^p(\cal{M})=(\L_0^\infty(\cal{M}),\L_0^1(\cal{M}))_{\frac{1}{p}}$ from \eqref{Lp-interpolation}. 

Finally, for $1 \leq p \leq q \leq \infty$ and some constant $d>0$ consider the property
\begin{equation}
\label{Rnpq}
\norm{T_t}_{\L^p(\cal{M}) \to \L^q(\cal{M})}
\lesssim \frac{1}{t^{\frac{d}{2}(\frac{1}{p}-\frac{1}{q})}}, \quad 0 < t \leq 1.
\end{equation}
By a classical interpolation argument of Coulhon and Raynaud \cite[Lemma 1]{Cou90},  generalized to the noncommutative case, this property holds for one pair $1 \leq p < q \leq \infty$ if and only if it holds for all $1 \leq p \leq q \leq \infty$. For the following estimate
\begin{equation}
\label{cbRnpq}
\norm{T_t}_{\cb,\L^p(\cal{M}) \to \L^q(\cal{M})}
\lesssim \frac{1}{t^{\frac{d}{2}(\frac{1}{p}-\frac{1}{q})}}, \quad 0 < t \leq 1,
\end{equation}
we can easily prove the next result in the same spirit.

\begin{lemma}
\label{Lemma-interpolation}
The estimate \eqref{cbRnpq} is true for $p=1$ and $q=\infty$ if and only if \eqref{cbRnpq}  also holds true for $1 \leq p < q \leq \infty$. 
\end{lemma}

\begin{proof}
By interpolation, we obtain using the assumption \eqref{cbRnpq} for $p=1$ and $q=\infty$ in the last inequality
\begin{align}
\MoveEqLeft
\label{div-877}
\norm{T_t}_{\cb,\L^p(\cal{M}) \to \L^\infty(\cal{M})}
\leq \norm{T_t}_{\cb,\L^\infty(\cal{M}) \to \L^\infty(\cal{M})}^{1-\frac{1}{p}} \norm{T_t}_{\cb,\L^1(\cal{M}) \to \L^\infty(\cal{M})}^{\frac{1}{p}} 
\lesssim \frac{1}{t^{\frac{d}{2p}}}.
\end{align}
Note that $\L^q(\cal{M})=(\L^\infty(\cal{M}),\L^p(\cal{M}))_{\frac{p}{q}}$ isometrically by the reiteration theorem \cite[Theorem 4.6.1 p.~101]{BeL76} and by the interpolation formula \eqref{Lp-interpolation}. Using again interpolation, we conclude that
\begin{align*}
\MoveEqLeft
\norm{T_t}_{\cb,\L^p(\cal{M}) \to \L^q(\cal{M})}
\leq \norm{T_t}_{\cb,\L^p(\cal{M}) \to \L^\infty(\cal{M})}^{1-\frac{p}{q}} \norm{T_t}_{\cb,\L^p(\cal{M}) \to \L^p(\cal{M})}^{\frac{p}{q}} 
\ov{\eqref{div-877}}{\lesssim} \bigg(\frac{1}{t^{\frac{d}{2p}}}\bigg)^{1-\frac{p}{q}}
=\frac{1}{t^{\frac{d}{2}(\frac{1}{p}-\frac{1}{q})}}.
\end{align*}


Now, we prove the converse. We can write $\frac{1}{p}=\frac{\theta}{1}+\frac{1-\theta}{q}$ for some $0 < \theta < 1$. We have
$
\frac{d}{2}(\frac{1}{p}-\frac{1}{q})
=\theta\frac{d}{2}(1-\frac{1}{q})$. 
Consider an element  $x \in \M_n(\L^1(\cal{M}))$ with $ \norm{x}_{\M_n(\L^1(\cal{M}))} \leq 1$. We let
\begin{equation}
\label{def-Cx}
C_x
\ov{\mathrm{def}}{=} \sup_{0 \leq t \leq 1} t^{\frac{d}{2}(1-\frac{1}{q})} \norm{(\Id \ot T_t)(x)}_{\M_n(\L^q(\cal{M}))}.
\end{equation}
For any $0 < t \leq 1$, using the complete boundedness of $T_{\frac{t}{2}}$ on the space $\L^1(\cal{M})$, we obtain
\begin{align*}
\MoveEqLeft
\norm{(\Id \ot T_t)(x)}_{\M_n(\L^q(\cal{M}))}         
=\bnorm{(\Id \ot T_{\frac{t}{2}})(\Id \ot T_{\frac{t}{2}})(x)}_{\M_n(\L^q(\cal{M}))} \\
&\leq \bnorm{\Id \ot T_{\frac{t}{2}}}_{\M_n(\L^p(\cal{M})) \to \M_n(\L^q(\cal{M}))}
\bnorm{(\Id \ot T_{\frac{t}{2}})(x)}_{\M_n(\L^p(\cal{M}))} \\
&\ov{\eqref{norm-cb8}\eqref{cbRnpq}\eqref{ine-interpolation}}{\lesssim} \frac{1}{t^{\frac{d}{2}(\frac{1}{p}-\frac{1}{q})}} \bnorm{(\Id \ot T_{\frac{t}{2}})(x)}_{\M_n(\L^1(\cal{M}))}^\theta \bnorm{(\Id \ot T_{\frac{t}{2}})(x)}_{\M_n(\L^q(\cal{M}))}^{1-\theta} \\
&\leq \frac{1}{t^{\frac{d}{2}(\frac{1}{p}-\frac{1}{q})}} \bnorm{(\Id \ot T_{\frac{t}{2}})(x)}_{\M_n(\L^q(\cal{M}))}^{1-\theta}\\
&\ov{\eqref{def-Cx}}{\leq}\frac{1}{t^{\frac{d}{2}(\frac{1}{p}-\frac{1}{q})}} \frac{1}{t^{(1-\theta)\frac{d}{2}(1-\frac{1}{q})}} C_x^{1-\theta}
= \frac{1}{t^{\frac{d}{2}(1-\frac{1}{q})}} C_x^{1-\theta}.
\end{align*}
Consequently, we have  $C_x \lesssim C_x^{1-\theta}$. Hence $C_x \leq M$ for some constant $M$ which is independent of $x$. So, we obtain the estimate $\norm{T_t}_{\cb,\L^1(\cal{M}) \to \L^q(\cal{M})} \lesssim \frac{1}{t^{\frac{d}{2}(1-\frac{1}{q})}}$ for any $0 < t \leq 1$. By duality, we deduce that for any $0 < t \leq 1$
$$
\norm{T_t}_{\cb,\L^{q^*}(\cal{M}) \to \L^\infty(\cal{M})} 
\lesssim \frac{1}{t^{\frac{d}{2}(1-\frac{1}{q})}}
=\frac{1}{t^{\frac{d}{2}(\frac{1}{q^*})}}
=\frac{1}{t^{\frac{d}{2}(\frac{1}{q^*}-\frac{1}{\infty})}}.
$$
Applying the first part with this estimate, we obtain $\norm{T_t}_{\cb,\L^1(\cal{M}) \to \L^\infty(\cal{M})} \lesssim \frac{1}{t^{\frac{d}{2}}}$ for any $0 < t \leq 1$.
\end{proof}

The following result is a slight generalization of \cite[Lemma 3.1 p.~3426]{GJL20} (see also \cite[Proposition 5.3 p.~346]{StS11} for the commutative case and \cite[Proposition 7.2 (i) p.~127]{Pis03}) which can proved with the same method. Easy details are left to the reader.

\begin{prop}
\label{prop-Junge-square}
Consider a $*$-preserving selfadjoint operator $T \co \cal{M} \to \cal{M}$ acting on a finite von Neumann algebra $\cal{M}$ equipped with a normal finite faithful trace. Then $T$ induces a completely bounded map from $\L^2(\cal{M})$ into $\L^\infty(\cal{M})$ if and only if $T^2$ induces a completely bounded map from $\L^{1}(\cal{M})$ into $\L^\infty(\cal{M})$. In this case, we have
\begin{equation}
\label{formule-square}
\norm{T}_{\cb,\L^2(\cal{M}) \to \L^\infty(\cal{M})}^2
=\norm{T^2}_{\cb,\L^{1}(\cal{M}) \to \L^\infty(\cal{M})}.
\end{equation}
A similar result is true for bounded maps instead of completely bounded maps. 
\end{prop}

We need the following part of \cite[Proposition 3.2 p.~3427]{GJL20} which gives an exponential decay on some completely bounded norm for large times. 

\begin{prop}
\label{prop-exp}
Let $\cal{M}$ be a finite von Neumann algebra equipped with a normal finite faithful trace.  Consider a semigroup $(T_t)_{t \geq 0}$ of $*$-preserving selfadjoint maps acting on $\cal{M}$ satisfying \eqref{cbRnpq} such that the operator $A_2 \co \dom A_2 \subset \L^2(\cal{M}) \to \L^2(\cal{M})$ has a spectral gap\footnote{\thefootnote.This means that there exists $\lambda >0$ such that $\lambda \norm{f}_{\L^2(\cal{M})}^2 \leq \la A_2 f,f \ra_{\L^2(\cal{M})}$.}.
Then there exists $\omega >0$ such that
\begin{equation}
\label{estimate-exp}
\norm{T_t}_{\cb,\L_0^1(\cal{M}) \to \L^\infty_0(\cal{M})} 
\lesssim \e^{-\omega t}, \quad t \geq 1.
\end{equation}
\end{prop}

\paragraph{Weak Schatten ideals}
Let $H$ be a Hilbert space. If $0 < p < \infty$, recall that the weak Schatten space is defined by 
\begin{equation}
\label{weak-Schatten}
S^{p,\infty}(H) 
\ov{\mathrm{def}}{=} \Big\{T \in S^\infty(H) : s_n(T)=O_{n \to \infty}\Big(\tfrac{1}{n^{\frac{1}{p}}}\Big) \Big\}.
\end{equation} 
Here $s_1(T),s_2(T),\ldots$ denotes the singular values of the compact operator $T \co H \to H$. These are the eigenvalues of the operator $|T| \ov{\mathrm{def}}{=} (T^*T)^{\frac{1}{2}}$ arranged in descending order. By \cite[Exercise p.~316]{GVF01}, we have the inclusions
\begin{equation}
\label{inclusions-weak-Schatten}
S^{p}(H)  
\subset S^{p,\infty}(H) 
\subset S^{\alpha}(H), \quad \alpha > p.
\end{equation}

\paragraph{Possibly kernel-degenerate compact spectral triples}
Consider a triple $(\cal{A},H,D)$ constituted of the following data: a Hilbert space $H$, a closed unbounded operator $D$ on $H$ with dense domain $\dom D \subset H$, an algebra $\cal{A}$ equipped with a homomorphism $\pi \co \cal{A} \to \B(H)$. In this case, we define the Lipschitz algebra
\begin{align}
\label{Lipschitz-algebra-def}
\MoveEqLeft
\Lip_D(\cal{A}) 
\ov{\mathrm{def}}{=} \big\{a \in \cal{A} : \pi(a) \cdot \dom D \subset \dom D 
\text{ and the unbounded operator } \\
&\qquad \qquad  [D,\pi(a)] \co \dom D \subset H \to H \text{ extends to an element of } \B(H) \big\}. \nonumber
\end{align} 
By \cite[Proposition 5.11 p.~219]{ArK22}, this is a subalgebra of the algebra $\cal{A}$. We will use the standard notation $[D,a]$ for the commutator $[D,\pi(a)]$ for any $a \in \cal{A}$.

Now, we introduce the following definition. We essentially follow \cite[Definition 2.1]{CGIS14} and \cite[Definition 5.10 p.~218]{ArK22} but at the level of von Neumann algebras since we want to do analysis on classical or noncommutative $\L^p$-spaces and these Banach spaces are defined with von Neumann algebras. 

\begin{defi}
\label{def-spec}
We say that $(\cal{A},H,D)$ is a possibly kernel-degenerate compact spectral triple if $H$ is a Hilbert space, $\cal{A}$ is a von Neumann algebra, $\pi$ is a $*$-homomorphism, $D$ is a selfadjoint operator on $H$ satisfying  
\begin{enumerate}
\item{} $D^{-1}$ is a compact operator on $\ovl{\Ran D} \ov{\eqref{lien-ker-image}}{=} (\ker D)^\perp$,
\item{} the subset $\Lip_D(\cal{A})$ is weak* dense in $\cal{A}$.
\end{enumerate}
\end{defi}
To shorten the terminology, we would just use the words <<compact spectral triple>>. Moreover, we can replace $D^{-1}$ by $|D|^{-1}$ in the first point by an elementary functional calculus argument, see \cite[Proposition 5.11 p.~219]{ArK22}.

With this definition, the kernel $\ker D$ can be infinite-dimensional. Note that in the classical  definition \cite[Definition 1.1]{EcI18} \cite[Definition 9.16 p.~400]{GVF01} of spectral triples, the first condition is replaced by the condition that the operator $D$ has compact resolvent and consequently has a finite-dimensional kernel by \cite[Theorem 6.29 p.~187]{Kat76}. Indeed, if an operator $D$ has a finite-dimensional kernel, it satisfies the first condition if and only if it has compact resolvent. So, we can see this definition as an extension of the classical one. This explains the terminology <<possibly kernel-degenerate>>.





\paragraph{Spectral dimension}
Let $D$ be a selfadjoint unbounded operator on a Hilbert space $H$. By \cite[Proposition 5.3.38]{Ped89}, we have $\ker |D|=\ker D$. Moreover, the operator $|D|^{-1}$ is well-defined on $\ovl{\Ran D} \ov{\eqref{lien-ker-image}}{=} (\ker D)^\perp$. Furthermore, we can extend it by letting $|D|^{-1}=0$ on $\ker D$. We now introduce a concept of summability as defined below, a notion we previously defined in our earlier paper \cite{Arh24}. Note that the definition of the spectral dimension of \cite{Arh24} is different and we thank one of the referees for advising us to use the following definition (in the spirit of \cite[Definition 7.1.4 p.~392]{LMSZ23}) instead.

\begin{defi}
\label{defi-summable}
We say that a compact spectral triple $(\cal{A},H,D)$ is $p$-summable for some $p > 0$ if $\tr |D|^{-p} < \infty$,  that is if the operator $|D|^{-1}$ belongs to the Schatten class $S^p(H)$. The triple is said to be $p^+$-summable if $|D|^{-1}$ belongs to the weak Schatten space $S^{p,\infty}(H)$. In this case, the spectral dimension of the spectral triple is defined by
\begin{equation}
\label{Def-spectral-dimension}
\dim(\cal{A},H,D)  
\ov{\mathrm{def}}{=} \inf\big\{ p > 0 : |D|^{-1} \in S^{p,\infty}(H) \big\}.
\end{equation}
\end{defi}

If $p > 0$, note that by \eqref{inclusions-weak-Schatten} a $p^+$-summable compact spectral triple is $\alpha$-summable for any $\alpha > p$. Consider a <<classical>> spectral triple $(\cal{A},H,D)$ as in \cite[Definition 1.1 p.~1]{EcI18} whose projection on the finite-dimensional subspace $\ker D$ is $P_0 \co H \to H$. The classical definition \cite[p.~4]{EcI18} of $p$-summability is $\tr|D+P_0|^{-p} <\infty$ and is clearly equivalent to $|D|^{-p}$ being trace-class on the complement of its kernel. Hence the first part of Definition \ref{defi-summable} can be seen as an extension of the standard definition \cite[Definition 1.1 p.~1]{EcI18} (see also \cite[p.~450]{GVF01} and \cite[p.~38 and Definition 6.2 p.~47]{CPR11}). 

Note that by \cite[p.~489]{GVF01} the <<commutative>> compact spectral triple associated to a compact Riemannian spin manifold of dimension $d$ admits a spectral dimension equal to $d$. It is worth noting that <<finite>> compact spectral triples as defined in \cite[Definition 2.19]{Van15} are 0-dimensional and that a spectral triple $(\cal{A},H,D)$ with an infinite-dimensional Hilbert space $H$ can also have a spectral dimension equal to 0. Standard examples are associated with Podle\'s spheres and we refer to \cite[p.~631]{EIS14} and \cite[Proposition B.9]{EcI18} for more information.



The following useful lemma is \cite[Lemma 6.2]{Arh24}. 

\begin{lemma} 
\label{Lemma-theta-summable}
Let $D$ be a selfadjoint unbounded operator on a Hilbert space $H$. If the operator $|D|^{-p}$ is trace-class for some $p > 0$ then for any $t>0$ the operator $\e^{-t D^2}$ is also trace-class and we have
\begin{equation}
\label{ul-esti2}
\tr \e^{-tD^2}
\lesssim \frac{1}{t^{\frac{p}{2}}}, \quad t>0.
\end{equation}
\end{lemma}

\paragraph{Dirichlet forms on noncommutative $\L^2$-spaces} Here, we recall some definitions concerning Dirichlet forms. These forms are in a one-to-one correspondence  with symmetric sub-Markovian semigroups and generalize the Dirichlet
energy integral $\cal{E}(f)=\int_\Omega |\nabla f|^2 \d x$ for a bounded open subset $\Omega$ of $\R^d$. A key feature of this last energy functional is that its value does not increase when the function $f$ is replaced by $f \wedge 1 \ov{\mathrm{def}}{=} \inf\{f,1\}$, i.e.~we have $\cal{E}(f \wedge 1) \leq \cal{E}(f)$.

 We refer to \cite{Cip97}, \cite{Cip08}, \cite{Cip16}, and \cite{CiS03} and references therein for background on Dirichlet forms in the noncommutative context. 
Let $\cal{M}$ be a von Neumann algebra equipped with a normal finite faithful trace $\tau$. Consider a quadratic form $\cal{E} \co \dom \cal{E} \to [0,\infty[$ defined on a dense selfadjoint subspace $\dom \cal{E}$ of the Hilbert space $\L^2(\cal{M})$. We say that $\cal{E}$ is real if $\cal{E}(f^*) = \cal{E}(f)$ for all $f \in \dom \cal{E}$. The form $\cal{E}$ is said to be Markovian if for all selfadjoint element $f$ of $\dom \cal{E}$ the element $f \wedge 1$ belongs to $\dom \cal{E}$ and if $\cal{E}(f \wedge 1) \leq \cal{E}(f)$. For each integer $n \geq 1$, we can define the amplification $\cal{E}_n \co S^2_n \ot \dom \cal{E} \to [0,\infty[$ of $\cal{E}$ by
$$
\cal{E}_n([f_{ij}]) 
\ov{\mathrm{def}}{=} \sum_{i,j=1}^n \cal{E}(f_{ij})
$$ 
where $[f_{ij}]$ belongs to the space $S^2_n \ot \dom \cal{E}$. We say that $\cal{E}$ is completely Markovian if $\cal{E}_n$ is Markovian for all integer $n \geq 1$. A completely Markovian real densely defined quadratic form $\cal{E}$ is said to be a completely Dirichlet form if it is lower semicontinuous on $\L^2(\cal{M})$. According to \cite[Theorem 3.5]{Cip16} (see \cite[Proposition 3.2.1 p.~12]{BoH91} for the commutative case), completely bounded Dirichlet forms correspond one-to-one, via a generalized Beurling–Deny correspondence, with symmetric sub-Markovian semigroups of operators acting on $\cal{M}$. If $-A_2$ is the generator of a symmetric sub-Markovian semigroup $(\e^{-tA_2})_{t \geq 0}$, we can recover the completely Dirichlet form $\cal{E}$ by letting $\dom \cal{E} \ov{\mathrm{def}}{=} \dom A_2^{\frac{1}{2}}$ and
\begin{equation}
\label{Beurling-Deny}
\cal{E}(f)
\ov{\mathrm{def}}{=} \bnorm{A_2^{\frac{1}{2}}(f)}_{\L^2(\cal{M})}^2, 
\quad f \in \dom A_2^{\frac{1}{2}}.
\end{equation}




\paragraph{Karamata's Theorem} We recall the well-known Karamata Tauberian theorem \cite{Kar31}, see \cite[Sec.~XIII.5, Theorem 2 p.~445]{Fel71}. 
For a lower bounded sequence $(\lambda_n)_{n \geq 1}$ of real numbers, this classical result connects the large asymptotics of the counting function $N(\lambda) \ov{\mathrm{def}}{=} \card\{n \geq 1 : \lambda_n \leq \lambda \}$ to the small time asymptotics $t$ of the sum $\sum_{n=1}^{\infty} \e^{-\lambda_n t}$, which identifies to the trace $\tr(\e^{-tA})$ of a positive operator in our setting. The last part is not included in \cite[Sec.~XIII.5, Theorem 2 p.~445]{Fel71} but is folklore.

\begin{thm}[Karamata]
\label{thm-Karamata}
Let $(\lambda_n)_{n \geq 1}$ be a lower bounded sequence of real numbers such that the series $\sum_{n \geq 1} \e^{-\lambda_n t}$ converges for all $t > 0$. Let $0 < p < \infty$ and $a > 0$. Then the following conditions are equivalent.
\begin{enumerate}
\item $\lim_{t \to 0} t^p \sum_{n=1}^{\infty} \e^{-\lambda_n t} = a$.
\item $\lim_{\lambda \to \infty}  \lambda^{-p} N(\lambda) = \frac{a}{\Gamma(p + 1)}$.
\item We have 
$$
\lim_{n \to \infty} \frac{\lambda_n}{n^{\frac{1}{p}}}  
= \bigg(\frac{\Gamma(p + 1)}{a}\bigg)^{\frac{1}{p}}.
$$
\end{enumerate}
\end{thm}
The proof of the equivalence of the two last parts is elementary since $N(\lambda_n)=n$ for any integer $n \geq 1$. We refer to \cite{AHT18}, \cite{MeS07} and \cite{IoZ23} for one-sided weaker variants of this result.

The following result is folklore. The recent paper \cite{IoZ23} contains related results and essentially the second part of the proof. 

\begin{prop}
\label{th-Karamata-Iochum}
Let $A$ be a positive selfadjoint unbounded operator with compact resolvent acting on a Hilbert space. We denote by $(\lambda_n(A))_{n \geq 1}$ the increasing sequence of its eigenvalues, which is assumed to be infinite. For any $0 < p <\infty$, the following conditions are equivalent.
\begin{enumerate}
	\item $(\e^{-tA})_{t \geq 0}$ is a Gibbs semigroup and $\norm{\e^{-tA}}_{S^1(H)}=O_{t \to 0}(t^{-p})$.
	\item $A \not=0$ and $\lambda_n(A)^{-1}=O_{n \to \infty}(n^{-\frac{1}{p}})$.
\end{enumerate}
\end{prop}

\begin{proof}
We let $\lambda_n \ov{\mathrm{def}}{=} \lambda_n(A)$. If the first property is satisfied, we follow a <<classical argument>> e.g.~\cite[proof of Theorem 1.7]{Wu23}. Consider some integer $n$. For any $k \in \{1,\ldots,n\}$ and any $t>0$ we have $\e^{-t\lambda_n} \leq \e^{-t\lambda_k}$. By summing, we obtain 
$$
n\e^{-t\lambda_n} 
\leq \sum_{k=1}^{n} \e^{-t \lambda_k} 
\leq \sum_{k=1}^{\infty} \e^{-t\lambda_k}
=\tr T_t \lesssim \frac{1}{t^p}.
$$ 
A simple computation shows that the function $t \mapsto \frac{\e^{t\lambda_n}}{t^p}$ takes its minimum at $t=\frac{p}{\lambda_n}$. Hence $
n
\leq \frac{\e^{p}}{(\frac{p}{\lambda_n})^{p}}$. The conclusion is obvious.

Suppose the second assertion. We can assume that each $\lambda_n$ is non-zero. We have $0< \lambda_n^{-1} \leq cn^{-\frac{1}{p}}$ for some positive constant $c$. Using an integral test, we obtain
$$
\norm{\e^{-tA}}_{S^1(H)}
=\sum_{n =1}^\infty \e^{-t\lambda_n} 
\leq \sum_{n = 1}^\infty \e^{-tc^{-1}n^{\frac{1}{p}}} 
\leq \int_{0}^{\infty} \e^{-tc^{-1}x^{\frac{1}{p}}} \d x
\lesssim \frac{1}{t^p}.
$$
\end{proof}

\paragraph{Conditional $\L^p$-spaces}
Suppose that $1 \leq q \leq p \leq \infty$ such that $\frac{1}{2}=\frac{1}{q}-\frac{1}{p}$. Consider an inclusion $\cal{N} \subset \cal{M}$ of von Neumann algebras with $\cal{M}$ equipped with a normal finite faithful trace. If $x \in \cal{M}$, we consider the asymmetric norm
\begin{equation}
\label{conditional-one-sided}
\norm{x}_{\L^{q}_{p,\ell}(\cal{N} \subset \cal{M})}
=\sup_{\norm{b}_{\L^{2}(\cal{N})} = 1} \norm{x b}_{\L^{q}(\cal{M})}.
\end{equation}
We employ the subscript $\ell$ to denote <<left>>. The <<right>> version of this norm is denoted by $\norm{\cdot}_{\L^{p}_{r,\infty}(\cal{N} \subset \cal{M})}$ in the paper \cite[p.~26]{GJL20b} for some suitable $r$.  The Banach space $\L^{q}_{p,\ell}(\cal{N} \subset \cal{M})$ is then the completion of $\cal{M}$ with respect to this norm. If the von Neumann algebra $\cal{M}$ is abelian, it is not difficult to prove that this norm is equal to
\begin{equation}
\label{mixed-norm}
\norm{x}_{\L_p^q(\cal{N} \subset \cal{M})}
\ov{\mathrm{def}}{=} \sup_{\norm{a}_{\L^{4}(\cal{N})} = 1,\norm{b}_{\L^{4}(\cal{N})} = 1} \norm{axb}_{\L^{q}(\cal{M})}. 
\end{equation}
If $\cal{R}$ and $\cal{N}$ are finite von Neumann algebras equipped with normal finite faithful traces, we have by \cite[Example 4.1 (b) p.~71]{JuP10} an isometry
\begin{equation}
\label{mixed-particular}
\L_p^q(\cal{N} \subset \cal{N} \otvn \cal{R}) 
= \L^p(\cal{N},\L^q(\cal{R})).
\end{equation}
We refer to \cite{GJL20b} and \cite{JuP10} for more information on these spaces. It should be noted that our notation differs from that of \cite{JuP10}, but it is the one commonly used today.

\section{Derivations and triples}
\label{sec-triples}


\paragraph{Hilbert bimodules} 
We start by reviewing the concept of Hilbert bimodule which is crucial for defining the derivations that allow the introduction of the Hodge-Dirac operators considered in this paper. Let $\cal{M}$ be a von Neumann algebra. A Hilbert $\cal{M}$-bimodule is a Hilbert space $\cal{H}$ together with a $*$-representation $\Phi \co \cal{M} \to \B(\cal{H})$ and a $*$-anti-representation $\Psi \co \cal{M} \to \B(\cal{H})$ such that $\Phi(x)\Psi(y)=\Psi(y)\Phi(x)$ for any $x,y \in \cal{M}$. For all $x, y \in \cal{M}$ and any $\xi \in \cal{H}$, we let $x\xi y \ov{\mathrm{def}}{=} \Phi(x)\Psi(y)\xi$.  We say that the bimodule is normal if $\Phi$ and $\Psi$ are normal, i.e.~weak* continuous. The bimodule is said to be symmetric if there exists an antilinear involution $\J \co \cal{H} \to \cal{H}$ such that $\J (x\xi y) = y^*\J (\xi)x^*$ for any $x, y \in \cal{M}$ and any $\xi \in \cal{H}$.


\paragraph{$\mathrm{W}^*$-derivations} Now, we introduce a notion of derivation that can be viewed as an abstract version of a gradient. If $(\cal{H},\Phi,\Psi)$ is a Hilbert $\cal{M}$-bimodule, then following \cite[p.~267]{Wea96} we define a $\mathrm{W}^*$-derivation to be a weak* closed densely defined unbounded operator $\partial \co \dom \partial \subset \cal{M} \to \cal{H}$ such that the domain $\dom \partial$ is a weak* dense unital $*$-subalgebra of $\cal{M}$ and 
\begin{equation}
\label{Leibniz}
\partial(fg) 
= f\partial(g) + \partial(f)g, \quad f, g \in \dom \partial.
\end{equation}
We say that a $\mathrm{W}^*$-derivation is symmetric if the bimodule $(\cal{H},\Phi,\Psi)$ is symmetric and if we have $\J(\partial(f)) = \partial(f^*)$ for any $x \in \dom \partial$. In the sequel, we let $\cal{B} \ov{\mathrm{def}}{=} \dom \partial$.

\paragraph{The triple $(\L^\infty(\cal{M}),\L^2(\cal{M}) \oplus_2 \cal{H},D)$}
From a derivation, we will now explain how to introduce a triple in the spirit of noncommutative geometry. Let $\partial \co \dom \partial \subset \cal{M} \to \cal{H}$ be a $\mathrm{W}^*$-derivation where the von Neumann algebra $\cal{M}$ is equipped with a normal faithful finite trace $\tau$. Suppose that the operator $\partial \co \dom \partial \subset \L^2(\cal{M}) \to \cal{H}$ is closable. We denote again its closure by $\partial$. Note that the subspace $\mathcal{B}$ is a core of $\partial$. Recall that it is folklore and well-known that a weak* dense subalgebra of $\cal{M}$ is dense in the space $\L^2(\cal{M})$. As the operator $\partial$ is densely defined and closed, by \cite[Theorem 5.29 p.~168]{Kat76} the adjoint operator $\partial^*$ is densely defined and closed on $\L^2(\cal{M})$ and $\partial^{**}=\partial$. 

Following essentially \cite{HiT13b} and \cite{Cip16}, we introduce the unbounded closed operator $D$ on the Hilbert space $\L^2(\cal{M}) \oplus_2 \cal{H}$ introduced in \eqref{Hodge-Dirac-I} and defined by
\begin{equation}
\label{Def-D-psi}
D(f,g)
\ov{\mathrm{def}}{=}
\big(\partial^*(g),
\partial(f)\big), \quad f \in \dom \partial,\  g \in \dom \partial^*.
\end{equation}
We call it the Hodge-Dirac operator associated to $\partial$ and it can be written as in \eqref{Hodge-Dirac-I}. It is not difficult to check that this operator is selfadjoint. If $f \in \L^\infty(\cal{M})$, we define the bounded operator $\pi_f \co \L^2(\cal{M}) \oplus_2 \cal{H} \to \L^2(\cal{M}) \oplus_2 \cal{H}$ by
\begin{equation}
\label{Def-pi-a}
\pi_f
\ov{\mathrm{def}}{=} \begin{bmatrix}
    \M_f & 0  \\
    0 & \Phi_f  \\
\end{bmatrix}, \quad f \in \L^\infty(\cal{M})
\end{equation}
where the linear map $\M_f \co \L^2(\cal{M}) \to \L^2(\cal{M})$, $g \mapsto fg$ is the multiplication operator by $f$ and where $\Phi_f \co \cal{H} \to \cal{H}$, $h \mapsto fh$ is the left bimodule action. 

Note that completely Dirichlet forms give rise to $\mathrm{W}^*$-derivations, see \cite{CiS03}, \cite{Cip97}, \cite{Cip08}, \cite{Cip16} and \cite{Wir22} and references therein. More precisely, if $\cal{E}$ is a completely Dirichlet form on a noncommutative $\L^2$-space $\L^2(\cal{M})$ with associated semigroup $(T_t)_{t \geq 0}$ and associated operator $A_2$, there exist a symmetric Hilbert $\cal{M}$-bimodule $(\cal{H},\Phi,\Psi,\J)$ and a $\mathrm{W}^*$-derivation $\partial \co \dom \partial \subset \cal{M} \to \cal{H}$ such that $\partial \co \dom \partial \subset \L^2(\cal{M}) \to \cal{H}$ is closable with closure also denoted by $\partial$ and satisfying $A_2=\partial^*\partial$. In other words, this means that the opposite of the infinitesimal generator $-A_2$ of a symmetric sub-Markovian semigroup of operators can always be represented as the composition of a <<divergence>> $\partial^*$ with a <<gradient>> $\partial$. Note that the symmetric Hilbert $\cal{M}$-bimodule $(\cal{H},\Phi,\Psi,\J)$ is not unique since one can always artificially expand $\cal{H}$. However, if we add an additional condition, namely that <<the bimodule is generated by the derivation $\partial$>>, we obtain uniqueness, see \cite[Theorem 6.19]{Wir22} for a precise statement. In this case, the bimodule (and its associated derivation) is called the tangent bimodule associated with the completely Dirichlet form (or the associated semigroup).


The Hodge-Dirac operator $D$ of \eqref{Hodge-Dirac-I} is related to the operator $A_2$ by
\begin{equation}
\label{carre-de-D}
D^2
\ov{\eqref{Hodge-Dirac-I}}{=}
\begin{bmatrix} 
0 & \partial^* \\ 
\partial & 0 
\end{bmatrix}^2
=\begin{bmatrix} 
\partial^*\partial & 0 \\ 
0 & \partial \partial^*
\end{bmatrix}
=\begin{bmatrix} 
A_2 & 0 \\ 
0 & \partial \partial^*
\end{bmatrix}. 
\end{equation} 

\paragraph{Some remarks in the $\Gamma$-regular case} In the following result, we connect the norm of the commutators $[D,f]$ and the norm $\norm{\cdot}_{\L^{2}_{\infty, \ell}(\cal{M} \subset \tilde{\cal{M}})}$ defined in \eqref{conditional-one-sided}, in the case where the completely Dirichlet form has the <<$\Gamma$-regularity>>. This situation, introduced in the unpublished paper \cite{JRS}, is the same that in \cite[p.~3415]{GJL20}, \cite[Theorem 2.1]{BGJ22} or in the tracial case of \cite[Corollary 7.6]{Wir22}. This condition means that for any element $f$ of $\dom A_2^{\frac{1}{2}}$ the <<carr\'e du champ>> $\Gamma(f,f)$ belongs to the space $\L^1(\cal{M})$. By \cite[Theorem 7.12]{Wir22}, this condition is equivalent to say that the tangent bimodule is normal.

We note the following straightforward result, satisfied in this context.

\begin{lemma}
The homomorphism $\pi \co \L^\infty(\cal{M}) \to \B(\L^2(\cal{M}) \oplus_2 \cal{H})$ is weak* continuous.
\end{lemma}

\begin{proof}
Let $(f_j)$ be a \textit{bounded} net of $\L^\infty(\cal{M})$ converging in the weak* topology to $f$. It is obvious that the net $(M_{f_j})$ is bounded. If $g,h \in \L^2(\cal{M})$, we have $\langle M_{f_j}(g),h \rangle_{\L^2(\cal{M})}=\tau((f_jg)^*h) = \tau(hg^*f_j^*)\xra[j]{} \tau(hg^*f^*)=\langle M_f(g),h \rangle_{\L^2(\cal{M})}$ since $hg^* \in \L^1(\cal{M})$. So $(M_{f_j})$ converges to $M_{f}$ in the weak operator topology. By \cite[Lemma 2.5 p.~69]{Tak02}, the weak operator topology and the weak* topology of $\B(\L^2(\cal{M}))$ coincide on bounded sets. We conclude that $(M_{f_j})$ converges to $M_{f}$ in the weak* topology. Since the tangent bimodule is normal, the map $\Phi$ is weak* continuous. So the net $(\Phi_{f_j})$ converges to $\Phi_{f}$ in the weak* topology. 
By \cite[Theorem A.2.5 (2) p.~360]{BLM04}, we conclude that $\pi$ is weak* continuous.    
\end{proof}

In this context, the Hilbert space of the tangent bimodule can be realized as a subspace of a noncommutative $\L^2$-space $\L^2(\tilde{\cal{M}})$ of a bigger finite von Neumann algebra $\tilde{\cal{M}}$ equipped with a normal finite faithful trace, i.e.~we can take $\cal{H}=\L^2(\tilde{\cal{M}})$. More precisely, we have a trace preserving embedding $J \co \cal{M} \to \tilde{\cal{M}}$. Moreover, in this case, the maps $\Phi$ and $\J$ are given by $\Phi_f \co \L^2(\tilde{\cal{M}}) \to \L^2(\tilde{\cal{M}})$, $g \mapsto fg$ for any $f \in \L^\infty(\cal{M})$ and $\J \co \L^2(\tilde{\cal{M}}) \to \L^2(\tilde{\cal{M}})$, $x \mapsto x^*$. Furthermore, for any $f \in \mathcal{B}$ we have $\partial f \in \tilde{\mathcal{M}}$ by construction, see \cite{JRS}. We denote by $\E \co \tilde{\cal{M}} \to \cal{M}$ the canonical trace preserving normal faithful conditional expectation associated with $J$ and by $J_2 \co \L^2(\cal{M}) \to \L^2(\tilde{\cal{M}})$ and $\E_2 \co \L^2(\tilde{\cal{M}}) \to \L^2(\cal{M})$ the induced maps on the associated $\L^2$-noncommutative spaces.

The next result says in particular that the second condition of Definition \ref{def-spec} is automatically satisfied, since the unital $*$-subalgebra $\mathcal{B}$ is a weak* dense in the von Neumann algebra $\cal{M}$.

\begin{prop}
\label{Rem-amal}
In the previous situation, we have the inclusion
\begin{equation}
\label{Lip-algebra-I-Schur}
\cal{B}
\subset \Lip_{D}(\L^\infty(\cal{M}))
\end{equation}
where the latter algebra is defined in \eqref{Lipschitz-algebra-def}. Moreover, for any $f \in \cal{B}$, we have
\begin{equation}
\label{commutator-D}
[D,f]
=\begin{bmatrix}
    0 &  -\E_2{\M}_{\partial f} \\
   \M_{\partial f}J_2  & 0  \\
\end{bmatrix},
\end{equation}
where $\M_{\partial(f)} \co \L^2(\tilde{\cal{M})} \to \L^2(\tilde{\cal{M})}$ is the multiplication operator by $\partial_{}(f)$, and finally
\begin{equation}
\label{norm-commutator}
\bnorm{[D,f]}_{\L^2(\cal{M}) \oplus_2 \L^2(\tilde{\cal{M}}) \to \L^2(\cal{M}) \oplus_2 \L^2(\tilde{\cal{M}})}
= \max\big\{ \norm{\partial f}_{\L^{2}_{\infty, \ell}(\cal{M} \subset \tilde{\cal{M}})}, \norm{\partial f^*}_{\L^{2}_{\infty, \ell}(\cal{M} \subset \tilde{\cal{M}})} \big\}.
\end{equation}
\end{prop}

\begin{proof} 
Let $f \in \cal{B}$. A standard calculation shows that
\begin{align*}
\MoveEqLeft
\big[D,f\big]
\ov{\eqref{Hodge-Dirac-I}\eqref{Def-pi-a}}{=}\begin{bmatrix} 
0 & \partial^* \\ 
\partial & 0
\end{bmatrix}\begin{bmatrix}
    \M_f & 0  \\
    0 & \Phi_f  \\
\end{bmatrix}
-\begin{bmatrix}
    \M_{f} & 0  \\
    0 & \Phi_{f}  \\
\end{bmatrix}
\begin{bmatrix} 
0 & \partial^* \\ 
\partial & 0
\end{bmatrix}\\
&=\begin{bmatrix}
   0 &  \partial^* \Phi_f \\
   \partial \M_f  &0  \\
\end{bmatrix}-\begin{bmatrix}
    0 &  \M_{f} \partial^*\\
    \Phi_{f}\partial & 0  \\
\end{bmatrix} 
=\begin{bmatrix}
    0 &  \partial^*\Phi_f-\M_{f} \partial^* \\
   \partial\M_f-\Phi_{f}\partial  & 0  \\
\end{bmatrix}.
\end{align*}
Now, we calculate the two non-zero entries of the commutator. For the lower left corner, if $g \in \cal{B}$ we have
\begin{align}
\label{Bon-commutateur}
\MoveEqLeft
(\partial\M_f-\Phi_{f}\partial)(g)           
=\partial\M_f(g)-\Phi_{f}\partial g
=\partial(fg)-f\partial(g) 
\ov{\eqref{Leibniz}}{=} \partial(f) g
=\M_{\partial f}J_2(g).
\end{align}
For the upper right corner, note that for any $h \in \dom \partial^*$ and any $f,g \in \cal{B}$,  we have
\begin{align*}
\MoveEqLeft
\la \partial g,fh \ra_{\L^2(\tilde{\cal{M}})}         
=\tau(\partial(g)^* fh ) 
=\tau(\partial(g^*) f h )\\
&\ov{\eqref{Leibniz}}{=} \tau(\partial(g^*f) h-g^*\partial(f)h )
=\tau(\partial(g^*f) h)-\tau(g^*\partial(f)h ) \nonumber \\
&=\tau(g^*f \partial^*(h))-\tau(g^*\partial(f)h ) 
=\tau(g^*(f \partial^*(h)-\partial(f)h ))
=\big\la g, f \partial^*(h)-\partial(f)h \big\ra_{\L^2(\cal{M})}. \nonumber
\end{align*}
Since $\cal{B}$ is a core of $\partial$, it is easy to see that the same relation is true for $g \in \dom \partial$. Hence $f h$ belongs to $\dom \partial^*$ and we have 
\begin{equation}
\label{divers-66}
\partial^*(f h)
=f \partial^*(h)-\partial(f)h.
\end{equation}
If $f \in \mathcal{B}$, $g \in \L^2(\cal{M})$ and $h \in \dom \partial^*$, we infer that
\begin{align*}
\MoveEqLeft
\big\langle \big(\partial^*\Phi_f-\M_f \partial^*\big)(h),g \big\rangle_{\L^2(\cal{M})}
=\big\langle \partial^*\Phi_f(h),g \big\rangle -\big\langle \M_f \partial^*(h),g \big\rangle_{}
=\big\langle \partial^*(fh),g \big\rangle -\big\langle f\partial^*(h),g \big\rangle_{} \\
&\ov{\eqref{divers-66}}{=} \big\langle f \partial^*(h)-\partial(f)h,g \big\rangle -\big\langle f\partial^*(h),g \big\rangle_{} 
=-\big\langle \partial(f)h,g \big\rangle 
= - \big\langle h,\partial_{}(f^*)g \big\rangle_{\L^2(\tilde{\cal{M})}} \\
&=-\big\langle h, \M_{(\partial_{}(f))^*}J_2(g) \big\rangle_{\L^2(\tilde{\cal{M})}} 
=-\big\langle \M_{\partial_{}(f)}(h), J_2(g)\big\rangle_{\L^2(\tilde{\cal{M})}} 
=-\big\langle \E_2{\M}_{\partial_{}(f)}(h),g\big\rangle_{\L^2(\cal{M})}.
\end{align*}
We conclude that
\begin{equation*}
\label{Commutateur-etrange}
\big(\partial^*\Phi_f-\M_f \partial^*\big)(h)
=-\E_2{\M}_{\partial f}(h).
\end{equation*}
Consequently, we obtained formula \eqref{commutator-D}. Since $\partial(f)$ belongs to the von Neumann algebra $\tilde{\mathcal{M}}$, it is now apparent that $\cal{B}$ is a subset of the Lipschitz algebra $\Lip_{D}(\L^\infty(\cal{M}))$. Furthermore,  we have 
\begin{align*}
\MoveEqLeft
\norm{\M_{\partial f}J_2}_{\L^2(\cal{M}) \to \L^2(\tilde{\cal{M}})}
=\sup_{\norm{g}_{\L^2(\cal{M})}=1}\norm{\M_{\partial f}J_2(g)}_{\L^2(\tilde{\cal{M}})} \\
&=\sup_{\norm{g}_{\L^2(\cal{M})}=1} \bnorm{\partial(f) g}_{\L^2(\tilde{\cal{M}})} 
\ov{\eqref{conditional-one-sided}}{=} \norm{\partial f}_{\L^{2}_{\infty, \ell}(\cal{M} \subset \tilde{\cal{M}})}.
\end{align*}
So we have proved, with an argument of duality for the second non-null entry of the commutator, the equality \eqref{norm-commutator}.
\end{proof}

\begin{remark} \normalfont
We probably have $\Lip_{D}(\L^\infty(\cal{M}))=\dom \partial$. The argument should be quite elementary. Moreover, the equality $\cal{B}= \cal{M} \cap \dom \partial$ seems to be essentially proven in \cite[Proposition 3.4]{DaL92}.
%
\end{remark}

\begin{example}
\normalfont
With the previous remark, we can recover the norm of \cite[Proposition 6.4 3.]{Arh24} since 
$$
\L^{2}_{\infty}(\L^\infty(G) \subset \L^\infty(G,\ell^\infty_m))
=\L^{2}_{\infty}(\L^\infty(G) \subset \L^\infty(G) \otvn \ell^\infty_m)
\ov{\eqref{mixed-particular}}{=}\L^\infty(G,\ell^2_m).
$$
\end{example}




\section{Local Coulhon-Varopoulos dimension and spectral dimension}
\label{sec-cdimension}

\paragraph{Noncommutative Dunford-Pettis theorem} We begin by explaining the noncommutative Dunford-Pettis theorem which, roughly speaking, states that completely bounded maps from a noncommutative $\L^1$-space into a noncommutative $\L^\infty$-space are integral operators with bounded kernels. Let $\cal{M}$ and $\cal{N}$ be von Neumann algebras. By \cite[Theorem 2.5.2 p.~49]{Pis03}, we have a complete isometry
$$
\label{}
\CB(\cal{M}_*,\cal{N})
=\cal{N} \otvn \cal{M} 
$$
where $\cal{M}_*$ is the predual of $\cal{M}$. This classical result was proved by Effros and Ruan in \cite[Theorem 3.2]{EfR90}. Now, let us assume that the von Neumann algebra $\cal{M}$ is finite and equipped with a normal finite faithful trace $\tau$. If $\cal{M}^{\op}$ denotes the opposite algebra, the isometric map $\L^1(\cal{M})\to \cal{M^\op_*}$, $f \mapsto \tau(f\,\cdot)$ allows us to define an operator space structure on the Banach space $\L^1(\cal{M})$, i.e.~this map becomes completely isometric. So we obtain an identification
\begin{equation}
\label{Effros-Ruan}
\CB(\L^1(\cal{M}),\L^\infty(\cal{M}))
=\L^\infty(\cal{M}) \otvn \raisebox{-0.3pt}{$\L^\infty(\cal{M}^\op)$}.
\end{equation}
More precisely, we associate to a <<kernel>> $K \in \L^\infty(\cal{M}) \otvn \raisebox{-0.1ex}{$\L^\infty(\cal{M}^\op)$}$ the completely bounded map $T_K \co \L^1(\cal{M}) \to \L^\infty(\cal{M})$ defined by 
\begin{equation}
\label{Kernel-operator}
T_K(f) 
\ov{\mathrm{def}}{=} (\Id \ot \tau)(K(1 \ot f)), \quad f \in \L^1(\cal{M}) 
\end{equation}
which satisfies 
\begin{equation}
\label{NC-Dunford-Pettis-1}
\norm{K}_{\L^\infty(\cal{M}) \otvn \L^\infty(\cal{M}^\op)}
=\norm{T_K}_{\cb,\L^1(\cal{M}) \to \L^\infty(\cal{M})}.
\end{equation}
In particular, if $\cal{M}$ is a commutative von Neumann algebra $\L^\infty(\Omega)$ for a $\sigma$-finite measure space $\Omega$, we recover with \eqref{min-et-cb} the classical Dunford-Pettis theorem \cite[p.~528]{Rob91} \cite[Theorem 1.3]{ArB94} \cite[Proposition 3.1]{GrH14} and we obtain the equality
\begin{equation}
\label{Dunford-Pettis-1}
\norm{K}_{\L^\infty(\Omega \times \Omega)}
=\norm{T_K}_{\L^1(\Omega) \to \L^\infty(\Omega)}.
\end{equation}
In this case, the operator $T_K$ is defined by $(T_Kf)(x) =\int_\Omega K(x,y) f(y) \d y$ 
for any $f \in \L^1(\Omega)$ and almost all $x \in \Omega$. 

For any $1 \leq p <\infty$, we also have a similar identification for any bounded operator $T \co \L^p(\Omega) \to \L^\infty(\Omega)$, see \cite[p.~528]{Rob91}, \cite[Theorem 1.3]{ArB94} and \cite[Lemma 3.3]{GrH14}. In this case, the equality \eqref{Dunford-Pettis-1} is replaced by
\begin{equation}
\label{Dunford-Pettis-2}
\norm{T}_{\L^p(\Omega) \to \L^\infty(\Omega)}
=\norm{K}_{\L^\infty(\Omega,\L^{p^*}(\Omega))}.
\end{equation}

\paragraph{Hilbert-Schmidt operators} 
It is well-known that Hilbert-Schmidt operators on classical $\L^2$-spaces identify to integral operators with square-integrable kernels, see e.g.~\cite[Proposition 3.4.16 p.~122]{Ped89} and \cite[Exercise 2.8.38 (ii) p.~170]{KaR97b}. We need a noncommutative generalization of this result. By the way, we refer to the paper \cite{Neu90} for a study of integral operators between noncommutative $\L^p$-spaces. Let us assume that $\cal{M}$ and $\cal{N}$ are finite von Neumann algebras equipped with normal finite faithful traces. We will use the scalar product $\la g,f \ra_{\L^2(\cal{M})} \ov{\mathrm{def}}{=}\tau(g^* f)$ where $f,g \in \L^2(\cal{M})$. Using the notation $\ovl{E}$ for the complex conjugate of the Banach space (or operator space) $E$ (see \cite[p.~63]{Pis03}), recall that we have a completely isometric linear isomorphism $\ovl{\L^2(\cal{N})} \to \L^2(\cal{N}^\op)$ induced by the linear complete isometry $\ovl{\cal{N}} \to \cal{N}^\op$, $x \mapsto x^*$. Note that we have an isometry $\L^2(\cal{M}) \ot_2 \ovl{\L^2(\cal{N})}=\L^2(\cal{M} \otvn \cal{N}^\op)$. Moreover, we have $\la g,f \ra_{\ovl{\L^2(\cal{N})}} = \ovl{\la g,f \ra_{\L^2(\cal{N})}}$ for any $f,g \in \L^2(\cal{N})$.

\begin{prop}
\label{prop-HS}
Let $\cal{M}$ and $\cal{N}$ be finite von Neumann algebras equipped with normal finite faithful traces denoted by the same notation $\tau$. For any element $K \in \L^2(\cal{M}) \ot_2\raisebox{-0.1ex}{$\overline{\L^2(\cal{N})}$}$, there exists a unique linear map $T_K \co \L^2(\cal{N}) \to \L^2(\cal{M})$ such that
\begin{equation}
\label{def-HS}
\la g,T_K(f) \ra_{\L^2(\cal{M})}  
=(\tau \ot \tau)(K (g^* \ot f)), \quad f \in \L^2(\cal{N}),g \in \L^2(\cal{M}).
\end{equation}
Moreover, this map is a Hilbert-Schmidt operator and the linear map $\L^2(\cal{M}) \ot_2\raisebox{-0.1ex}{$\overline{\L^2(\cal{N})}$} \to S^2(\L^2(\cal{N}),\L^2(\cal{M}))$, $K \mapsto T_K$ is a surjective isometry:
\begin{equation}
\label{NC-HS-1}
\norm{T_K}_{S^2(\L^2(\cal{N}),\L^2(\cal{M}))}
=\norm{K}_{\L^2(\cal{M}) \ot_2\overline{\L^2(\cal{N})}}.
\end{equation}
\end{prop}

\begin{proof}
The existence and the uniqueness of $T_K$ are clear. Consider orthonormal bases $(e_i)_{i \in I}$  and $(f_j)_{j \in J}$ of the Hilbert spaces $\L^2(\cal{M})$ and $\L^2(\cal{N})$. For any element $K \in \L^2(\cal{M}) \ot_2\raisebox{-0.1ex}{$\overline{\L^2(\cal{N})}$}$, we have with \cite[(2.27) p.~262]{Kat76}
$$
\norm{T_K}_{S^2(\L^2(\cal{N}),\L^2(\cal{M}))}
=\bigg(\sum_{i \in I} \norm{T_K(f_j)}_{\L^2(\cal{M})}^2  \bigg)^{\frac{1}{2}}
=\bigg(\sum_{(i,j) \in I \times J} \big|\la e_i,T_K(f_j) \ra_{\L^2(\cal{M})} \big|^2 \bigg)^{\frac{1}{2}}
$$
with
\begin{align*}
\MoveEqLeft
\big\la e_i,T_K (f_j) \big\ra_{\L^2(\cal{M})}        
\ov{\eqref{def-HS}}{=} (\tau \ot \tau)(K (e_i^* \ot f_j)) \\
&=\la e_i \ot f_j^*,K \ra_{\L^2(\cal{M}) \ot_2 \L^2(\cal{N})}
=\ovl{\la K, e_i \ot f_j^* \ra_{\L^2(\cal{M}) \ot_2 \ovl{\L^2(\cal{N})}}}.         
\end{align*}
Since $(e_i \ot f_j^*)$ is a basis of the Hilbert space $\L^2(\cal{M}) \ot_2\raisebox{-0.1ex}{$\overline{\L^2(\cal{N})}$}$, we have 
$$
\norm{K}_{\L^2(\cal{M}) \ot_2 \ovl{\L^2(\cal{N})}}
=\bigg(\sum_{(i,j) \in I \times J}\big|\la K, e_i \ot f_j^* \ra_{\L^2(\cal{M}) \ot_2 \ovl{\L^2(\cal{N})}}\big|^2\bigg)^{\frac{1}{2}}.
$$
We deduce that the map $K \to T_K$ is an isometry from the Hilbert space $\L^2(\cal{M}) \ot_2 \raisebox{-0.1ex}{$\overline{\L^2(\cal{N})}$}$ into the space $S^2(\L^2(\cal{N}),\L^2(\cal{M}))$ of all Hilbert-Schmidt operators. Its range is a closed subspace of the Banach space $S^2(\L^2(\cal{N}),\L^2(\cal{M}))$. If $K=h \ot k$ is an element of the Hilbert space $\L^2(\cal{M}) \ot_2\raisebox{-0.1ex}{$\overline{\L^2(\cal{N})}$}$, we have for any $g \in \L^2(\cal{M})$ and any $f \in \L^2(\cal{N})$
$$
\la g,T_K(f) \ra_{\L^2(\cal{M})} 
\ov{\eqref{def-HS}}{=} (\tau \ot \tau)(K (g^* \ot f))
=(\tau \ot \tau)\big((h \ot k) (g^* \ot f)\big)
=\tau(h g^*)  \tau(k f).
$$
Moreover, we have $\la  g,\tau(k f)h\ra_{\L^2(\cal{M})}
=\tau(g^*h)\tau(k f)$. Hence the finite-rank operator $f \mapsto \tau(k f)h$ belongs to the range of the map $K \mapsto T_K$. So the range contains all finite sums of these operators. This means that it contains the set of all finite-rank operators. We conclude that this map is surjective.
\end{proof}


\begin{remark} \normalfont
Consider two Hilbert spaces, $\cal{H}$ and $\cal{K}$, each equipped with an inner product that is antilinear in the first argument and linear in the second. For any $k \in \cal{K}$ and any $h \in \cal{H}$, we may identify the tensor $k \ot h$ with the finite-rank operator $\cal{H} \to \cal{K}$, $\xi \mapsto  \la h,\xi \ra k$. By \cite[Proposition 2.6.9 p.~142]{KaR97a}, this identification induces an isometric isomorphism $\cal{K} \ot_2 \ovl{\cal{H}}=S^2(\cal{H},\cal{K})$. Proposition \ref{prop-HS} does not seem to be a direct consequence of the case $\cal{H}=\L^2(\cal{M})$ and $\cal{K}=\L^2(\cal{N})$. 
\end{remark}





We will use the following simple observation obtained from the previous result and some obvious inclusions. 

\begin{lemma}
\label{Lemma-Schmidt}
Let $\cal{M}$ be a von Neumann algebra equipped with a normal finite faithful trace. For each element $K$ of $\L^\infty(\cal{M}) \otvn \L^\infty(\cal{M}^\op)$ (or $\L^\infty(\cal{M}) \otvn \L^2(\cal{M})^\op$) the linear map $T_K \co \L^1(\cal{M}) \to \L^\infty(\cal{M})$ (or $T_K \co \L^2(\cal{M}) \to \L^\infty(\cal{M}))$ induces a Hilbert-Schmidt operator $T_K \co \L^2(\cal{M}) \to \L^2(\cal{M})$ (both maps coincide on $\cal{M}$).
\end{lemma}



Following \cite[Definition 4.1 p.~99]{Zag19}, we say that a strongly continuous semigroup $(T_t)_{t \geq 0}$ of operators acting on a Hilbert space $H$ is a Gibbs semigroup if the operator $T_t \co H \to H$ is trace-class for any $t>0$.

\begin{lemma}
\label{Lemma-Gibbs}
Let $\cal{M}$ be a finite von Neumann algebra equipped with a normal finite faithful trace. Consider a symmetric sub-Markovian semigroup $(T_t)_{t \geq 0}$ of operators acting on $\cal{M}$. Suppose the estimate $
\norm{T_t}_{\cb,\L^1(\cal{M}) \to \L^\infty(\cal{M})} \lesssim \frac{1}{t^{\frac{d}{2}}}$ for some $d > 0$ and any $0 < t \leq 1$. Then the semigroup induces a Gibbs semigroup $(T_{t})_{t \geq 0}$ on the Hilbert space $\L^2(\cal{M})$. Moreover, its generator $-A_2$ has compact resolvent and its spectrum consists entirely of isolated eigenvalues with finite multiplicities. Finally, the fixed-point subalgebra $\{ x \in \cal{M} : T_t(x)=x \text{ for any } t \geq 0 \}$ is finite-dimensional.
\end{lemma}

\begin{proof}
For any $t \geq 0$, we have a contractive operator $T_{t} \co \L^2(\cal{M}) \to \L^2(\cal{M})$. By the identification \eqref{NC-Dunford-Pettis-1} and the assumption \eqref{Varo-Coulhon-dim-cb}, each operator $T_t \co \L^1(\cal{M}) \to \L^\infty(\cal{M})$ admits a kernel $K_t$ belonging to the von Neumann algebra $\L^\infty(\cal{M}) \otvn \raisebox{-0.3pt}{$\L^\infty(\cal{M}^\op)$}$. Lemma \ref{Lemma-Schmidt} says that for all $0 < t \leq 1$, we have a Hilbert-Schmidt operator $T_{t} \co \L^2(\cal{M}) \to \L^2(\cal{M})$. From \cite[Lemma 4.22 p.~117]{EnN00}, we deduce that $T_{t} \co \L^2(\cal{M}) \to \L^2(\cal{M})$ is compact for \textit{any} $t>0$. By \cite[Theorem 4.29 p.~119]{EnN00}, we infer that the generator $-A_2$ has compact resolvent. Consequently, by \cite[Theorem 6.29 p.~187]{Kat76} its spectrum consists entirely of isolated eigenvalues with finite multiplicities. 

Moreover, for any $0 < t \leq 1$, we can write $T_{t}=(T_{\frac{t}{2}})^2$. So by \cite[Corollary 3.16 p.~90]{Zag19} the operator $T_{t} \co \L^2(\cal{M}) \to \L^2(\cal{M})$ is trace-class for any $0 < t \leq 1$. If $t>1$, the equality $T_{t}=T_{t-1}T_{1}$ shows that $T_{t}$ is also trace-class if $t>1$ by the ideal property of the space of trace-class operators.

By Proposition \ref{Prop-mean}, we infer that the fixed subspace $\Fix (T_{t})_{t \geq 0}$  is finite-dimensional, where we see the operators of the semigroup as operators on $\L^2(\cal{M})$. Recall that $\Fix (T_{t})_{t \geq 0}$ is the noncommutative $\L^2$-space associated to the von Neumann algebra $\{ x \in \cal{M} : T_t(x)=x \text{ for any } t \geq 0 \}$ and the induced trace. By finite-dimensionality, the space $\{ x \in \cal{M} : T_t(x)=x \text{ for any } t \geq 0 \}$ is equal to the space $\Fix (T_{t})_{t \geq 0}$ and consequently this algebra is finite-dimensional.
\end{proof}


Observe the assumption of injectivity on the von Neumann algebra $\cal{M}$ in the following result and note Proposition \ref{prop-Junge-square}.


\begin{prop}
\label{prop-ine-major}
Let $\cal{M}$ be an injective finite von Neumann algebra equipped with a normal finite faithful trace. Consider a $*$-preserving selfadjoint operator $T \co \cal{M} \to \cal{M}$. Assume that it induces a completely bounded map $T \co \L^{2}(\cal{M}) \to \L^\infty(\cal{M})$. Then this operator induces a trace-class operator $T^2 \co \L^2(\cal{M}) \to \L^2(\cal{M})$ and we have
\begin{equation}
\label{magic-estimates-5}
\norm{T^2}_{S^1(\L^2(\cal{M}))}
\lesssim \norm{T^2}_{\cb,\L^1(\cal{M}) \to \L^{\infty}(\cal{M})}.
\end{equation}
\end{prop}

\begin{proof} 
Observe that $\L^2(\cal{M})^*=\ovl{\L^2(\cal{M})}=\L^2(\cal{M}^\op)$ completely isometrically. So, by the identification described in \eqref{normal-minimal} (since $\cal{M}$ is injective), there exists an element $K \in \L^\infty(\cal{M}) \otvn \raisebox{-0.1ex}{$\L^2(\cal{M}^\op)$}$ such that $T=T_{K}$ with
\begin{equation}
\label{equa-infini}
\norm{T}_{\cb,\L^2(\cal{M}) \to \L^\infty(\cal{M})}
=\norm{K}_{\L^\infty(\cal{M}) \otvn \raisebox{-0.1ex}{$\scriptstyle\L^2(\cal{M}^\op)$}}.
\end{equation}
By Lemma \ref{Lemma-Schmidt}, we obtain a Hilbert-Schmidt operator $T \co \L^2(\cal{M}) \to \L^2(\cal{M})$. Hence, the operator $T^2 \co \L^2(\cal{M}) \to \L^2(\cal{M})$ is trace-class. Using again the injectivity of the von Neumann algebra $\cal{M}$, we see that 
\begin{align}
\MoveEqLeft
\label{interpol-formula-4}
\L^2(\cal{M} \otvn \cal{M}^\op)
=\L^2(\cal{M},\L^2(\cal{M}^\op))
\ov{\eqref{Lp-vect}}{=} \big(\L^\infty(\cal{M},\L^2(\cal{M}^\op)), \L^{1}(\cal{M},\L^{2}(\cal{M}^\op))\big)_{\frac{1}{2}}.
\end{align}
Now, by interpolation we obtain 
\begin{align*}
\MoveEqLeft
\bnorm{T^2}_{S^1(\L^2(\cal{M}))}          
= \norm{T}_{S^2(\L^2(\cal{M}))}^2
\ov{\eqref{NC-HS-1}}{=} \norm{K}_{\L^2(\cal{M}) \ot_2 \L^2(\cal{M}^\op)}^2 \\
&=\norm{K}_{\L^2(\cal{M} \otvn \cal{M}^\op)}^2 
\ov{\eqref{inclusions-interpolation}}{\lesssim} \norm{K}_{\L^\infty(\cal{M},\L^2(\cal{M}^\op))}^2
\ov{\eqref{def-Lp-vec-1}}{=} \norm{K}_{\L^\infty(\cal{M}) \otvn \L^2(\cal{M}^\op)}^2 \\
&\ov{\eqref{equa-infini}}{=} \norm{T}_{\cb,\L^2(\cal{M}) \to \L^\infty(\cal{M})}^2 
\ov{\eqref{formule-square}}{=} \bnorm{T^2}_{\cb,\L^{1}(\cal{M}) \to \L^\infty(\cal{M})}
\end{align*}
(in the commutative case, we can use \eqref{Dunford-Pettis-2} instead of \eqref{def-Lp-vec-1} and \eqref{equa-infini}). 
\end{proof}

Now, we can obtain the following result. 

\begin{cor}
\label{cor-trace-class}
Let $\cal{M}$ be a finite von Neumann algebra equipped with a normal finite faithful trace. Consider a symmetric sub-Markovian semigroup $(T_t)_{t \geq 0}$ of operators acting on $\cal{M}$. Assume the estimate $
\norm{T_t}_{\cb,\L^1(\cal{M}) \to \L^\infty(\cal{M})} 
\lesssim \frac{1}{t^{\frac{d}{2}}}$ for some $d > 0$ and any $0 < t \leq 1$. Then the operator $|D|^{-1}$ belongs to the weak Schatten space $S^{d,\infty}(\L^2(\cal{M}) \oplus_2 \cal{H})$ and we have
\begin{equation}
\label{magic-estimates-3-bis}
\norm{T_{t}}_{S^1(\L^2(\cal{M}))}
\lesssim \frac{1}{t^{\frac{d}{2}}}, \quad 0 < t \leq 1.
\end{equation}
\end{cor}

\begin{proof}
With Proposition \ref{prop-ine-major} and Proposition \ref{prop-Junge-square}, we obtain the estimate \eqref{magic-estimates-3-bis}. 
Observe that the subspace $\L_0^2(\cal{M})=\ovl{\Ran A_2}$ is stable under each operator $T_t$ by \eqref{A-et-Tt-commute}. Hence, we clearly have by restriction
\begin{equation}
\label{magic-estimates-3}
\norm{T_{t}}_{S^1(\L^2_0(\cal{M}))}
\lesssim \frac{1}{t^{\frac{d}{2}}}, \quad 0 < t \leq 1.
\end{equation}
Note that by Lemma \ref{Lemma-Gibbs}, the positive selfadjoint operator $A_2$ has compact resolvent. So we can use Proposition \ref{th-Karamata-Iochum} since we can suppose that the sequence of eigenvalues is infinite. This implies that the eigenvalues of the operator satisfy
$$
\lambda_n\big(A_2|_{\ovl{\Ran A_2}}\big)^{-1}
=O_{n \to \infty}\bigg(\frac{1}{n^{\frac{2}{d}}}\bigg).
$$ 
On the Hilbert space $(\ker \partial^* \partial)^\perp \oplus_2 (\ker \partial \partial^*)^\perp$, we have
\begin{align}
\label{D-1}
\MoveEqLeft
|D|^{-1}
=(D^2)^{-\frac{1}{2}}
\ov{\eqref{carre-de-D}}{=} \begin{bmatrix} 
\partial^* \partial & 0 \\ 
0 & \partial \partial^*
\end{bmatrix}^{-\frac{1}{2}}
=\begin{bmatrix} 
\big(\partial^* \partial\big)^{-\frac{1}{2}} & 0 \\ 
0 & \big(\partial \partial^*\big)^{-\frac{1}{2}}
\end{bmatrix}.            
\end{align} 
By Theorem \ref{Th-unit-eq}, we know that the operators $\partial^* \partial |_{(\ker \partial)^\perp}$ and $\partial \partial^*|_{(\ker \partial^*)^\perp}$ are unitarily equivalent. Moreover, we have
$$
(\ker \partial)^\perp
\ov{\eqref{lien-ker-image}}{=} \ovl{\Ran \partial^*}
\ov{\eqref{Kadison-image}}{=}
\ovl{\Ran \partial^* \partial}
\ov{\eqref{lien-ker-image}}{=} (\ker \partial^* \partial)^\perp
=(\ker A_2)^\perp
=\ovl{\Ran A_2}
$$
and
$$
(\ker \partial \partial^*)^\perp
\ov{\eqref{Kadison-image}}{=} (\ker \partial^*)^\perp.
$$
Consequently $A_2^{-\frac{1}{2}}=(\partial^* \partial)^{-\frac{1}{2}}|_{|_{\ovl{\Ran A_2}}}$ and $(\partial \partial^*)^{-\frac{1}{2}}|_{(\ker \partial \partial^*)^\perp}$ are unitarily equivalent. We conclude that $\lambda_n(|D|^{-1})
=O_{n \to \infty}\big(\frac{1}{n^{\frac{1}{d}}}\big)$. By \eqref{weak-Schatten}, this means that the Hodge-Dirac operator $D$ belongs to the weak Schatten space $S^{d,\infty}(\L^2(\cal{M}) \oplus_2 \cal{H})$.
\end{proof}

With the second inclusion of \eqref{inclusions-weak-Schatten}, Corollary \ref{cor-trace-class}  immediately gives the following result.

\begin{cor}
\label{th-spectral-coulhon}
Let $\cal{M}$ be a finite von Neumann algebra equipped with a normal finite faithful trace. Consider a symmetric sub-Markovian semigroup $(T_t)_{t \geq 0}$ of operators acting on $\cal{M}$. Suppose that $\norm{T_t}_{\cb,\L^1(\cal{M}) \to \L^\infty(\cal{M})} \lesssim \frac{1}{t^{\frac{d}{2}}}$  for some $d > 0$ and any $0 < t \leq 1$. Then for any $p > d$, the operator $|D|^{-p}$ is trace-class, i.e.~$|D|^{-1}$ belongs to the Schatten class $S^p(\L^2(\cal{M}) \oplus_2 \cal{H})$. 
\end{cor}



Note that the knowledge of the small-time asymptotic of the trace of the operators of the semigroup gives the spectral dimension by the following result.

\begin{prop}
\label{prop-trace-spectral triple}
Let $\cal{M}$ be a finite von Neumann algebra equipped with a normal finite faithful trace. Consider a symmetric sub-Markovian semigroup $(T_t)_{t \geq 0}$ of operators acting on $\cal{M}$. Assume that the spectrum of $A_2$ consists entirely of isolated eigenvalues with finite multiplicities. If for some $d >0$ and some constant $c > 0$ we have $\tr T_t \underset{0}{\sim} \frac{c}{t^{\frac{d}{2}}}$. The spectral dimension of the Dirac operator $D$ defined in \eqref{Hodge-Dirac-I} is equal to $d$.
\end{prop}

\begin{proof}
We denote by $(\lambda_n)$ the sequence of eigenvalues of the operator $A_2$, which can be assumed to be infinite. We have
$
\sum_{n=1}^{\infty} \e^{-\lambda_n t}
=\tr T_t
\underset{0}{\sim}  \frac{c}{t^\frac{d}{2}}$.  
By Karamata's theorem \ref{thm-Karamata}, we deduce that for some constant $c'$ we have $
\lambda_n 
\sim c'n^{\frac{2}{d}}$. Similarly to the end of the proof of Theorem \ref{cor-trace-class}, we see that the non-zero elements of the spectrum of the operator $|D|$ are the $\lambda_n^{\frac{1}{2}}$'s. Consequently, the ones of the operator $|D|^{-1}$ are the $\lambda_n^{-\frac{1}{2}}$'s. Moreover, we have $\lambda_n^{-\frac{1}{2}} \sim c''n^{-\frac{1}{d}}$ for some positive constant $c''$. The proof is complete.
\end{proof}


\section{Link with the completely bounded Hardy-Littlewood-Sobolev theory}
\label{Sec-Hardy}


In this section, we give a second proof of Corollary \ref{th-spectral-coulhon} in proving some result connected to the completely bounded Hardy-Littlewood-Sobolev theory. The following result is announced in \cite[Corollary 4.31 p.~69]{Zha18} and established with a complicated proof. The commutative case is a classical result, see \cite[Theorem II.2.7 p.~12]{VSCC92} and \cite[Theorem 9.4.5 p.~173]{App19}. We refer also to \cite[Corollary 1.1.4 p.~622]{JuM10} for a noncommutative variant\footnote{\thefootnote. Note that this result requires a slightly weaker assumption $\norm{T_t}_{\cb, \L^p_0(\cal{M}) \to \L^q_0(\cal{M})} \lesssim \frac{1}{t^{\frac{d}{2}(\frac{1}{p}-\frac{1}{q})}}$ for any $t > 0$.}.

\begin{thm}
\label{th-Zhao}
Let $\cal{M}$ be a semifinite von Neumann algebra equipped with a normal semifinite faithful trace. Suppose that $1 < p < q < \infty$. Consider a symmetric sub-Markovian semigroup $(T_t)_{t \geq 0}$ of operators acting on $\cal{M}$ such that $\norm{T_t}_{\cb, \L^p(\cal{M}) \to \L^q(\cal{M})} \lesssim \frac{1}{t^{\frac{d}{2}(\frac{1}{p}-\frac{1}{q})}}$ for any $t > 0$. If $\alpha \ov{\mathrm{def}}{=}\frac{d}{2}(\frac{1}{p}-\frac{1}{q})$ this implies the complete boundedness of the operator $A^{-\alpha} \co \L_0^p(\cal{M}) \to \L^q_0(\cal{M})$. 	
\end{thm}

In the \textit{semifinite} case, we refer to \cite{PiX03} for the definition of noncommutative $\L^p$-spaces generalizing very slightly \eqref{norm-Lp}. Note that for any $1 < p < \infty$ the Banach space $\L_0^p(\cal{M})$ is defined by $\L_0^p(\cal{M}) \ov{\mathrm{def}}{=} \ovl{\Ran A_p}$ where $-A_p$ is the generator of the strongly continuous semigroup $(T_{t})_{t \geq 0}$ on the Banach space $\L^p(\cal{M})$. We still have $\L^p(\cal{M}) \ov{\eqref{decompo-reflexive}}{=} \L_0^p(\cal{M}) \oplus \ker A_p$. 

The following result is stated in \cite[Theorem 1.1.7 p.~624]{JuM10}. Note that the proof of this result uses the interpolation result \cite[Lemma 1.1.6 p.~623]{JuM10} which seems false in the light of the classical problem \cite[Problem 5.4 p.~1143]{KaMS03} on the  interpolation of compact operators. However, the result \cite[Lemma 1.1.6 p.~623]{JuM10} can be replaced by \cite[Theorem 5.5 p.~1143]{KaMS03}. It is worth mentioning that in this result, the choice $q=\infty$ is allowed.

\begin{thm}
\label{thm-interpolation-Junge-Mei}
Let $\cal{M}$ be a finite von Neumann algebra equipped with a normal finite faithful trace. Let $(T_t)_{t \geq 0}$ be a sub-Markovian semigroup of operators acting on $\cal{M}$  satisfying $\norm{T_t}_{\L_0^1(\cal{M}) \to \L^\infty(\cal{M})} \lesssim \frac{1}{t^{\frac{d}{2}}}$ for some $d>0$ and any $t > 0$ and such that $A^{-1}$ is compact on $\L_0^2(\cal{M})$. Then for all $1 \leq p < q \leq \infty$ such that $\Re z > \frac{d}{2}(\frac{1}{p}-\frac{1}{q})$ we have a well-defined compact operator $A^{-z} \co \L_0^p(\cal{M}) \to \L^q_0(\cal{M})$.
\end{thm}

Now, we will prove in the next result a completely bounded variant with different assumptions. For the second proof of Corollary \ref{th-spectral-coulhon}, the crucial case is the case $q=\infty$. 

\begin{thm}
\label{Th-Hardy-mien}
Let $\cal{M}$ be a finite von Neumann algebra equipped with a normal finite faithful trace. Consider a symmetric sub-Markovian semigroup $(T_t)_{t \geq 0}$ of operators acting on $\cal{M}$.   Suppose that $1 \leq p < q \leq \infty$. The estimate $\norm{T_t}_{\cb, \L^p(\cal{M}) \to \L^q(\cal{M})} \lesssim \frac{1}{t^{\frac{d}{2}(\frac{1}{p}-\frac{1}{q})}}$ for some $d>0$ and any $0 < t \leq 1$ implies the complete boundedness of the operator $A^{-z} \co \L_0^p(\cal{M}) \to \L^q_0(\cal{M})$ for any complex number $z$ such that $\Re z > \frac{d}{2}(\frac{1}{p}-\frac{1}{q})$. 	
The same is true for norms instead of completely bounded norms. 
\end{thm}

\begin{proof}
\textit{First proof for $1<p < q <\infty$.} We let $\alpha \ov{\mathrm{def}}{=} \frac{d}{2}(\frac{1}{p}-\frac{1}{q})$. By \cite[Proposition 1.1.5 p.~623]{JuM10}, we have a completely bounded operator $A^{-\zeta} \co \L_0^p(\cal{M}) \to \L^p_0(\cal{M})$ for any complex number $\zeta$ such that $\Re \zeta > 0$ since $1 < p <\infty$. Now, for any complex number $z$ such that $\Re z > \alpha$ we can write $A^{-z} = A^{-\alpha}A^{-(z-\alpha)}$. We have $\Re (z-\alpha) > 0$. With Theorem \ref{th-Zhao}, we have a completely bounded operator $A^{-\alpha} \co \L_0^p(\cal{M}) \to \L^q_0(\cal{M})$. So the conclusion is obvious by composition.

\textit{Second proof for $1 \leq p < \infty$.} By Lemma \ref{Lemma-Gibbs}, the operator $A_2$ has a spectral gap. So Proposition \ref{prop-exp} gives for some constant $\omega >0$ the estimate
$$
\norm{T_t}_{\cb,\L_0^1(\cal{M}) \to \L^\infty_0(\cal{M})} 
\lesssim \e^{-\omega t}, \quad t \geq 1.
$$
Since the von Neumann algebra $\cal{M}$ is finite, by composition with Proposition \ref{Prop-inclusion-cb}, we obtain the estimates
\begin{equation}
\label{estim-786}
\norm{T_t}_{\cb,\L_0^p(\cal{M}) \to \L^q_0(\cal{M})} 
\lesssim \e^{-\omega t}, \quad t \geq 1
\end{equation}
and $\norm{T_{t}}_{\L^2_0(\cal{M}) \to \L^2_0(\cal{M})} \ov{\eqref{magic-estimates-3}}{\lesssim} \e^{-\omega t}$ if $t \geq 1$. The last point means that the semigroup $(T_t)_{t \geq 0}$ is exponentially stable on the Hilbert space $\L^2_0(\cal{M})$. Consequently, by \cite[Corollary 3.3.6 p.~76]{Haa06} we know that
\begin{equation}
\label{A2-Haase}
A_2^{-z}
=\frac{1}{\Gamma(z)} \int_{0}^{\infty} t^{z-1} T_t \d t
\end{equation}
(strong integral) for any complex number $z$ such that $\Re z >0$. Now, if $\Re z > \frac{d}{2}(\frac{1}{p}-\frac{1}{q})$, we deduce that
\begin{align*}
\MoveEqLeft
\int_{0}^{\infty} |t^{z-1}| \norm{T_t}_{\cb,\L_0^p(\cal{M}) \to \L^q_0(\cal{M})} \d t \\
&=\int_{0}^{1} t^{\Re z-1} \norm{T_t}_{\cb,\L_0^p(\cal{M}) \to \L^q_0(\cal{M})} \d t
+\int_{1}^{\infty} t^{\Re z-1} \norm{T_t}_{\cb,\L_0^p(\cal{M}) \to \L^q_0(\cal{M})} \d t \\         
&\ov{\eqref{cbRnpq}\eqref{estim-786}}{\lesssim} \int_{0}^{1} t^{\Re z-1-\frac{d}{2}(\frac{1}{p}-\frac{1}{q})} \d t
+\int_{1}^{\infty} t^{\Re z-1} \e^{-\omega t} \d t 
<\infty.
\end{align*}
Moreover, it is not difficult to check that for any matrix $f \in \M_n(\L^\infty(\cal{M}))$ the map $\R^+ \mapsto \B(\M_n(\L^p(\cal{M})),\M_n(\L^q(\cal{M})))$, $t \mapsto (\Id_{\M_n} \ot T_t)(f)$ is continuous (continuous if the dual space $\B(\M_n(\L^p(\cal{M})),\M_n(\L^q(\cal{M})))$ is equipped with the weak* topology if $q=\infty)$. 
We obtain that the formula \eqref{A2-Haase} 
(strong integral or point weak*) defines a well-defined completely bounded operator $A^{-z} \co \L_0^p(\cal{M}) \to \L^q_0(\cal{M})$. 
\end{proof}

\begin{remark} \normalfont
Note that it is seems to the author that if the von Neumann algebra $\cal{M}$ is finite then using the boundedness of imaginary powers of the operator $A_p$ (relying on the unpublished preprint \cite{JRS} and some classical arguments, see e.g.~\cite[Corollary 2.3]{JiW17}), we can replace in Theorem \ref{th-Zhao} $A^{-\alpha}$ by $A^{-z}$ for any complex number $z$ such that $\Re z =\frac{d}{2}(\frac{1}{p}-\frac{1}{q})$. So in the case $1<p < q <\infty$ of Theorem \ref{Th-Hardy-mien}, we can take any complex number $z$ such that $\Re z \geq  \frac{d}{2}(\frac{1}{p}-\frac{1}{q})$.
\end{remark}

Now, we prove the next result. Combined with Theorem \ref{Th-Hardy-mien} applied with $p=1$, $q=\infty$ and $\alpha = z > \frac{d}{2}$ a second proof of Corollary \ref{th-spectral-coulhon}. 
This is \textit{essentially} the initial proof that the author discovered. It does not use  Proposition \ref{th-Karamata-Iochum}. 
As in \cite[p.~10]{VSCC92}, the definition in the following result of the boundedness of the operator $A^{-z}$ for $z=\frac{\alpha}{2}$ includes implicitly the convergence of the strong integral $A^{-z}=\frac{1}{\Gamma(z)} \int_{0}^{\infty} t^{z-1} T_t \d t$ as in \eqref{A2-Haase}. 

\begin{prop}
\label{prop-Dirac}
Let $\cal{M}$ be a finite von Neumann algebra equipped with a normal finite faithful trace. Consider a symmetric sub-Markovian semigroup $(T_t)_{t \geq 0}$ of operators acting on $\cal{M}$. Suppose that $A^{-\frac{\alpha}{2}} \co \L^1_0(\cal{M}) \to \L^\infty(\cal{M})$ is a well-defined completely bounded map for some $\alpha > 0$. Then for any complex number $z$ with $\Re z \geq \alpha$ the operator $|D|^{-z}$ is trace-class and we have
$$
\tr |D|^{-z} 
\lesssim \bnorm{A^{-\frac{\alpha}{2}}}_{\cb,\L^1_0(\cal{M}) \to \L^\infty_0(\cal{M})}.
$$
\end{prop}

\begin{proof}
We can extend the completely bounded operator $A^{-\frac{\alpha}{2}} \co \L^1_0(\cal{M}) \to \L^\infty(\cal{M})$ by letting $A^{-\frac{\alpha}{2}} = 0$ on the subspace $\ker A_1$. By Proposition \ref{prop-Junge-square}, we obtain a completely bounded operator $A^{-\frac{\alpha}{4}} \co \L^2(\cal{M}) \to \L^\infty(\cal{M})$. Since $\L^\infty(\cal{M}) \subset \L^2(\cal{M})$ it induces a bounded operator $A^{-\frac{\alpha}{4}} \co \L^\infty(\cal{M}) \to \L^\infty(\cal{M})$ which is clearly $*$-preserving and selfadjoint by the integral formula defining the fractional powers.

Applying Proposition \ref{prop-ine-major} to this operator, we obtain
$$
\bnorm{A_2^{-\frac{\alpha}{2}}}_{S^1_0(\L^2(\cal{M}))}
\leq \bnorm{A_2^{-\frac{\alpha}{2}}}_{S^1(\L^2(\cal{M}))}
\lesssim \bnorm{A^{-\frac{\alpha}{2}}}_{\cb,\L^1(\cal{M}) \to \L^\infty(\cal{M})}
= \bnorm{A^{-\frac{\alpha}{2}}}_{\cb,\L^1_0(\cal{M}) \to \L^\infty_0(\cal{M})}.
$$
Using \eqref{D-1}, we see that $
|D|^{-z}
=\begin{bmatrix} 
(\partial^* \partial)^{-\frac{z}{2}} & 0 \\ 
0 & \big(\partial \partial^*\big)^{-\frac{z}{2}}
\end{bmatrix}            
$. With Theorem \ref{Th-unit-eq}, we obtain that
$$
\tr |D|^{-z} 
=2\tr(\partial^* \partial)^{-\frac{z}{2}}
=2\tr A_2^{-\frac{z}{2}}
\leq 2\bnorm{ A_2^{-\frac{z}{2}}}_{S^1(\L^2_0(\cal{M}))} 
\leq 2\bnorm{ A_2^{-\frac{\alpha}{2}}}_{S^1(\L^2_0(\cal{M}))}. 
$$
where we use the ideal property of the space $S^1(\L^2_0(\cal{M}))$ and the fact that for any $t \in \mathbb{R}$ each operator $A_2^{\i t}$ is a unitary on $\L^2_0(\cal{M})$, see \cite[p.~596]{KaR97b}. The proof is complete.
\end{proof}


We finish this section with the following converse to Theorem \ref{th-Zhao} which contains a result stated in \cite[Corollary 4.31 p.~69]{Zha18} without proof. We adapt a well-known argument of Coulhon \cite[p.~490]{Cou90} (see also \cite[Theorem II.3.1 p.~14]{VSCC92} or the generalization \cite[Lemma 1.3]{CoM93}) using completely bounded norms instead norms, generalized by the way to complex numbers.

\begin{thm}
\label{Th-Hardy-converse}
Let $\cal{M}$ be a semifinite von Neumann algebra equipped with a normal semifinite faithful trace. Consider a symmetric sub-Markovian semigroup $(T_t)_{t \geq 0}$ of operators acting on $\cal{M}$.   Suppose that $1 < p < q \leq \infty$. If there exists some complex number $z \in \mathbb{C}$ with $\Re z = \frac{d}{2}(\frac{1}{p}-\frac{1}{q})$ for some $d>0$ such that $A^{-z} \co \L_0^p(\cal{M}) \to \L^q(\cal{M})$ is completely bounded then $\norm{T_t}_{\cb, \L^p_0(\cal{M}) \to \L^q_0(\cal{M})} \lesssim \frac{1}{t^{\frac{d}{2}(\frac{1}{p}-\frac{1}{q})}}$ is satisfied for any $t > 0$.
The same is true for norms instead of completely bounded norms. 
\end{thm}

\begin{proof}
Let $z \in \mathbb{C}$ such that $\Re z = \frac{d}{2}(\frac{1}{p}-\frac{1}{q})$. Suppose that the operator $A^{-z} \co \L_0^p(\cal{M}) \to \L^q_0(\cal{M})$ is completely bounded. Note that by \cite[Section 5]{JMX06}, the semigroup $(\Id_{\B(\ell^2)} \ot T_t)_{t \geq 0}$ is symmetric sub-Markovian with weak* infinitesimal generator $-(\Id \ot A)$. By \cite[Proposition 5.4]{JMX06}, the opposite $\Id \ot A_p$ of the infinitesimal generator of the strongly continuous semigroup $(T_t)_{t \geq 0}$ of contractions acting on the Banach space $\L^p(\B(\ell^2) \otvn \cal{M})$ is sectorial of type $\pi|\frac{1}{p}-\frac{1}{2}| < \frac{\pi}{2}$. Consequently by \cite[Lemma 3.1]{JMX06}, this semigroup $(\Id \ot T_t)_{t \geq 0}$ is bounded analytic on the Banach space $\L^p(\B(\ell^2) \otvn \cal{M})$ since $p > 1$. So by \cite[Proposition 3.4.3 p.~76]{Haa06}, we deduce that for any $t > 0$
\begin{align}
\MoveEqLeft
\label{eq-979789}
\bnorm{A_p^{z}T_t}_{\cb,\L^p_0(\cal{M}) \to \L^p_0(\cal{M})} 
\ov{\eqref{norm-cb-Sp}}{=} \bnorm{(\Id \ot A_p^{z})(\Id \ot T_t)}_{S^p[\L^p_0(\cal{M})] \to S^p[\L^p_0(\cal{M})]} 
\leq \frac{1}{t^{\Re z}}
= \frac{1}{t^{\frac{d}{2}(\frac{1}{p}-\frac{1}{q})}}.
\end{align} 
For any $t > 0$, note that $T_t=A^{-z}A_p^{z}T_t$. We conclude that
\begin{align*}
\MoveEqLeft
\norm{T_t}_{\cb,\L_0^p(\cal{M}) \to \L^q_0(\cal{M})}
\leq \norm{A^{-z}}_{\cb,\L_0^p(\cal{M}) \to \L^q_0(\cal{M})} \bnorm{A_p^{z}T_t}_{\cb,\L_0^p(\cal{M}) \to \L^p_0(\cal{M})} 
\ov{\eqref{eq-979789}}{\lesssim} \frac{1}{t^{\frac{d}{2}(\frac{1}{p}-\frac{1}{q})}}.
\end{align*}
Finally, we can use a similar method for the norms instead of the completely norms.
\end{proof}


\section{Discussion in various classical contexts}
\label{sec-examples}
\subsection{Some general remarks on the commutative case}

We suppose in this section that $\cal{M}$ is a commutative von Neumann algebra $\L^\infty(\Omega)$ associated to a second countable compact topological space $\Omega$ equipped with a finite Borel regular measure $\mu$ with full support and that we have a symmetric sub-Markovian semigroup $(T_t)_{t \geq 0}$ of operators acting on the Hilbert space $\L^2(\Omega)$ with infinitesimal generator $-A_2$. We  assume that each operator $T_t \co \L^2(\Omega) \to \L^2(\Omega)$ is a convolution operator by a \textit{continuous} kernel $K_t \co \Omega \times \Omega \to \R^+$. Note that the existence of this integral kernel is not automatic. This means that
$$
(T_tf)(x)
=\int_\Omega K_t(x,y)f(y) \d \mu(y), \quad f \in \L^2(\Omega), x \in \Omega, t>0.
$$
We have the symmetry property
\begin{equation}
\label{symmetry}
K_t(x,y) 
= K_t(y,x), \quad x,y \in \Omega, t>0
\end{equation}
and the semigroup property
\begin{equation}
\label{semigroup-property}
K_{s+t}(x,y) 
=\int_\Omega K_s(x, z)K_t(z, y)\d \mu (z),\quad x,y \in \Omega, s,t >0.
\end{equation}	
If $\delta >0$, we introduce the diagonal upper estimate
\begin{equation}
\label{upper}
K_t(x,x) 
\lesssim \frac{1}{t^{\delta}}, \quad x \in \Omega,\ 0 < t \leq 1
\end{equation}
and the estimate
\begin{equation}
\label{lower}
\frac{1}{t^{\delta}} 
\lesssim \int_\Omega K_t(x,x) \d\mu(x), \quad x \in \Omega,\ 0 < t \leq 1.
\end{equation}
Now, we connect these estimates to the main topics of this paper. The first part of the following result is folklore. For the sake of completeness, we give a proof.

\begin{prop}
\label{Th-kernel-spectral}
\begin{enumerate}
\item The estimate given in \eqref{upper} is equivalent to having a local Coulhon-Varapoulos dimension less than or equal to $2\delta$.
\item  Suppose that \eqref{lower} is satisfied. If the operator $D$ of \eqref{Hodge-Dirac-I} is $p$-summable for some $p>0$ then $p \geq 2\delta$.
\end{enumerate}
\end{prop}

\begin{proof}
1. For all $x,y \in \Omega$ and any $t>0$, we have by a standard computation using Cauchy-Schwarz inequality
\begin{align*}
\MoveEqLeft
K_{t}(x,y)
\ov{\eqref{semigroup-property}}{=} \int_\Omega K_{\frac{t}{2}}(x,z)K_{\frac{t}{2}}(z,y) \d\mu(z) 
\leq \left(\int_\Omega K_{\frac{t}{2}}(x,z)^2 \d\mu(z)\right)^{1/2} \left(\int_\Omega K_{\frac{t}{2}}(z,y)^2 \d\mu(z)\right)^{\frac{1}{2}} \\
&\ov{\eqref{symmetry}}{=} \left(\int_\Omega K_{\frac{t}{2}}(x,z)K_{\frac{t}{2}}(z,x) \d\mu(z)\right)^{1/2}\left(\int_\Omega K_{\frac{t}{2}}(y,z)K_{\frac{t}{2}}(z,y) \d\mu(z)\right)^{\frac{1}{2}} 
\ov{\eqref{semigroup-property}}{=} \big(K_{t}(x,x) K_{t}(y,y)\big)^{\frac{1}{2}}.
\end{align*}
By \eqref{upper}, we obtain the estimate $
\norm{K_t}_{\L^\infty(\Omega \times \Omega)}
\lesssim \frac{1}{t^\delta}$ for any $0 < t \leq 1$. By Dunford-Pettis theorem \eqref{Dunford-Pettis-1}, this is equivalent to the estimate $\norm{T_t}_{\L^1(\Omega) \to \L^\infty(\Omega)} \lesssim \frac{1}{t^\delta}$. 

2. Now, suppose that the operator $|D|^{-p}$ is trace-class. By Lemma \ref{Lemma-theta-summable}, we have 
\begin{equation}
\label{inter-24}
\tr \e^{-tD^2} 
\lesssim \frac{1}{t^{\frac{p}{2}}}, \quad t > 0.
\end{equation}
With \eqref{carre-de-D}, we infer for any $t>0$ that the operator $T_t=\e^{-tA_2}$ is trace-class and that
\begin{align*}
\MoveEqLeft
\int_{\Omega} K_t(x,x) \d \mu (x) 
=\tr T_t 
=\tr \e^{-tA_2}
\ov{\eqref{carre-de-D}\eqref{inter-24}}{\leq} \frac{1}{t^{\frac{p}{2}}}
\end{align*}  
where we use \cite[Corollary 3.2 p.~237]{Bri91} in the first equality, since the function $K_t$ is continuous. However, we have  
$$
\frac{1}{t^\delta}
\ov{\eqref{lower}}{\lesssim} \int_{\Omega} K_t(x,x) \d \mu (x), \quad 0 < t \leq 1.
$$ 
We conclude that $\frac{1}{t^\delta} \lesssim \frac{1}{t^{\frac{p}{2}}}$ for any $0< t \leq 1$ and consequently $p \geq 2\delta$.
\end{proof}

Note that is clear that the estimate \eqref{upper} implies the estimate $\tr T_t =\int_\Omega K_t(x,x) \d\mu(x) \lesssim \frac{1}{t^{\delta}}$ for any $0 < t \leq 1$. So using Proposition \ref{th-Karamata-Iochum} and a reasoning similarly to the proof of Proposition \ref{cor-trace-class}, we obtain that the spectral dimension is less than $2\delta$ if $K_t$ is \textit{continuous}.

\begin{remark} \normalfont
With \cite[Theorem 4.3]{Bri91} (see also \cite{Jef16}), we can prove a generalization of the second part of Theorem \ref{Th-kernel-spectral} replacing the continuity of the kernel $K_t$ by a weaker assumption.
\end{remark}

\begin{remark} \normalfont
The result \cite[Theorem 2.7]{AtE17} and its proof, connects the estimate \eqref{upper} with the estimate $\int_\Omega K_t(x,x) \d\mu(x) \lesssim \frac{1}{t^{\delta'}}$ for $0<t \leq 1$ if we have some control of the $\L^\infty(\Omega)$-norm of eigenfunctions of the operator $A_2$.
\end{remark}

In the sequel one will sometimes encounter the two-sided estimate
\begin{equation}
\label{sub-gaussian-intro}
K_t(x,y) 
\approx \frac{1}{t^{\frac{\alpha}{\beta}}} \exp \bigg(-c\bigg(\frac{\dist^\beta(x,y)}{t}\bigg)^{\frac{1}{\beta-1}}\bigg), \quad 0 < t \leq 1, \ x,y \in \Omega
\end{equation}
for some distance $\dist$ on $\Omega$, where $\alpha > 0$ and $\beta > 1$ are two parameters. This equivalence is called a Gaussian estimate if $\beta=2$. 


\paragraph{Dirichlet forms on classical $\L^2$-spaces} 
Now, we focus on semigroups and Dirichlet forms defined on commutative $\L^2$-spaces. We will briefly review the definitions and results that will be needed in the remainder of this paper. For further details, we refer to \cite{FOT11}, \cite{BoH91} and \cite{MaR92}. Consider a Dirichlet form $\cal{E} \co \dom \cal{E} \to [0,\infty[$ defined on a dense  subspace $\dom \cal{E}$ of $\L^2(X)$ for a second countable locally compact space $X$ equipped with a positive Borel measure $\mu$ with full support. We also denote the associated bilinear form by $\cal{E}$. A Dirichlet form $\cal{E}$ is said to be strongly local \cite[p.~6]{FOT11} if for any functions $f, g \in \dom \cal{E}$ with compact supports such that $f$ is constant on a neighbourhood of $\supp g$ we have $\cal{E}(f,g)= 0$. The form $\cal{E}$ is called regular \cite[p.~6]{FOT11} if the subspace $\dom \cal{E} \cap \mathrm{C}_c(X)$ is dense in the space $\mathrm{C}_c(X)$ endowed with the uniform norm and in $\dom \cal{E}$ equipped with the norm $\norm{f}_{\dom \cal{E}} \ov{\mathrm{def}}{=} [\norm{f}^2_{\L^2(X)}+\cal{E}(f,f)]^{1/2}$.




 
If $f$ belongs to $\dom \cal{E} \cap \mathrm{C}_c(X)$, the energy measure $\mu_{\la f\ra}$ is a positive Borel measure on the space $X$ defined in \cite[p.~123]{FOT11} by
\begin{equation}
\int_ X g \d\mu_{\la f\ra}
\ov{\mathrm{def}}{=} \cal{E}(gf,f)-\frac{1}{2}\cal{E}(f^2,g), \quad g \in \dom \cal{E} \cap \mathrm{C}_c(X).
\end{equation}
A positive Borel measure $\nu$ on the space $X$ with support $X$ is called energy dominant for $\cal{E}$ if all energy measures $\mu_{\la f\ra}$ are absolutely continuous with respect to $\nu$. 

\subsection{Laplace-Beltrami operators on compact Riemannian manifolds}
\label{Section-Beltrami}

Let $M$ be a (second countable) smooth $n$-dimensional connected compact Riemannian manifold with $n \geq 1$. We consider the Riemannian distance $\d$ on $M$ defined by 
$$
\d(x,y)
\ov{\mathrm{def}}{=} \inf \big\{ \textrm{Length}(\gamma) : \gamma \co [a,b] \to M \text{ smooth}, \gamma(a)=x,\gamma(b)=y \big\}, \quad x,y \in M
$$ 
and the Riemannian volume $\mu$ defined in \cite[Theorem 3.11 p.~59]{Gri09}. For any $x \in M$ and any $r>0$ let us denote by $B(x,r)$ the Riemannian ball of radius $r$ centered at $x$. We denote by $V(x,r) \ov{\mathrm{def}}{=} \mu(B(x,r))$ the Riemannian volume of this ball. As in \cite[p.~108]{Gri09}, we can consider the symmetric bilinear form $\cal{E} \co \W^1(M) \times \W^1(M) \to \R$ defined by
\begin{equation}
\label{Dirichlet-Riemann}
\cal{E}(f,g)
=\int_M \la \nabla f, \nabla g\ra \d\mu.
\end{equation}
It is known that it is a Dirichlet form on $\L^2(M)$.
The associated selfadjoint operator on the Hilbert space $\L^2(M)$ is the Laplace-Beltrami operator $\Delta$ with domain $\dom \Delta=\W^2(M)$ satisfying
$$
\cal{E}(f,g)
\ov{\mathrm{def}}{=} \la-\Delta f,g \ra_{\L^2(M)}, \quad f \in \W^2(M), g \in \W^1(M).
$$ 
The generated semigroup $(T_t)_{t \geq 0}$ is the heat semigroup $(\e^{t \Delta})_{t \geq 0}$ whose kernel $K_t \co M \times M \to \R$ is strictly positive and $\C^\infty$ by \cite[Theorem 5.3.1 p.~149]{Dav89} and \cite[p.~198]{Gri09} for any $t >0$. 

By \cite[Corollary 5.3.5]{Hsu02} combined with \cite[Theorem 1.1]{Gri97} (see also \cite[Corollary 15.17 p.~407]{Gri03}), we have
\begin{equation}
\label{Heat-kernel-estimates}
K_t(x,y)
\approx \frac{1}{t^{\frac{n}{2}}}\exp\bigg(-c\frac{\d(x,y)^2}{t}\bigg)  , \quad 0<t<1, x,y \in M.
\end{equation}
which is the particular case of \eqref{sub-gaussian-intro} with $\alpha=n$ and $\beta=2$. We can use Theorem \ref{Th-kernel-spectral}. Finally, recall that in this setting, we have a gradient operator $\nabla$ which is a closed operator from $\L^2(M)$ into $\L^2(M,\mathrm{T} M)$ with domain $\dom \nabla=\W^1(M)$.

Here the bimodule $\cal{H}$ identifies to $\L^2(M,\mathrm{T} M)$ where the left bimodule action is given by $(f X)(x)=f(x) X(x)$ where $f \in \L^\infty(M)$ and $X \in \L^2(M,\mathrm{T} M)$. This defines of course the right bimodule action. The Hodge-Dirac operator \eqref{Hodge-Dirac-I} identifies to the operator
$
\begin{bmatrix} 
0 & -\div \\ 
\nabla & 0 
\end{bmatrix}$ 
and acts on the Hilbert space $\L^2(M) \oplus_2 \L^2(M,\mathrm{T} M)$. Finally, recall that for any $t>0$ the operator $\e^{t \Delta}$ is trace-class and has a short time asymptotic expansion as $t \to 0$ given by
\begin{equation*}
\tr \e^{t \Delta} 
= \frac{\mathrm{Vol}(M)}{(4\pi t)^{n/2}} + \frac{t}{6(4\pi t)^{n/2}} \int_M \cal{K}(x) \d\mu (x) +o(t^{1-n/2})
\end{equation*}
where $\cal{K}(x)$ denotes the scalar curvature of $(M,g)$. In particular $\tr \e^{t \Delta}  \sim \frac{\mathrm{Vol}(M)}{(4\pi t)^{n/2}}$ when $t \to 0$. Consequently, we deduce by Proposition \ref{prop-trace-spectral triple} that the spectral dimension is equal to $n$.

For any sufficiently regular function $f$, we have $\mu_{\la f\ra}=\la \nabla f, \nabla f\ra \mu $. So the measure $\mu$ is energy dominant. Using \cite[Theorem 5.1, Remark 4.4]{HKT15} with Lemma \ref{Lemma-Gibbs}, it is not difficult to see that we have a compact spectral triple.

In conclusion, in the setting of connected \textit{compact} smooth Riemannian manifolds, the spectral dimension of the compact spectral triple $\left(\L^\infty(M),\L^2(M) \oplus_2 \L^2(M,\mathrm{T} M),\begin{bmatrix} 
0 & -\div \\ 
\nabla & 0 
\end{bmatrix}\right)$ is always equal to the local Coulhon-Varopoulos dimension of the associated heat semigroup.

%


\subsection{Sub-Laplacians on compact Lie groups}
\label{Sec-Lie-groups}

Let $G$ be a connected compact Lie group equipped with a Haar measure $\mu_G$ on $G$. We consider a finite sequence $X \ov{\mathrm{def}}{=}(X_1,\ldots,X_m)$ of left invariant smooth vector fields which generate the Lie algebra $\frak{g}$ of the group $G$ such that the vectors $X_1(e),\ldots, X_m(e)$ are linearly independent. We say that it is a family of left-invariant H\"ormander vector fields. For any $r > 0$ and any $x \in G$, we denote by $B(x,r)$ the open ball with respect to the Carnot-Carath\'eodory metric centered at $x$ and of radius $r$, and by $V(r) \ov{\mathrm{def}}{=} \mu_G(B(x,r))$ the Haar measure of any ball of radius $r$. It is well-known, e.g. \cite[p. 124]{VSCC92} that there exist $d \in \N^*$, $c,C > 0$ such that
\begin{equation}
\label{local-dim}
c r^d 
\leq V(r) 
\leq C r^d, \quad r \in ]0, 1[.
\end{equation}
The integer $d$ is called the local dimension of $(G,X)$. We consider the sub-Laplacian $\Delta$ on $G$ defined by $ 
\Delta 
\ov{\mathrm{def}}{=} -\sum_{k=1}^{m} X_k^2$. 
Let $\Delta_2 \co \dom \Delta_2 \subset \L^2(G) \to \L^2(G)$ be the smallest closed extension of the closable unbounded operator $\Delta|_{\C^\infty(G)}$ to $\L^2(G)$. We denote by $(T_t)_{t \geq 0}$ the associated symmetric Makovian semigroup of convolution operators on $\L^\infty(G)$, see \cite[pp.~20-21]{VSCC92} and \cite[p.~301]{Rob91}, with kernel $(h_t)_{t \geq 0}$. Its local Coulhon-Varopoulos dimension is equal to $d$. Indeed, on the one hand  we have by \cite[Theorem V.4.3 p.~70]{VSCC92} (see also \cite[p.~274]{Rob91}) the estimate $\norm{h_t}_{\L^\infty(G)} \lesssim \frac{1}{t^\frac{d}{2}} $ for any $0 < t \leq 1$. On the other hand, we have by \cite[(3) p.~113]{VSCC92}
$$
\frac{1}{t^\frac{d}{2}} 
\ov{\eqref{local-dim}}{\approx}  \frac{1}{V(\sqrt{t})} 
\lesssim h_t(e) 
\leq \norm{h_t}_{\L^\infty(G)}, \quad 0 < t \leq 1.
$$
We conclude with the isometry defined in \eqref{Dunford-Pettis-1}.

By \cite[Proposition 6.4, Theorem 6.6]{Arh24}, we have a compact spectral triple $(\L^\infty(G),\L^2(G) \oplus_2 \L^2(G,\ell^2_m),D)$ where $\partial$ is the closure of the gradient operator defined by $\nabla f \ov{\mathrm{def}}{=} \big(X_{1} f, \ldots, X_{m} f\big)$ for any function $f \in \mathrm{C}^\infty(G)$. By Theorem \ref{Th-kernel-spectral} combined with \eqref{inclusions-weak-Schatten}, the spectral dimension \eqref{Def-spectral-dimension} of this compact spectral triple is equal to the local dimension $d$ of $(G,X)$. Our previous result \cite[Theorem 6.9]{Arh24} only says that $D$ is $p$-summable for any $p>d$. Again, the spectral dimension of the spectral triple is always equal to the local Coulhon-Varopoulos dimension of the semigroup. 

\subsection{Sub-Laplacians on compact Riemannian manifolds}
Let $M$ be a (second countable) $n$-dimensional connected manifold and let $X=(X_1,\ldots,X_m)$ be a family of smooth vector fields. For any $s \in \N$, we can define the subspace
$$
\mathrm{Lie}_s\, X \ov{\mathrm{def}}{=}
\Span \big\{ X_{j_1},[X_{j_1}, X_{j_2}],\ldots,[X_{j_1},[X_{j_2},\ldots,[X_{j_{k-1}},X_{j_k}]\cdots] : k \leq s,1 \leq j_1,\ldots,j_k \leq m \big\}. 
$$ 
of the Lie algebra of vector fields. Assume that $X$ satisfies the uniform H\"ormander condition of step $s$, i.e.~for any $x \in M$ we have 
$$
\{Y(x): Y \in \mathrm{Lie}_s X\} 
= \mathrm{T}_xM.
$$
Now, suppose in addition that $M$ is a compact Riemannian manifold equipped with its  Riemannian volume $\mu$. Let $\Delta_X \co \dom \Delta_X \subset \L^2(M) \to \L^2(M)$ be the corresponding sub-Laplacian and consider the associated semigroup $(T_t)_{t \geq 0}$ of operators. By \cite[p.~72]{VSCC92}, the estimates \eqref{sub-gaussian-intro} are satisfied for $\beta=2$, some $\alpha>0$ and some distance $\d$ (see also \cite[p.~3439]{GJL20} for a related discussion). Hence, we can use Theorem \ref{Th-kernel-spectral} (with \eqref{inclusions-weak-Schatten}) and we obtain that the spectral dimension is still equal to the local Coulhon-Varopoulos dimension of the semigroup. The spectral dimension also seems covered by the main result of the paper \cite{InT20} by a different method. 


\subsection{$\RCD^*(K,N)$-spaces}


In \cite[p.~4]{EKS15}, the class of metric measure spaces satisfying the Riemannian curvature-dimension condition $\RCD^*(K,N)$ $K \in \R$ and $N \in [1,\infty)$ was introduced. This condition is a generalization of the <<Ricci curvature lower bound>> for the non-smooth setting and a strengthening of the curvature-dimension condition introduced by Lott and Villani \cite{LoV09} and Sturm \cite{Stu06a} \cite{Stu06b}. A $\RCD^*(K,N)$-space looks like to a Riemannian manifold in the sense that Cheeger energy and Sobolev spaces are quadratic and Hilbertian with the conditions $\textrm{Ric} \geq K$ and $\dim \leq N$. We refer the reader to the papers \cite{EKS15}, \cite{JLZ16}, \cite{AHT18}, \cite{AMS19}, \cite{AHPT21}, \cite{KuK19}, \cite{KuL21} and references therein for more information on these spaces. 


\paragraph{Cheeger energy} A triple $(X,\dist,\mu)$ is said to be a metric measure space if $(X,\dist)$ is a complete separable metric space and if $\mu$ is a Borel measure on $X$ with full support. Let $\Lip(X,\dist)$ be the set of all Lipschitz continuous functions on $X$. For $f \in \Lip(X)$, we define the local Lipschitz constant $\lip(f)(x)$ of $f$ at $x \in X$ by
\[
\lip(f)(x) 
\ov{\mathrm{def}}{=} \varlimsup_{y \to x} \frac{ | f (y) - f (x) | }{ d (x,y) }.
\]
We regard $\lip (f)$ as a function on $X$. We define the Cheeger energy $\Ch \co \L^2(X) \to [0,\infty]$ by
\begin{equation}
\label{eq:defchee}
\Ch(f)
\ov{\mathrm{def}}{=} \frac{1}{2}\inf\left\{\liminf_{n \to \infty}\int_X(\lip f_n)^2 \d \mu :\ f_n \in \Lip(X) \cap \L^2(X) \cap \L^{\infty}(X),  \mbox{$f_n \to f$ in $\L^2(X)$}
\right\}.
\end{equation}
The Sobolev space $\W^{1,2}(X,\d,\mu)$ is by definition the space of functions of $\L^2(X)$ having finite Cheeger energy, equipped with the norm $\norm{f}_{\W^{1,2}(X,\d,\mu)} \ov{\mathrm{def}}{=} \big(\norm{f}_{\L^2(X)}+2\Ch(f)\big)^{\frac{1}{2}}$. In general, $\W^{1,2}(X,\d,\mu)$ is not a Hilbert space.

\paragraph{Minimal relaxed slopes} Following \cite[Definition 4.2]{AGS13}, we say that a function $g \in \L^2(X)$ is a relaxed slope of $f\in \L^2(X)$ if there exist $\tilde{g}\in \L^2(X)$ and Lipschitz functions $f_n\in \L^2(X)$ such that:
\begin{itemize}
\item[(a)] $f_n\to f$ in $\L^2(X)$ and $|\nabla f_n|$ weakly converge to $\tilde{g}$ in $\L^2(X)$,
\item[(b)] $\tilde{g}\leq g$ $\mu$-a.e.~in $X$.
\end{itemize}
We say that $g$ is the minimal relaxed slope of $f$ if its $\L^2(X)$ norm is minimal among relaxed slopes. We refer to \cite[Theorem 7.4]{AGS13} \cite[Theorem~6.2]{AGS14b} for other equivalent definitions (e.g.~minimal generalized upper gradients). 

\paragraph{Infinitesimally Hilbertian metric measure space} A metric measure space $(X,\d,\mu)$ is said to be infinitesimally Hilbertian if the associated Cheeger energy is a quadratic form on $\L^2(X)$, that is $\Ch(f)+\Ch(g)=\frac{1}{2}(\Ch(f+g)+\Ch(f-g))$ for any functions $f,g$ of $\W^{1,2}(X,\d,\mu)$. This is equivalent to saying that $\W^{1,2}(X,\d,\mu)$ is a Hilbert space.

Moreover, for infinitesimally Hilbertian spaces, a classical argument shows the existence of a minimal relaxed slope $|\nabla f|_*$ for any function $f \in \W^{1,2}(X,\dist,\meas)$, which gives an integral representation of the Cheeger energy:
$$
\Ch(f)
=\frac{1}{2} \int_X |\nabla f|_{*}^2 \d \mu, \qquad  f \in \W^{1,2}(X,\dist,\mu).
$$
The minimal relaxed slope is a local object, meaning that $|\nabla f|_* = |\nabla g|_*$ $\mu$-a.e.~on $\{f = g\}$ for any funtions $f, g \in \W^{1,2}(X,\dist,\mu)$, and it satisfies the chain rule, namely $|\nabla (fg)|_* \le f |\nabla g|_* + g |\nabla f|_*$ $\mu$-a.e.~on $X$ for any functions $f,g \in \W^{1,2}(X,\dist,\mu)$. In addition, by \cite[Section 4.3]{AGS14a} (see also \cite{Gig15}), the formula 
$$
\Gamma(f,g)
\ov{\mathrm{def}}{=} \lim_{\epsi \to 0}\frac{|\nabla (f+\epsi g)|_*^2-|\nabla f|_*^2}{2\epsi}
$$
provides a symmetric bilinear form $\Gamma \co \W^{1,2}(X,\dist,\mu) \times \W^{1,2}(X,\dist,\mu) \to \L^1(X)$ such that $\Gamma(f,f) = |\nabla f |_*^2$ for any $f \in \W^{1,2}(X,\dist,\mu)$. 
We denote the bilinear form corresponding to $2 \Ch$ by $\cal{E}$ with $\dom \cal{E} = \W^{1,2}(X,\d,\mu)$, that is
\[
\cal{E}(f,g) 
= \int_X \Gamma(f,g) \d \mu, \quad f,g \in \W^{1,2}(X,\d,\mu).
\]
Hence $\cal{E}(f,f) = 2 \Ch(f)$ for any function $f \in \W^{1,2}(X,\d,\mu)$. We obtain a strongly local, regular and symmetric Dirichlet form. We denote by $\Delta$ the selfadjoint operator on $\L^2(X)$ associated with $(\cal{E}, \dom \cal{E})$ and by $(T_t)_{t \geq 0}$ the associated symmetric sub-Markovian semigroup, called heat semigroup. 
We refer to \cite[Section 2.3]{Tew18} for a nice summary of its properties.

\paragraph{Riemannian curvature-dimension condition} A infinitesimally Hilbertian metric measure space $(X,\d,\mu)$ is said satisfy the Riemannian curvature-dimension condition $\RCD^*(K,N)$ for some $K \in \R$ and $N \in [1,\infty)$ if it satisfies the $\CD^*(K,N)$ condition introduced by Bacher and Sturm in \cite{BaS10}. In \cite[Definition 3.9]{EKS15} and \cite[Definition 9.11]{AMS19}, it is shown that this condition is equivalent to the following three conditions:
\begin{itemize}
\item There exists $C > 0$ and $x_0 \in X$ such that
\begin{equation} 
\label{eq:i'ble}
\int_X \e^{ - C \dist(x_0,x)^2} \d\mu(x) < \infty.
\end{equation}

\item If $f$ belongs to $\W^{1,2}(X,\d,\mu)$ with $| \nabla f |_* \leq 1$ $\mu$-a.e., then $f$ has a 1-Lipschitz representative.

\item For any function $f \in \dom \Delta$ with $\Delta f \in \dom \cal{E}$ and any positive function $g$ of $\dom \Delta  \cap \L^\infty(X)$ with $\Delta g \in \L^\infty(X)$, we have
\begin{equation*}
\label{eq:Bochner}
\frac12 \int_X | \nabla f |^2 \Delta g  \d \mu- \int_X \langle \nabla f , \nabla \Delta f \rangle g  \d \mu
\geq
K \int_X | \nabla f |^2 g  \d \mu + \frac{1}{N} \int_X ( \Delta f )^2 g \d \mu.
\end{equation*}
\end{itemize}
The last condition is a weak formulation of Bochner inequality 
\begin{equation*}
\label{Bochner}
\frac12\Delta|\nabla f|^2-\langle\nabla f, \nabla\Delta f\rangle
\geq K \cdot |\nabla f|^2+\frac1N \cdot |\Delta f|^2
\end{equation*}
which is satisfied for each smooth function $f$ on a Riemannian manifold $(M,g)$ provided that $K \in \R$ is a lower bound for the Ricci curvature on $M$ and $N \in (0,\infty]$ is an upper bound for the dimension of $M$. 

\begin{remark} \normalfont
It is proved in the paper \cite[Corollary 13.7]{CaM21} that the condition $\RCD^*(K,N)$ is identical to the condition $\RCD(K,N)$ for any $K \in \R$ and any $N \in [1,\infty]$.
\end{remark}

Moreover, heat kernel estimates are known. We will use only the following estimates for $\RCD^*(0,N)$-spaces provided by \cite[Theorem 1]{JLZ16}. For any $\epsi >0$ there exists a constant $C \geq 0$ such that
\begin{equation}
\label{estimates-RCDstar}
\frac{1}{C \mu(B(x,\sqrt{t}))}\exp \left(-\frac{\dist^2(x,y)}{4(1-\epsi)t} \right) 
\leq K_t(x,y) 
\leq \frac{C}{\mu(B(x,\sqrt{t}))}\exp \left( -\frac{\dist^2(x,y)}{4(1+\epsi)t} \right)
\end{equation}
for any $x,y \in X$ and any $t>0$. It is a generalization of the famous estimates of Li and Yau \cite{LiY86} for heat kernels of complete smooth Riemannian manifolds of dimension $n \geq 2$ with positive Ricci curvature.

\begin{example} \normalfont
\label{Ex-RCD}
For $N \in (1,\infty)$, the metric measure space
\begin{equation}
\label{eq:intro exa}
(X, \dist,\mu)
\ov{\mathrm{def}}{=} \big([0, \pi], \dist_{[0, \pi]},\sin^{N-1}(t)\d t\big)
\end{equation}
is a $\RCD^*(N-1,N)$-space by \cite[Proposition 11.5]{AmS17}. In particular it is a $\RCD^*(0,N)$-space. 
\end{example}

Now, we can state the following result, which highlights a distinction between the spectral dimension and the local dimension of Coulhon-Varopoulos.

\begin{thm}
\label{th-RCD}
For any $N \in (1,\infty)$, the spectral dimension of the Dirac operator associated to the heat semigroup $(T_t)_{t \geq 0}$ of the metric measure space $(X,\dist,\mu)$ described in Example \ref{Ex-RCD} is equal to 1. The local Coulhon-Varopoulos dimension of $(T_t)_{t \geq 0}$ is bigger than $N$.
\end{thm}

\begin{proof}
By \cite[(1.6)]{AHT18}, we have the Weyl's law 
$
\lim_{\lambda \to +\infty}\frac{N(\lambda)}{\lambda^{1/2}}
=1$. 
Combinating Karamata Theorem \ref{thm-Karamata} and Proposition \ref{prop-trace-spectral triple}, we see that the spectral dimension is equal to 1. The heat kernel estimates \eqref{estimates-RCDstar} give the equivalence
\begin{equation}
\label{estim-RCD}
K_t(x,x) \approx
\frac{1}{\mu(B(x,\sqrt{t}))}, \quad t>0,x \in X
\end{equation}
 where $\d\mu=\sin^{N-1}(t)\d t$. Moreover, when $r \to 0^+$ it is easy to check that the volume of balls centered at $x=\pi$ satisfy $
\mu(B(\pi,r))
\approx r^N$. Now, suppose that the inequality \eqref{Varo-Coulhon-dim} is true for some $d>0$. By \eqref{Dunford-Pettis-1}, we have the estimate $K_t(x,x) \lesssim \frac{1}{t^{\frac{d}{2}}}$ for any $x \in [0, \pi]$ and any $0 < t < 1$. Using \eqref{estim-RCD}, we obtain
$$
\frac{1}{\mu(B(x,\sqrt{t}))} 
\lesssim K_t(x,x) 
\lesssim \frac{1}{t^{\frac{d}{2}}}, \quad x \in [0, \pi],\ 0<t<1. 
$$
So $\mu(B(x,\sqrt{t})) \gtrsim t^{\frac{d}{2}}$ and finally $\mu(B(x,r)) \gtrsim r^{d}$ for any $0  < r < 1$. Taking $x=\pi$, we obtain $r^N \gtrsim r^{d}$ if $0  < r < 1$. We conclude that $d \geq N$.
\end{proof}

\begin{remark} \normalfont
When $r \to 0^+$ it is not difficult to prove that $\mu(B(x,r)) \sim r$ for $x \in (0,\pi)$ and $\mu(B(0,r))\sim r^N$.
\end{remark}

\begin{remark} \normalfont
Recall that a metric measure space $(X,\dist,\mu)$ is Ahlfors regular, e.g. \cite[p.~62]{Hei01} if there exists a constant $\alpha>0$ such that $\mu(B(x,r)) \approx r^\alpha$ for all $x \in X$ and all $r \in (0, \diam X)$. It is apparent that the previous proof relies on the fact that the metric measure space defined in \eqref{eq:intro exa} is not <<$\alpha$-lower Ahlfors regular for small balls>> if $\alpha<N$, i.e.~we does not have an estimate $\mu(B(x,r)) \geq r^\alpha$ for any sufficiently small radius $r$ and for all $x \in X$.
\end{remark}


\section{Noncommutative examples}
\label{Sec-noncommutative-examples}

\subsection{Noncommutative tori} 
\label{NC-tori}
We refer to the book \cite{GVF01} and to the paper \cite{CXY13} for background on the noncommutative tori. Let $d \geq 2$. To each $d \times d$ real skew-symmetric matrix $\theta$, one may associate a 2-cocycle $\sigma_\theta \co \Z^d \times \Z^d \to \T$ of the group $\Z^d$ defined by $\sigma_\theta (m,n) \ov{\mathrm{def}}{=} \e^{\frac{\i}{2} \langle m, \theta n\rangle}$ where $m,n \in \Z^d$. We have $\sigma(m,-m) = \sigma(-m,m)$ for any $m \in \Z^d$.

We define the $d$-dimensional noncommutative torus $\L^\infty(\T_{\theta}^d)$ as the twisted group von Neumann algebra $\VN(\Z^d,\sigma_\theta)$. One can provide a concrete realization in the following manner. 
If $(\epsi_n)_{n \in \Z^d}$ is the canonical basis of the Hilbert space $\ell^2_{\Z^d}$ and if $m \in \Z^d$, we can consider the bounded operator $U_m \co \ell^2_{\Z^d} \to \ell^2_{\Z^d}$ defined by 
\begin{equation}
\label{def-lambdas}
U_m(\epsi_n)
\ov{\mathrm{def}}{=} \sigma_\theta(m,n) \epsi_{m+n}, \quad n \in \Z.	
\end{equation}
The $d$-dimensional noncommutative torus $\L^\infty(\T_{\theta}^d)$ is the von Neumann subalgebra of $\B(\ell^2_{\Z^d})$ generated by the $*$-algebra $
\mathcal{P}_{\theta}
\ov{\mathrm{def}}{=}\mathrm{span} \big\{ U^m \ : \ m \in \Z^d \big\}$. Recall that for any $m,n \in \Z^d$ we have
\begin{equation}
\label{product-adjoint-twisted}
U_m U_n 
= \sigma_\theta(m,n) U_{m+n}
\quad \text{and} \quad 
\big(U_m \big)^* 
= \ovl{\sigma_\theta(m,-m)} U_{-m}.	
\end{equation}
The von Neumann algebra $\L^\infty(\T_{\theta}^d)$ is finite with normalized trace given by $\tau(x) \ov{\mathrm{def}}{=}\langle\epsi_{0},x(\epsi_{0})\rangle_{\ell^2_{\Z^d}}$ where $x \in \L^\infty(\T_{\theta}^d)$. In particular, we have $\tau(U_m) = \delta_{m=0}$ for any $m \in \Z^d$.



Let $\Delta$ be the unbounded operator acting on $\L^\infty(\T_{\theta}^d)$ defined on the weak* dense subspace $\mathcal{P}_{\theta}$ by $\Delta(U_m) \ov{\mathrm{def}}{=} 4\pi^2 |m|^2U_m$ where $|m| \ov{\mathrm{def}}{=} m_1^2+\cdots+m_d^2$. Then this operator is weak* closable and its weak* closure is the opposite of a weak* generator of a symmetric Markovian semigroup $(T_t)_{t \geq 0}$ of operators acting on $\L^\infty(\T_{\theta}^d)$, called the noncommutative heat semigroup on the noncommutative torus. We will see that this semigroup satisfies \eqref{Varo-Coulhon-dim-cb} using some properties of the heat semigroup on the $d$-dimensional classical torus $\T^d$. In this setting, the gradient operator $\partial$ is a closed operator from the dense subspace $\dom \Delta_2^{\frac{1}{2}}$ of the space $\L^2(\T_{\theta}^d)$ into the Hilbert space $\ell^2_d(\L^2(\T_{\theta}^d))$ satisfying
$$
\partial(U_m)
=(2\pi\i\, m_1 U_m,\ldots,2\pi\i\, m_d U_m), \quad m \in \Z^d.
$$
Now, we compute the completely bounded local dimension Coulhon-Varopoulos dimension and the spectral dimension. The proof relies on a link with the classical heat semigroup on the classical torus $\T^d$.

\begin{thm}
\label{thm-NC}
The completely bounded local dimension Coulhon-Varopoulos dimension of the noncommutative heat semigroup $(T_t)_{t \geq 0}$ on the noncommutative torus $\T_{\theta}^d$ is $d$. The spectral dimension of $D$ is also equal to $d$.
\end{thm}

\begin{proof}
We denote by 
\begin{equation}
\label{heat-kernel-torus}
h_t(x)
\ov{\mathrm{def}}{=} \sum_{k \in \Z^d} \e^{-|k|^2 t} \e^{2\pi\i k \cdot}, \quad x \in [0,1],\, t>0
\end{equation}
the heat kernel on the $d$-dimensional classical torus $\T^d$ where the series is absolutely and uniformly convergent. 
We will use the normal injective unital $*$-homomorphism $\pi_\theta \co \L^\infty(\T^d) \to \L^\infty(\T^d) \otvn \L^\infty(\T^d)^\op$, $\e^{2\pi\i k \cdot} \mapsto U^k \ot (U^k)^*$. For any $m \in \Z^d$, we have 
\begin{align*}
\MoveEqLeft
(\Id \ot \tau)[\pi_\theta(h_t)(1 \ot U_m)]
\ov{\eqref{heat-kernel-torus}}{=} 
\sum_{k \in \Z^d} \e^{-|k|^2 t} (\Id \ot \tau)\big[\pi_\theta(\e^{2\pi\i k \cdot})(1 \ot U_m) \big] \\
&=\sum_{k \in \Z^d} \e^{-|k|^2 t} (\Id \ot \tau)\big[(U_k \ot U_k^*)(1 \ot U_m)\big] 
=\sum_{k \in \Z^d} \e^{-|k|^2 t} (\Id \ot \tau)\big(U_k \ot U_k^*U_m)\big) \\
&= \sum_{k \in \Z^d} \e^{-|k|^2 t} U_k \tau\big(U_k^*U_m\big) 
\ov{\eqref{product-adjoint-twisted}}{=} \sum_{k \in \Z^d} \ovl{\sigma_\theta(k,-k)}\e^{-|k|^2 t} U_k \tau\big(U_{-k}U_m\big) \\
&\ov{\eqref{product-adjoint-twisted}}{=} \sum_{k \in \Z^d} \ovl{\sigma_\theta(k,-k)}\sigma_\theta(-k,m)\e^{-|k|^2 t} U_k \tau\big(U_{m-k}\big)
= \ovl{\sigma_\theta(m,-m)}\sigma_\theta(-m,m)\e^{-|m|^2 t} U_m\\
&=T_t(U_m).
\end{align*}
For any $t>0$, we deduce with \eqref{Kernel-operator} that the operator $T_t$ admits the kernel $\sigma_\theta(h_t)$.

Since an injective $*$-homomorphism between $\C^*$-algebras is isometric by \cite[Corollary II.2.2.9 p.~61]{Bla06}, we obtain that 
$$
\norm{T_t}_{\cb,\L^1(\T^d_\theta) \to \L^\infty(\T^d_\theta)}
\ov{\eqref{NC-Dunford-Pettis-1}}{=} \norm{\sigma_\theta(h_t)}_{\L^\infty(\T^d_\theta) \otvn \L^\infty(\T^d_\theta)^\op}
=\norm{h_t}_{\L^\infty(\T^d)}.
$$
We conclude by using the properties of the classical heat kernel of the $d$-dimensional torus $\T^d$. Indeed, as explained in Section \ref{Sec-Lie-groups} we have 
$$
\norm{h_t}_{\L^\infty(\T^d)}
\approx \frac{1}{t^\frac{d}{2}} , \quad 0 < t \leq 1
$$
(also note the estimates of \cite{Ber76}). 


Observe that the associated Dirac operator $D=\begin{bmatrix} 
0 & \partial^* \\ 
\partial & 0 
\end{bmatrix}$ defined in \eqref{Hodge-Dirac-I} acts on the Hilbert space $\L^2(\T_{\theta}^d) \oplus_2 \ell^2_d(\L^2(\T_{\theta}^d))$. This operator is not the same that the Dirac operator $\scr{D} \ov{\mathrm{def}}{=} -\i \delta_\mu \ot \gamma^u$ of \cite[(B.6) p.~147]{EcI18} (see also \cite[Definition 12.14 p.~545]{GVF01}). However, the spectral dimension of $D$ is also equal to $d$. Indeed, by \cite[Proposition B.3 p.~149]{EcI18}, we have $\tr \e^{-t \scr{D}^2} \underset{t \to 0}{\sim} \frac{c}{t^{\frac{d}{2}}}$ for some constant $c > 0$. Moreover, by \cite[Exercise A.39]{Var06}, the operator $\scr{D}^2$ identifies to the operator $\Delta \ot \Id_{2^m}$ where $m=\lfloor \frac{d}{2}\rfloor$. Finally, it suffices to observe that 
$D^2
\ov{\eqref{carre-de-D}}{=} \begin{bmatrix} 
\partial^* \partial & 0 \\ 
0 & \partial \partial^*
\end{bmatrix}
=\begin{bmatrix} 
\Delta & 0 \\ 
0 & \partial \partial^*
\end{bmatrix}$ and to remember that the non-zero part of the spectrum of a product $T^*T$ is the same as
the non-zero part of the spectrum of $TT^*$, see Theorem \ref{Th-unit-eq}.
\end{proof}





\subsection{Group von Neumann algebras}
We will use the notations of the book \cite{ArK22}. Let $G$ be a discrete group and $(T_t)_{t \geq 0}$ be a Markovian semigroup of Fourier multipliers on the group von Neumann algebra $\VN(G)$ generated by the operators $\lambda_s \co \ell^2_G \to \ell^2_G$ where $s \in G$. These semigroups admit a nice description. Indeed, by \cite[Proposition 3.3 p.~33]{ArK22}, there exists a unique real-valued conditionally negative definite function $\psi \co G \to \R$ satisfying $\psi(e) = 0$ such that 
\begin{equation}
\label{divers-100}
T_t(\lambda_s) 
= \e^{-t \psi(s) \lambda_s}, \quad t \geq 0,\quad s \in G
\end{equation}
and there exists a \textit{real} Hilbert space $H$ together with a mapping $b_\psi \co G \to \mathbb{C}$ and a homomorphism $\pi \co G \to \mathrm{O}(H)$ such that the $1$-cocycle law holds 
$\pi_s(b_\psi(t))
=b_\psi(st)-b_\psi(s),
$ 
for any $s,t \in G$ and such that $\psi(s)=\|b_\psi(s)\|_H^2$ for any $s \in G$.

\paragraph{Rapid decay property} We need background on the rapid decay property. A length function on a group $G$ with neutral element $e$ is a function $|\cdot| \co G \to \R^+$ which satisfies the conditions $|e| = 0$, $|s| = |s^{-1}|$,  $|ss'| \leq |s| + |s'|$ for all $s,s' \in G$. If $G$ is a finitely generated discrete group generated by a finite symmetric set $S$, that is $S^{-1}=S$ and $G=\cup_{n \geq 0} S^n$ then the function $G \to \R^+$, $s \mapsto \min\{n \geq 1 : x \in S^n\}$ is a length function of $G$ called the word length (with respect to $S$).

Let $G$ be a discrete group. Recall that $G$ has rapid decay of order $r$ with respect to a length function $|\cdot|$ if we have an estimate
\begin{equation}
\label{Rapid-decay}
\norm{x}_{\L^\infty(\VN(G))}
\lesssim (n+1)^r \norm{x}_{\L^2(\VN(G))}
\end{equation}
for any element $x=\sum_{|s| \leq n} x_s \lambda_s$ of the group von Neumann algebra $\VN(G)$ of $G$. Indeed, by \cite[Lemma 1.1.4]{Jol90}, a finitely generated group $G$ with generating finite symmetric set $S$ has rapid decay property with respect to the word length as soon as it has rapid decay property for any other length. See also \cite[Theorem 2.2.12 p.~51]{Bat21} for other classical characterizations.  This property has its origin in the famous paper \cite{Haa78} where it is showed that free groups have this property with order $r=\frac{3}{2}$. The explicit definition of this property is due to Jolissaint \cite{Jol90}. We refer to the survey \cite{Cha17} and to \cite[Section 2.2]{Bat21} for more information. Now, we introduce a variant of the rapid decay property of order $s$. We say that $G$ has sphere rapid decay property (SRD) of order $r$ with respect to a length function $|\cdot|$ if we have the estimate \eqref{Rapid-decay} 
for any element $x=\sum_{|s| = n} x_s \lambda_s$ of the group von Neumann algebra $\VN(G)$ of $G$. A folklore argument shows that if $G$ has sphere rapid decay of order $r$ then $G$ has rapid decay of order $r +\frac{1}{2}$. Conversely, it is obvious that RD of order $s$ implies SRD of order $s$.  

Let $G$ be a discrete group satisfying the sphere rapid decay property of order $r$ with respect to a length function $|\cdot|$ which is a conditionally negative definite function. A particular case of \cite[Lemma 1.3.1]{JuM10} says that the associated Markovian semigroup $(T_t)_{t \geq 0}$ of Fourier multipliers on the von Neumann algebra $\VN(G)$ defined by 
\begin{equation}
\label{def-semi-3}
T_t(\lambda_s)
\ov{\mathrm{def}}{=} \e^{-t|s|} \lambda_s, \quad s \in G, t \geq 0,
\end{equation}
satisfies
\begin{equation}
\label{ine-coul-bis}
\norm{T_t}_{\L^2(\VN(G)) \to \L^\infty(\VN(G))} 
\lesssim \frac{1}{t^{\frac{2r+1}{2}}}, \quad 0 < t \leq 1.
\end{equation}

\begin{example} \normalfont
\label{Example-free}
If $n \geq 1$, consider the free group $G=\mathbb{F}^n$ with $n$ generators denoted by $c_1, \ldots, c_n$. Any element $s$ different from the identity element $e$ of this group admits a unique decomposition of the form
\begin{equation}
\label{factor}
s=
c_{i_1}^{k_1}c_{i_2}^{k_2} \cdots c_{i_\ell}^{k_\ell},
\end{equation}
where $\ell \geq 1$ is an integer, each $i_j$ belongs to $\{1,\ldots, n\}$, each $k_j$ is a non-zero integer, and  $i_j \not = i_{j+1}$ if $\ell \geq 2$ and $1 \leq j \leq \ell-1$. By definition, the (natural) word-length of $s$ is defined as
\begin{equation}
\label{word-length}
\vert s \vert
\ov{\mathrm{def}}{=} \vert k_1 \vert +\cdots + \vert k_\ell \vert.
\end{equation}
We also let $\vert e \vert \ov{\mathrm{def}}{=} 0$. This is the number of factors in the previous decomposition of the element $s$. By \cite{Haa78}, this group admits the sphere rapid decay property with order $r=1$ with respect to this length $|\cdot|$. This estimate is sharp by \cite{Nic17}. So we obtain $\norm{T_t}_{\L^2(\VN(G)) \to \L^\infty(\VN(G))} \lesssim \frac{1}{t^{\frac{3}{2}}}$ for any $0 < t \leq 1$. See also \cite[Corollary 5.2 (1) p.~3359]{You20} for a related result. 
 If $n=1$ then $G=\Z$ and the associated semigroup $(T_t)_{t \geq 0}$ identifies to the Poisson semigroup of operators acting on $\L^\infty(\mathbb{T})$. By interpolation \cite[Lemma 1]{Cou90}, the previous estimate gives $\norm{T_t}_{\L^1(\T) \to \L^\infty(\T)} \lesssim \frac{1}{t^3}$ for any $0 < t \leq 1$. 
\end{example}



\paragraph{Amenability and completely bounded local dimension} Let $G$ be a discrete group and let $(T_t)_{t \geq 0}$ be a Markovian semigroup of Fourier multipliers on the group von Neumann algebra $\VN(G)$ as in \eqref{divers-100}. If for each $0 < t \leq 1$ the operator $T_t$ induces a completely bounded operator $T_t \co \L^2(\VN(G)) \to \L^\infty(\VN(G))$ then the group $G$ is amenable. 
Here, we develop \cite[Remark 2.5]{GJP17} and we will generalize this result in Theorem \ref{thm-inj}. Our only contribution is to clarify a point (the authors unfortunately uses a sequence of integrable functions). Indeed, for any $t > 0$, using the isometry $\ell^2_G \to \CB(\L^2(\VN(G)),\L^\infty(\VN(G)))$, $\varphi \mapsto M_\varphi$ from the Hilbert space $\ell^2_G$ onto the space of completely bounded Fourier multipliers from $\L^2(\VN(G)$ into $\L^\infty(\VN(G))$ of \cite[Theorem 2.3]{GJP17} we obtain $
\norm{\e^{-t\psi}}_{\ell^2_G}
=\norm{T_t}_{\cb,\L^1(\VN(G)) \to \L^2(\VN(G))} =\norm{T_t}_{\cb,\L^2(\VN(G)) \to \L^\infty(\VN(G))}
< \infty$. The weak* continuity of the semigroup $(T_t)_{t \geq 0}$ on the von Neumann algebra $\VN(G)$ implies that for any $s \in G$ we have  $
\e^{-t\psi(s)}
\xra[t \to 0]{} 1$. 
Since each map $T_t \co \VN(G) \to \VN(G)$ is completely positive, the function $\e^{-t\psi} \co G \to \mathbb{C}$ is a (continuous) positive definite function by \cite[Proposition 5.4.9 p.~184]{KaL18}. 
Recall that by a combining \cite[Proposition 18.3.5 p.~357]{Dix77} and \cite[Proposition 18.3.6 p.~358]{Dix77} a locally compact group $G$ is amenable if and only if the function 1 is the uniform limit over every compact set of square-integrable continuous positive-definite functions\footnote{\thefootnote. We can replace ``square-integrable continuous positive-definite functions'' by ``continuous positive definite functions of compact support''.}. So we conclude that the group $G$ is amenable.

In particular, by Lemma \ref{Lemma-interpolation}, if the completely bounded local Coulhon-Varopoulos dimension of a Markovian semigroup of Fourier multipliers acting on the von Neumann algebra $\VN(G)$ is finite then the group $G$ is amenable.

\begin{example} \normalfont
\label{Ex-free-groups}
This observation applies to the noncommutative Poisson semigroup $(T_t)_{t \geq 0}$ on the free group factor $\VN(\mathbb{F}_n)$ with $n \geq 2$ defined in \cite[Section 10]{JMX06} by the word-length \eqref{word-length} and the formula \eqref{def-semi-3}. This semigroup does not admit a local estimate $\norm{T_t}_{\cb,\L^2(\cal{M}) \to \L^\infty(\cal{M})} \lesssim \frac{1}{t^{\frac{d}{2}}}$ for $0 < t \leq 1$ and some $d>0$ but admits the estimate $\norm{T_t}_{\L^2(\cal{M}) \to \L^\infty(\cal{M})} \lesssim \frac{1}{t^{\frac{3}{2}}}$ for any $0 < t \leq 1$ by Example \ref{Example-free}. Note that it is possible to give an other elementary proof of the first fact. We learned this folklore argument from Sang-Gyun Youn. One of the referees advised us to include it here. More precisely, we have the quantitative following result.
\end{example}

\begin{prop}
\label{Prop-7-4}
Let $n \geq 1$ be an integer. Consider the noncommutative Poisson semigroup $(T_t)_{t \geq 0}$ on the free group factor $\VN(\mathbb{F}_n)$. For any $0 < t \leq \frac{\log(2n-1)}{2}$, the map $T_t \co \VN(\mathbb{F}_n) \to \VN(\mathbb{F}_n)$ does not induce a completely bounded operator $T_t \co \L^1(\VN(\mathbb{F}_n)) \to \L^\infty(\VN(\mathbb{F}_n))$.
\end{prop}

\begin{proof}
Suppose that for some $t >0$ we have a completely bounded map $T_t \co \L^1(\VN(\mathbb{F}_n)) \to \L^\infty(\VN(\mathbb{F}_n))$. There exists by the identification defined in \eqref{Effros-Ruan} an element $K_{t} \in \VN(\mathbb{F}_n) \otvn \VN(\mathbb{F}_n)$ such that $T_{t}=T_{K_{t}}$ with
$$
\norm{T_t}_{\cb,\L^1(\VN(\mathbb{F}_n)) \to \L^\infty(\VN(\mathbb{F}_n))}
\ov{\eqref{NC-Dunford-Pettis-1}}{=} \norm{K_t}_{\VN(\mathbb{F}_n) \otvn \VN(\mathbb{F}_n)^\op}.
$$
Since we have an inclusion $\VN(\mathbb{F}_n) \subset \L^2(\VN(\mathbb{F}_n)$, the element $K_t$ identifies to an element $\sum_{r,v \in \mathbb{F}_n} a_{r,v} \lambda_r \ot \lambda_v$ (<<Fourier series>>) of the Hilbert space $\L^2(\VN(\mathbb{F}_n) \otvn \VN(\mathbb{F}_n)^\op)$. Consequently, for any element $s \in \mathbb{F}_n$ and any $t \geq 0$ we have
\begin{align*}
\MoveEqLeft
\e^{-t|s|} \lambda_s
\ov{\eqref{def-semi-3}}{=} T_t(\lambda_s)
\ov{\eqref{Kernel-operator}}{=} 
\sum_{r,v \in \mathbb{F}_n} a_{r,v} \lambda_r \ot \tau(\lambda_v\lambda_s) 
=\sum_{r \in \mathbb{F}_n} a_{r,s} \lambda_r.
\end{align*}
We deduce the equality $K_t=\sum_{s \in \mathbb{F}_n} \e^{-t|s|} \lambda_s \ot \lambda_{s^{-1}}$ in the Hilbert space $\L^2(\VN(\mathbb{F}_n) \otvn \VN(\mathbb{F}_n))$. 
Recall that that the number of elements of length $k$ in the free group is equal to $2n(2n-1)^{k-1}$, see e.g.~\cite[Exercise 3 p.~169]{Coh81}. Consequently, we have
\begin{align*}
\MoveEqLeft
\norm{K_{t}}_{\L^2(\VN(\mathbb{F}_n))}         
= \bigg(\sum_{s \in \mathbb{F}_n} \e^{-2t|s|}\bigg)^{\frac{1}{2}} 
=\bigg(\sum_{k=0}^{\infty} \sum_{s \in \mathbb{F}_n,|s|=k} \e^{-2tk}\bigg)^{\frac{1}{2}} \\
&=\bigg(\sum_{k=0}^{\infty} \e^{-2tk} 2n(2n-1)^{k-1}\bigg)^{\frac{1}{2}}
\geq \bigg(\sum_{k=0}^\infty \big[(2n-1)\e^{-2t}\big]^k\bigg)^{\frac{1}{2}}
\end{align*}
where we used the inequality $2n \geq 2n-1$. Now, observe that we have $(2n-1)\e^{-2t} \geq 1$ if and only if $t \leq \frac{\log(2n-1)}{2}$. So the geometric series $\sum_{k \geq 0} \big[(2n-1)\e^{-2t}\big]^k$ diverges if $t \leq \frac{\log(2n-1)}{2}$ and consequently the map $T_t$ does not induce a completely bounded operator $T_t \co \L^1(\VN(\mathbb{F}_n)) \to \L^\infty(\VN(\mathbb{F}_n))$ for any $t \leq \frac{\log(2n-1)}{2}$. 
\end{proof}



Now we give a sufficient condition for the estimate \eqref{cbRnpq}. 

\begin{prop}
Assume that the real Hilbert space $H$ is finite-dimensional with dimension $d$, that the set $\{s \in  G : \psi(s) = 0\}$ is finite and that $\inf_{b_\psi(s) \neq b_\psi(t)} \norm{b_\psi(s) - b_\psi(t)}_H^2 >0$. Then
\begin{equation}
\label{Tt-G}
\norm{T_t}_{\cb, \L^1(\VN(G)) \to \L^\infty(\VN(G))}
\lesssim \frac{1}{t^{\frac{d}{2}}}, \quad 0 < t \leq 1. 
\end{equation}
\end{prop}

\begin{proof}
By \cite[Lemma 5.8]{JMP14}, the semigroup $(T_t)_{t \geq 0}$ satisfies the estimate 
\begin{equation}
\label{GAFA}
\norm{T_t}_{\cb, \L^1_0(\VN(G)) \to \L^\infty(\VN(G))} 
\lesssim \frac{1}{t^{\frac{d}{2}}}, \quad t>0.
\end{equation}
Note that the subspace $\Fix (T_{t})_{t \geq 0}$ defined in \eqref{def-fix} is equal to $\Span \{\lambda_s : \psi(s)=0\}$ which is finite-dimensional. Consequently, the restriction of each operator $T_t$ on the subspace $\Fix (T_{t})_{t \geq 0}$ into $\L^\infty(\VN(G))$ is completely bounded and its completely bounded norm is bounded above by a constant independent of $t$. By a direct sum argument \cite[p.~26]{BLM04}, we conclude that we have a well-defined completely bounded map $T_t \co \L^1(\VN(G)) \to \L^\infty(\VN(G))$ satisfying \eqref{Tt-G}.
\end{proof}

The assumptions for this result are satisfied for the Heat semigroup on the $d$-dimensional classical torus $\T^d$ by \cite[p.~215]{ArK22}

Finally, we give a lot of interesting examples by using one of the main result of \cite{Arh21}.

\begin{thm}
\label{th-ame}
Let $G$ be an amenable discrete group satisfying the sphere rapid decay property of order $r$ with respect to a length function $|\cdot|$ which is a conditionally negative definite function. Then the associated Markovian semigroup $(T_t)_{t \geq 0}$ satisfies 
\begin{equation}
\label{ine-coul-bla}
\norm{T_t}_{\cb,\L^1(\VN(G)) \to \L^\infty(\VN(G))} 
\lesssim \frac{1}{t^{2r+1}}, \quad 0 < t \leq 1.
\end{equation}
\end{thm}

\begin{proof}
By, \cite[Lemma 1.3.1]{JuM10}, we have the estimate \eqref{ine-coul-bis}. Using interpolation \eqref{Rnpq}, we obtain the estimate \eqref{ine-coul-bla} without the subscript $\cb$. However, the result \cite[Theorem 4.10]{Arh21} says that the completely boundedness is free. So, we get \eqref{ine-coul-bla}.
\end{proof}

This means that the completely bounded local Coulhon-Varopoulos is less than $4r+2$.

\begin{example} \normalfont
Let $G$ be a finitely generated amenable discrete group. According to \cite[Proposition B]{Jol90}, $G$ has rapid decay precisely when it has polynomial growth, a concept thoroughly explored in \cite[Chapter VI]{DLH00}.
\end{example}

\section{Completely bounded local dimension and injective von Neumann algebras}
\label{sec-dimension}

In this section, we consider an infinite-dimensional finite von Neumann algebra $\cal{M}$ equipped with a normal finite faithful trace $\tau$ and a sub-Markovian semigroup $(T_t)_{t \geq 0}$ of operators acting on $\cal{M}$ with associated Dirichlet form $\cal{E}$. Let $A_2$ be the associated positive selfadjoint operator which is the opposite of the infinitesimal generator on the Hilbert space $\L^2(\cal{M})$.


Assume that the spectrum of $A_2$ is discrete, i.e.~its points are isolated eigenvalues $\lambda_1,\lambda_2,\ldots$ of finite multiplicity, repeated in increasing order according to their multiplicities. With the eigenvalue counting function
$
N(\lambda)
\ov{\mathrm{def}}{=} \card\{k \geq 1: \lambda_k \leq \lambda\},
$ 
the spectral growth rate of $\cal{E}$ is introduced in \cite[Definition 3.9]{CiS17} by the formula
$
\Omega(\cal{E})
\ov{\mathrm{def}}{=} \limsup_{n \to \infty} N(n)^{\frac{1}{n}}$. The Dirichlet form $\cal{E}$ is said to have exponential growth if $\Omega(\cal{E})>1$ and subexponential growth if $\Omega (\cal{E})=1$.

Generalizing slightly \cite[Definition 10.8 p.~450]{GVF01}, we say that that a possibly kernel-degenerate compact spectral triple $(\cal{A},H,D)$ is $\theta$-summable if the operator $\e^{-tD^2}$ is trace-class for any $t>0$.

The following observation is in the same spirit that the result \cite[Theorem 1 p.~407]{Con94} (see also \cite{Con89}) which says that the reduced $\mathrm{C}^*$-algebra $\mathrm{C}^*_\lambda(G)$ of a \textit{non-amenable} discrete group $G$ does not admit any finitely summable K-cycle. Recall that the von Neumann algebra $\VN(G)$ of a discrete group $G$ is injective if and only if $G$ is amenable. In the next result, the selfadjoint operator $D$ is defined as in \eqref{Def-D-psi}.

\begin{thm}
\label{thm-summable}
Let $\cal{M}$ be a finite von Neumann algebra equipped with a normal finite faithful trace. Consider a symmetric sub-Markovian semigroup $(T_t)_{t \geq 0}$ of operators acting on $\cal{M}$. If the triple $(\cal{M},\L^2(\cal{M}) \oplus_2 \cal{H},D)$ is $\theta$-summable then the von Neumann algebra $\cal{M}$ is injective.
\end{thm}

\begin{proof}
The operator $\e^{-tD^2}$ is trace-class for any $t > 0$. 
The computation \eqref{carre-de-D} gives 
$$
\e^{-t D^2}
=\begin{bmatrix} 
\e^{-tA_2} & 0 \\ 
0 & \e^{-t\partial \partial^*}
\end{bmatrix}
=\begin{bmatrix} 
T_{t} & 0 \\ 
0 & \e^{-t\partial \partial^*}
\end{bmatrix}
, \quad t \geq 0. 
$$
So the semigroup $(T_t)_{t \geq 0}$ induces a Gibbs semigroup of operators on the Hilbert space $\L^2(\cal{M})$. Using the characterization \cite[Lemma 3.13]{CiS17}, we see that the operator $A_2$ has discrete spectrum and that the associated Dirichlet form $\cal{E}$ has subexponential spectral growth. In \cite[Theorem 3.15]{CiS17}, it is proved that if there exists a Dirichlet form on $\L^2(\cal{M})$ having subexponential spectral growth, then the von Neumann algebra $\cal{M}$ is amenable. The proof is complete since it is well-known \cite[Theorem IV.2.5.4 p.~366]{Bla06} that this condition implies the injectivity of $\cal{M}$.
\end{proof}

\begin{remark} \normalfont
With Lemma \ref{Lemma-theta-summable}, we observe that we can use the result if the spectral triple $(\cal{M},\L^2(\cal{M}) \oplus_2 \cal{H},D)$ is finitely summable, i.e.~if the operator $|D|^{-p}$ is trace-class for some $p >0$.
\end{remark}

\begin{thm}
\label{thm-inj}
Let $\cal{M}$ be a finite von Neumann algebra equipped with a normal finite faithful trace. Consider a symmetric sub-Markovian semigroup $(T_t)_{t \geq 0}$ of operators acting on $\cal{M}$.  Suppose that the estimate $
\norm{T_t}_{\cb,\L^1(\cal{M}) \to \L^\infty(\cal{M})} \lesssim \frac{1}{t^{\frac{d}{2}}}$ for some $d>0$ and any $0 < t \leq 1$ holds. Then the von Neumann algebra $\cal{M}$ is injective. 
\end{thm}

\begin{proof}
By Lemma \ref{Lemma-Gibbs}, the semigroup induces a Gibbs semigroup $(T_{t})_{t \geq 0}$ of operators on the Hilbert space $\L^2(\cal{M})$. The end of the proof is identical to the one of Theorem \ref{thm-summable}.
\end{proof}

So we can recover a part of Example \ref{Ex-free-groups} with this result.

\begin{example} \normalfont
We refer to \cite[Corollary 5.2 p.~3359]{You20} for examples of semigroups on free orthogonal quantum groups $\mathrm{O}_n^+$ if $n \geq 3$ and on free permutation quantum groups $\mathrm{S}_n^+$ if $n \geq 5$ satisfying the estimate \eqref{Varo-Coulhon-dim-AVN}. Note that if $n \geq 3$ the von Neumann algebra $\L^\infty(\mathrm{O}_n^+)$ is a \textit{non-injective} factor of type $\textrm{II}_1$ by \cite[Theorem 3.18]{Bra17}. So the estimate \eqref{Varo-Coulhon-dim-cb} cannot be satisfied in this case. On the other hand, the von Neumann algebra $\L^\infty(\mathrm{O}_2^+)$ is injective, again by \cite[Theorem 3.18]{Bra17} or \cite[Example 3.18 p.~324]{CiS17} for another proof.
\end{example}

\section{Future directions and open questions}
\label{sec-future}

\paragraph{Completely bounded local Coulhon-Varopoulos dimension vs spectral dimension} The local completely bounded Coulhon-Varopoulos dimension and the spectral dimension may not be equal, as shown in this study. Delving deeper into this newfound phenomenon would be enlightening. Specifically, find other examples where these dimensions are different, in other contexts than the one of $\RCD^*(K,N)$-spaces, would be probably instructive. The ideal would be to identify a necessary and sufficient condition for the equality of both quantities.
 
For example, it is not clear if the same phenomenon is true in the context of \textit{intrinsic} semigroups of operators associated to Schr\"odinger operators on $\R^n$ (or a Riemannian manifold). The operators of such a semigroup act on the Hilbert space $\L^2(\R^n, \nu)$ where $\nu$ is a probability measure. We refer for more information to \cite{RoS20}, \cite{Stu93}, \cite{Sim82} to \cite{Sim18}, \cite{Sim19} for the history of this topic and finally to \cite[Example 10.10]{Cip16} for the connection with spectral triples.

In the context of strict elliptic operators on a non-empty connected open bounded subset $\Omega$ of $\R^n$, the same question arises. 
Without seeking the utmost generality, we consider in this framework \textit{real} $\mathrm{C}^\infty$ functions $a_{ij}$ for any  $1 \leq i,j \leq n$ defined on $\Omega$ with $a_{ij}=a_{ji}$. 
We also make an assumption of strict ellipticity: there exist some constant $\delta$ and some positive continuous function $\nu > 0$ on $\Omega$ such that 
\begin{equation}
\label{}
\delta|\xi| \leq \sum_{i,j=1}^{n} a_{ij}(x)\xi_i\xi_j 
\leq \nu(x) |\xi|,\quad \xi \in \R^n,x \in \Omega
\end{equation}
where $|\xi| \ov{\mathrm{def}}{=} (\sum_{i=1}^{n} \xi_i^2)^{\frac{1}{2}}$. By \cite[p.~10]{Dav89} and \cite[p.~111]{FOT11} (see also \cite[p.~41]{MaR92}), we can consider the regular Dirichlet form $\cal{E} \co \W_{0}^1(\Omega) \times \W_{0}^1(\Omega) \to \mathbb{C}$  on the Hilbert space $\L^2(\Omega)$ defined by
\begin{equation}
\label{Dirichlet-elliptic}
\cal{E}(f,g)
=\int_\Omega \la \nabla f, A\nabla g \ra \d x 
\quad \text{ i.e.} \quad \cal{E}(f,g)
=\sum_{i,j=1}^{n} \int_{\Omega} a_{ij} \frac{\partial f}{\partial x_i} \frac{\partial g}{\partial x_j} \d x,\quad f,g \in \W_{0}^{1}(\Omega).
\end{equation}
This form is obviously strongly local. We can introduce the associated positive selfadjoint operator $A_2$ on $\L^2(\Omega)$. 
It is often denoted by $-\sum_{i,j=1}^{n} \frac{\partial }{\partial x_i} \big(a_{ij}\frac{\partial }{\partial x_j}\big)$ and determines a symmetric sub-Markovian semigroup $(T_t)_{t \geq 0}$ defined by $T_t \ov{\mathrm{def}}{=} \e^{-tA_2}$ where $t \geq 0$. It is stated in \cite[p.~82]{Dav89} that each $T_t$ is a convolution operator by a continuous kernel $K_t \co \Omega \times \Omega \to \R$. 
The result \cite[Theorem 3.2.7 p.~89]{Dav89}, gives the estimate
\begin{equation}
\label{estimate-elliptic}
0 \leq K_t(x,y) 
\lesssim \frac{1}{t^{\frac{n}{2}}} \exp\bigg(-c\frac{\d(x,y)^2}{t}\bigg),\quad t >0, x,y \in \Omega
\end{equation}
where the distance $d$ is given by 
\begin{equation}
\label{Distance-elliptic}
\d(x,y) 
\ov{\mathrm{def}}{=} \sup \bigg\{ |f(x)-f(y)|: f \in \C^\infty(\Omega), \text{ bounded with }\sum_{i,j=1}^n a_{ij}\frac{\partial f}{\partial x_i}\frac{\partial f}{\partial x_j} \leq 1\bigg\}.
\end{equation}
This distance was introduced by Davies in its classical paper \cite[p.~328]{Dav87}. See also \cite[Theorem 6.10 p.~171 and p.~194]{Ouh05} for strongly related results. This kind of estimate has a long history and we refer to we refer to \cite{Aro17} for a (too) brief history in this topic and to the books \cite{Dav89}, \cite{Ouh05}, \cite{Stro12} and finally to the surveys \cite{PoE84} and \cite[Section 5]{Dav97} and references therein for more information. 

In particular, we obtain that the local Coulhon-Varopoulos dimension is less than $n$. We does not know if this dimension and the spectral dimension of the associated Dirac operator $D$ defined in \eqref{Hodge-Dirac-I} are always equal to $n$ in this setting. 



\begin{remark} \normalfont
It is apparent that the <<Davies distance>> defined in \eqref{Distance-elliptic}  coincide with the intrinsic pseudo-metric of the associated Dirichlet form and the Connes spectral pseudo-distance.
\end{remark}

The same issue comes up for the most studied fractal set which is the Sierpi\'nski gasket $\SG$ introduced in \cite{Sie15}. Loosely speaking, this set is obtained from an equilateral triangle by removing its middle triangle and by reiterating this procedure (see Figure 1). We refer to the book \cite{Kig01}, to the papers \cite{Kig08}, \cite{Kaj12} and to the surveys \cite{Kaj21} and \cite{Kaj13} for more information.

\begin{center}
\includegraphics[scale=0.45]{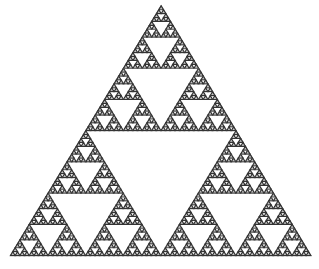}

Figure 1: the Sierpi\'nski gasket
\end{center}

It is known that the Hausdorff dimension (with respect to the euclidean distance) of $\SG$ is equal to $\alpha \ov{\mathrm{def}}{=} \log_2 3$. We let $q_1 \ov{\mathrm{def}}{=} (0,0)$, $q_2 \ov{\mathrm{def}}{=} (1,0)$ and $q_3 \ov{\mathrm{def}}{=} (\frac{1}{2},\frac{\sqrt{3}}{2})$. We suppose that $V_0 \ov{\mathrm{def}}{=} \big\{q_1,q_2,q_3 \big\}$ is the set of the three vertices of the biggest equilateral triangle of Figure 1. For any $1 \leq i \leq 3$, consider the contractions $F_i \co \R^2 \to \R^2$ defined by $F_i(x) \ov{\mathrm{def}}{=} \frac{x+q_i}{2}$ where $x \in \R^2$. For $m \in \N^*$, define the set $V_m$ inductively by $V_m \ov{\mathrm{def}}{=} F_1(V_{m-1}) \cup F_2(V_{m-1}) \cup F_3(V_{m-1})$. We have $V_m \subset V_{m+1}$ for any $m \geq 0$. The Sierpinski gasket $\SG$ is defined to be the closure  $
\SG 
\ov{\mathrm{def}}{=} \ovl{\cup_{m=0}^{\infty} V_m}$. 
Now, we describe an analogue of the Dirichlet form \eqref{Dirichlet-Riemann} by constructing a sequence of bilinear forms. For each $m \in \N$ and any function $f \co \cup_{m=0}^{\infty} V_m \to \R$, let
$$
\cal{E}_m(f) 
\ov{\mathrm{def}}{=} \sum_{x,y \in V_m, |x-y|=2^{-m}} \frac{1}{2}\bigg(\frac{5}{3}\bigg)^m\big(f(x)-f(y)\big)^2.
$$
The sequence $(\cal{E}_m(f))_{m \geq 0}$ is non-decreasing. Consequently, the limit $
\cal{E}(f) 
\ov{\mathrm{def}}{=} \lim_{m \to \infty} \cal{E}_m(f)$  
exists (possibly infinite). By polarization, we define $\cal{E}(f,g)$ for any functions $f,g \co \cup_{m=0}^{\infty} V_m \to \R$. Let  $
\F 
\ov{\mathrm{def}}{=} \{f : f \text{ is a function on } \cup_{m=0}^{\infty} V_m \text{ with } \cal{E}(f) < \infty \}$. 
According to \cite[Theorem 2.2.6]{Kig01}, every function $f \in \F$ uniquely extends to a continuous function on the Sierpinski gasket $\SG$. In other words, we have an inclusion $\F \subset \C(\SG)$. The couple $(\cal{E},\F)$ is called the standard resistance form on the Sierpi\'nski gasket.

If $\mu$ is a finite Borel measure on $\SG$ with full support then by \cite[Theorem 2.7]{Kaj13} $(\cal{E},\F)$ is a strongly local regular Dirichlet form on $\L^2(\SG,\mu)$, and its associated sub-Markovian semigroup $(T_t)_{t \geq 0}$ on the Hilbert space $\L^2(\SG,\mu)$ admits a \textit{continuous} <<heat kernel >> $K_t \co \SG \times \SG \to \R$. 

The choice of the measure $\mu$ on $\SG$ has a crucial importance. If we consider the Kusuoka measure $\nu$ defined in \cite{Kus89} on the Sierpi\'nski gasket $\SG$, we obtain a spectral triple by \cite[Example 5.1 (i)]{HKT15}. Furthermore, \cite[Theorem 7.2]{Kaj13} gives the estimate $
\int_{\SG} K_t(x,x) \d \mu (x) 
 \underset{0}{\sim} \frac{c}{t^\frac{d}{2}}$ 
for some positive constant $c$ and where $d$ is the Hausdorff dimension of $\SG$ with respect to the harmonic geodesic distance $\rho_{\cal{H}}$. So using Proposition \ref{prop-trace-spectral triple}, we obtain that the spectral dimension is equal to $d$. It is known that $d \leq 2\log_{25/3} 5$ by \cite[Theorem 6.1]{Kaj13}. Therefore, it naturally leads us to consider the following problem, implicit in \cite{Kaj13}.

\begin{prob}
To compute the exact value of $d$.
\end{prob}

Furthermore, it is established \cite[Theorem 6.3]{Kig08} (see also \cite[Corollary 5.3]{Kaj13}) that we have the following Gaussian estimates of the heat kernel
\begin{equation}
\label{gaussien-Sierpinski}
K_t(x,y)
\approx \frac{1}{V(x,\sqrt{t})} \exp\bigg(-c\frac{\rho_{\cal{H}}(x,y)^2}{t}\bigg), \quad x,y \in \SG,  0 < t \leq 1
\end{equation} 
where $\rho_{\cal{H}}$ is a distance called harmonic geodesic metric and $V(x,t)$ is the measure $\nu(B(x,t))$ of the ball $B(x,t)$ with center $x$ and radius $t$ for the distance $\rho_{\cal{H}}$. But this does not imply that we obtain \eqref{sub-gaussian-intro} with $\beta=2$ and $\alpha=d$. In this context, we introduce the following conjecture.
 The proof should be fairly straightforward.

\begin{conj}
The local Coulhon-Varopoulos dimension of the semigroup $(T_t)_{t \geq 0}$ is strictly greater than the spectral dimension $d$.
\end{conj}




\begin{remark} \normalfont
Consider the normalized version $\mu_0$ of the $\log_2 3$-dimensional Hausdorff measure (see \cite[Appendix A]{Kaj13} for more information). Note that this measure and the Kusuoka measure are singular by \cite{Kus89}. A classical result of Barlow and Perkins \cite[Theorem 1.5]{BaP88} (see also \cite[p.~2]{Gri03}) says that the kernel of the associated semigroup satisfies \eqref{sub-gaussian-intro} by taking $\alpha=\frac{\log_2 3}{2}$, i.e.~the Hausdorff dimension  of the Sierpi\'nski gasket $\SG$ divided by two, and the walk dimension $\beta=\log_2 5$. However, we do not obtain a spectral triple since this measure is not energy dominant, see \cite{Kus89} and \cite{BST99}. The problem is the boundedness of commutators $[D,f]$. 
\end{remark}

Finally, we will explore the case of the Ornstein-Uhlenbeck semigroup in another paper.

\paragraph{Examples with finite local completely bounded local Coulhon-Varopoulos dimension} 
It would be intersting to identify additional examples of semigroups acting on noncommutative spaces with a \textit{finite} completely bounded  local dimension and to determine the precise value of this dimension.

\paragraph{Completely bounded local Coulhon-Varopoulos dimension in the $\sigma$-finite case} The natural setting of noncommutative probabilities is the one of $\sigma$-finite von Neumann algebras. Let $\cal{M}$ be a $\sigma$-finite von Neumann algebra equipped with a normal faithful state $\varphi$ with density operator $D_\varphi$ and modular group $(\sigma_t^\varphi)_{t \in \R}$. In this context, we can again define notions of noncommutative $\L^p$-space, <<symmetric sub-Markovian semigroup>>, <<Markovian semigroup>> and <<Dirichlet form>>, see \cite{Cip97} and the survey \cite{Cip08}.

Consider a symmetric \textit{Markovian} semigroup $(T_{t})_{t \geq 0}$ of operators acting on $\cal{M}$. We can still construct symmetric derivations with \cite[Theorem 6.8]{Wir22} (relying crucially on the ideas of the unpublished paper \cite{JRS}) and we can still define a Hodge-Dirac operator $D$ and a notion of (completely bounded) local Coulhon-Varopoulos dimension. Using the ideas of this paper, it is not difficult to prove the following generalization of Theorem \ref{cor-trace-class}. Here $T_t \co \L^1(\cal{M}) \to \L^\infty(\cal{M})$ is defined for each $t \geq 0$ by the equality
$$
T_t(D^{\frac{1}{2}}_{\varphi} x  D^{\frac{1}{2}}_{\varphi})
=T_{t}(x), \quad x \in \L^\infty(\cal{M}).
$$
 
\begin{thm}
Let $\cal{M}$ be a $\sigma$-finite injective von Neumann algebra equipped with a normal faithful state $\varphi$. Consider a symmetric \textit{Markovian} semigroup $(T_{t})_{t \geq 0}$ of operators acting on $\cal{M}$. Suppose the estimate $
\norm{T_t}_{\cb,\L^1(\cal{M}) \to \L^\infty(\cal{M})} 
\lesssim \frac{1}{t^{\frac{d}{2}}}$ for some $d>0$ and for any $0 < t \leq 1$.  Then the operator $|D|^{-1}$ defined on $(\ker D)^\perp$ belongs to the weak Schatten space $S^{d,\infty}$. 
\end{thm}

With the notations of the paper \cite{HJX10} and $0 \leq \theta \leq 1$, we can also introduce the map $T_{t,\theta} \co \L^1(\cal{M}) \to \L^\infty(\cal{M})$ defined by
\begin{equation}
\label{Tt-theta}
T_{t,\theta}(D^{1-\theta}_{\varphi} x  D^{\theta}_{\varphi})
=T_{t}(x), \quad x \in \L^\infty(\cal{M}).
\end{equation}
It is not clear if we obtain the same notion of completely bounded local Coulhon-Varopoulos dimension. In the same spirit, we can obtain the following extension of Theorem \ref{thm-inj} with a similar proof.

\begin{thm}
\label{thm-inj-2}
Let $\cal{M}$ be a $\sigma$-finite von Neumann algebra equipped with a normal faithful state $\varphi$. Suppose that the estimate $
\norm{T_t}_{\cb,\L^1(\cal{M}) \to \L^\infty(\cal{M})} \lesssim \frac{1}{t^{\frac{d}{2}}}$ for some $d>0$ and any $0 < t \leq 1$ holds for some symmetric Markovian semigroup $(T_t)_{t \geq 0}$. Then the von Neumann algebra $\cal{M}$ is injective. 
\end{thm}

It is transparent that the derivations of \cite{JRS} and \cite{Wir22} are connected  to the notion of \textit{twisted} spectral triple in the case where the Dirichlet form has the <<$\Gamma$-regularity>>. In this situation, the derivation identifies by \cite[Theorem 7.6]{Wir22} to a \textit{twisted} derivation $\partial \co \mathfrak{A} \to \L^2(\tilde{\cal{M}})$ defined on a suitable algebra $\mathfrak{A}$ with values in a noncommutative $\L^2$-space $\L^2(\tilde{\cal{M}})$, i.e.~satisfies
\begin{equation}
\label{twisted-derivation}
\partial(xy)
=x\partial(y)+\partial(y)\sigma_{\frac{\i}{2}}^\varphi(y)
\end{equation}
for any suitable elements $x$ and $y$. This allows anyone to link this concept with a <<possibly kernel-degenerate>> extension of the notion of twisted spectral triple, introduced in \cite[Definition 3.1]{CoH08}, where the boundedness of the commutator $[D,a]$ is replaced by the one of the twisted commutators
$ [D,a]_{\sigma}
\ov{\mathrm{def}}{=} D\pi(a)-\pi(\sigma(a)) D$ for some automorphism $\sigma \co \cal{A} \to \cal{A}$.

\begin{example} \normalfont
Here we use the notations of Section \ref{NC-tori} for noncommutative tori. If $h$ is a selfadjoint element of the von Neumann algebra $\L^\infty(\T_{\theta}^2)$, then we introduce the normal faithful state $\varphi$ on $\L^\infty(\T_{\theta}^2)$ defined by $\varphi(x) \ov{\mathrm{def}}{=} \tau(x \e^{-h})$ where $x \in \L^\infty(\T_{\theta}^2)$. Using \cite[Theorem 2.11 p.~105]{Tak03} we see that the modular group $(\sigma_t^\varphi)_{t \in \R}$ is given by $\sigma_t^\varphi(x)=\e^{-\i t h}x\e^{\i t h}$ where $t \in \R$ and $x \in \L^\infty(\T_{\theta}^2)$. 

\vspace{0.2cm}

In the papers \cite{CoT11}, \cite{CoM14} and \cite{LeM19} (see also the survey \cite{FaK19}), an operator $D \ov{\mathrm{def}}{=} \begin{bmatrix}
   0  & \partial_\varphi^*  \\
   \partial_\varphi  & 0  \\
\end{bmatrix}$ is introduced, serving as the foundation for a \textit{twisted} spectral triple $(\L^\infty(\T_{\theta}^2)^\op,\L^2(\T_{\theta}^2,\varphi) \oplus \cal{H}^{(1,0)},D)$ associated with the automorphism $\L^\infty(\T_{\theta}^2)^\op \to \L^\infty(\T_{\theta}^2)^\op$, $x^\op \mapsto (\e^{-\frac{h}{2}} x\e^{\frac{h}{2}})^\op$, where we use the opposite algebra of $\L^\infty(\T_{\theta}^2)$. Indeed, these papers use instead the universal $\mathrm{C}^*$-algebra of the noncommutative two torus instead of the von Neumann algebra $\L^\infty(\T_{\theta}^2)$. Furthermore, note that 
$$
\sigma_{\frac{\i}{2}}^{\varphi^\op}(x^\op)
=\big[\sigma_{-\frac{\i}{2}}^{\varphi}(x)\big]^\op
=\big[\e^{\frac{h}{2}} x\e^{-\frac{h}{2}}\big]^\op, \quad x \in \L^\infty(\T_{\theta}^2).
$$
Actually, we believe that this example can be reformulated to align perfectly with the previously described context (possibly by enlarging $\cal{H}^{(1,0)}$). 
Consequently, we can introduce an associated semigroup $(T_t)_{t \geq 0}$. It would be interesting to solve the following problem.
\end{example}

\begin{prob}
To compute the value of the completely bounded local Coulhon-Varopoulos dimension of the semigroup $(T_t)_{t \geq 0}$ if $h \not=0$.
\end{prob}




\paragraph{Declaration of interest} None.

\paragraph{Competing interests} The author declare that he have no competing interests.

\paragraph{Acknowledgment}
The author gratefully acknowledges the support from the French National Research Agency grant ANR-18-CE40-0021 (project HASCON). I am thankful to , Adrián González-Pérez, Shouhei Honda, Naotaka Kajino, Jun Kigami, Arup Kumar Pal, El Maati Ouhabaz, Jiayin Pan, Sang-Gyun Youn, and Melchior Wirth for their valuable insights during brief discussions. My appreciation extends to Li Gao and Bogdan Nica for their feedback and corrections. I extend my sincere appreciation to Bruno Iochum for the numerous discussions we had between the inception of the first version of this paper and the release of its preprint \cite{IoZ23}. Finally, I would like to thank the referees: the first for convincing me to replace the definition of spectral dimension in my paper \cite{Arh24} with the one defined in \eqref{Def-spectral-dimension} and for providing very helpful advices; and the second for encouraging me to expand the section on completely bounded Hardy-Littlewood-Sobolev theory and to write Proposition \ref{Prop-7-4}.



{\footnotesize

\vspace{0.2cm}

\noindent C\'edric Arhancet\\ 
\noindent 6 rue Didier Daurat, 81000 Albi, France\\
URL: \href{http://sites.google.com/site/cedricarhancet}{https://sites.google.com/site/cedricarhancet}\\
cedric.arhancet@protonmail.com\\

}
\end{document}